\documentclass[11pt]{article}
\usepackage{amsthm}
\makeatletter
\def\th@plain{%
  \thm@notefont{}
  \itshape 
}
\def\th@definition{%
  \thm@notefont{}
  \normalfont 
}
\makeatother
\usepackage{amsmath}
\usepackage{amsfonts}
\usepackage{latexsym}
\usepackage{amssymb}
\usepackage{graphicx}
\usepackage{multirow}
\usepackage{soul}
\usepackage[normalem]{ulem}
\usepackage{tabularx,ragged2e}
\newcolumntype{C}{>{\Centering\arraybackslash}X} 
\usepackage{caption}
\usepackage{subcaption}
\usepackage{float}
\usepackage{xcolor}
\usepackage{pdfpages}

\newtheorem{theorem}{Theorem}[section]
\newtheorem{assumption}[theorem]{Assumption}
\newtheorem{corollary}[theorem]{Corollary}
\newtheorem{definition}[theorem]{Definition}
\newtheorem{example}[theorem]{Example}
\newtheorem{lemma}[theorem]{Lemma}

\newtheorem{remark}[theorem]{Remark}
\newtheorem*{remarknonumber}{Remark}

\numberwithin{equation}{section}
\usepackage[a4paper]{geometry}
\newcommand{\beq}{\begin{equation}}
\newcommand{\eeq}{\end{equation}}
\newcommand{\cClocal}{\mathcal{C}_{\mathrm{local}}}
\newcommand{\hmaxell}{h_{\ell}}

\newcommand{\Ccont}{C_{\mathrm{cont}}}
\newcommand{\Ccontell}{C_{\mathrm{cont}, \ell}}

\newcommand{\tVhl}{{\mathcal{V}}^h_\ell}

\newcommand{\cI}{{\mathcal I}}

\newcommand{\cS}{{\mathcal S}}
\newcommand{\cT}{{\mathcal T}}

\newcommand{\cV}{{\mathcal V}}

\newcommand{\bx}{\boldsymbol{x}}

\newcommand{\bV}{\mathbf{V}}
\newcommand{\bR}{\mathbf{R}}

\newcommand{\bW}{\mathbf{W}}

\newcommand{\bu}{\mathbf{u}}

\newcommand{\bn}{\mathbf{n}}

\newcommand{\bff}{\mathbf{f}}

\newcommand{\supp}{\mathrm{supp}}

\newcommand{\bd}{\mathbf{d}}

\newcommand{\Rea}{\mathbb{R}}
\newcommand{\Com}{\mathbb{C}}

\newcommand{\bbC}{\mathbb{C}}
\newcommand{\abs}{\varepsilon}
\newcommand{\eps}{\varepsilon}

\newcommand{\ri}{{\rm i}}
\newcommand{\rd}{{\rm d}}

\usepackage{color}
\usepackage{hyperref}
\definecolor{myblue}{rgb}{0,0,0.6}
\definecolor{darkgreen}{rgb}{0,0.5,0}
\hypersetup{colorlinks=true,
linkcolor=myblue,citecolor=myblue,filecolor=myblue,urlcolor=myblue}

\definecolor{escol}{rgb}{0,0,0.8}
\definecolor{estcol}{rgb}{0,0.5,0}
\definecolor{esnewcol}{rgb}{0,0.5,0}

\newcommand{\igg}[1]{{\color{black}{#1}}}
\newcommand{\es}[1]{{\color{black}{#1}}}

\newcommand{\beqs}{\begin{equation*}}
\newcommand{\eeqs}{\end{equation*}}
\newcommand{\bit}{\begin{itemize}}
\newcommand{\eit}{\end{itemize}}
\newcommand{\ben}{\begin{enumerate}}
\newcommand{\een}{\end{enumerate}}
\newcommand{\bal}{\begin{align}}
\newcommand{\eal}{\end{align}}
\newcommand{\bals}{\begin{align*}}
\newcommand{\eals}{\end{align*}}
\newcommand{\bre}{\begin{remark}}
\newcommand{\ere}{\end{remark}}
\newcommand{\bpf}{\begin{proof}}
\newcommand{\epf}{\end{proof}}
\newcommand{\ble}{\begin{lemma}}
\newcommand{\ele}{\end{lemma}}
\newcommand{\bco}{\begin{corollary}}
\newcommand{\eco}{\end{corollary}}
\newcommand{\bex}{\begin{example}}
\newcommand{\eex}{\end{example}}
\newcommand{\bth}{\begin{theorem}}
\newcommand{\enth}{\end{theorem}}

\newcommand{\mymatrix}[1]{\mathsf{#1}}
\newcommand{\MA}{{\mymatrix{A}}}
\newcommand{\MX}{{\mymatrix{X}}}
\newcommand{\MB}{{\mymatrix{B}}}
\newcommand{\MD}{{\mymatrix{D}}}
\newcommand{\MR}{\widetilde{\mymatrix{R}}}
\newcommand{\ntMR}{\mymatrix{R}}

\newcommand{\MS}{{\mymatrix{S}}}
\newcommand{\MM}{{\mymatrix{M}}}
\newcommand{\MN}{{\mymatrix{N}}}
\newcommand{\SPD}{{\mathsf{SPD}}}

\newcommand{\domain}{\Omega}

\newcommand{\tfa}{\text{ for all }}

\newcommand{\ton}{\text{ on }}
\newcommand{\tas}{\text{ as }}
\newcommand{\tand}{\text{ and }}

\newcommand{\gu}{\nabla u}

\newcommand{\bnu}{\boldsymbol{\nu}}
\newcommand{\Ctr}{C_{\rm tr}}
\newcommand{\Cchi}{{C_{\chi}}}
\newcommand{\Cchiell}{C_{\chi, \ell}}
\newcommand{\CPi}{{C_{\Pi}}}
\newcommand{\Cint}{{C_{\rm int}}}
\newcommand{\Cintell}{C_{\rm int, \ell}}
\newcommand{\Csol}{C_{\rm sol}}
\newcommand{\Csolell}{C_{\rm sol, \ell}}
\newcommand*{\N}[1]{\left\|#1\right\|}

\newcommand{\sign}{{\rm sign}}
\newcommand{\bxi}{\boldsymbol{\xi}}
\newcommand{\Om}{\Omega}
\newcommand{\noi}{\noindent}
\newcommand{\tendi}{\rightarrow\infty}
\newcommand{\rI}{\mathrm{I}}

\newtheorem{experiment}[theorem]{Experiment}
\usepackage[outdir=./]{epstopdf}
\usepackage{verbatim}

\title{Domain decomposition preconditioners  for high-order discretisations of  the  heterogeneous Helmholtz  equation }
\author{Shihua Gong, Ivan G.~Graham and Euan A.~Spence, 
  \\[2ex]
  {\tt sg2328@bath.ac.uk, I.G.Graham@bath.ac.uk, eas25@bath.ac.uk}
\\[2ex]
Department of Mathematical Sciences, University of Bath, Bath BA2
7AY, UK.}
\date{\today}

\begin{document}
\maketitle

\begin{abstract}
 We consider  one-level additive Schwarz domain decomposition preconditioners for the   Helmholtz equation with  variable coefficients (modelling wave propagation in heterogeneous media), 
 subject to  boundary conditions that include wave scattering problems.
   Absorption is included as a parameter in the problem.  This  problem is  discretised using $H^1$-conforming nodal  finite elements of fixed  local degree $p$ on meshes with diameter $h = h(k)$,  chosen so that the error remains bounded with increasing  $k$.   The action of the one-level preconditioner
  consists of the parallel solution of problems on subdomains (which can be of 
  general geometry), each equipped with an  impedance  boundary condition. We prove rigorous estimates
  on the norm and field of values of the left- or right-preconditioned matrix that  show explicitly how the absorption, the heterogeneity in the coefficients and the dependence on the degree enter the estimates.
These estimates  prove rigorously that,   with enough absorption and for $k$ large enough, 
  GMRES is guaranteed to converge in a number of iterations that is independent of  $k,p,$ and the coefficients.    The theoretical threshold for $k$  to be large enough depends on $p$  and on the \emph{local} variation of coefficients in subdomains
  (and not globally).
  Extensive numerical experiments are given for both the absorptive and the propagative cases;
  in the latter case we investigate examples both when the coefficients are nontrapping and when they are trapping.
  These experiments (i) support our  theory  in terms of dependence on polynomial degree and  the coefficients; 
  (ii) support the sharpness of our  field of values estimates in terms of the level of absorption required. \\
  
\noi{\bf MSC2010 classification}: 65N22, 65N55, 65F08, 65F10, 35J05\\

\noi{\bf Keywords}: Preconditioning, Helmholtz equation, High Frequency, Variable Coefficients, High Order  Elements, Domain Decomposition

\end{abstract}

\begin{center}
\textbf{Dedicated to the memory of John W.~Barrett}
\end{center}

\section{Introduction}  \label{sec:Intro} 

\subsection{
Preconditioning    the Helmholtz equation}

Motivated by  the large range of applications,  there is currently great interest in designing and analysing preconditioners for finite element discretisations of the Helmholtz equation
\begin{align} \label{eq:HetHelm}
 \nabla \cdot (A\gu) + k^2 n\, u =-f,
  \end{align}  
 on a $d-$dimensional domain  ($d = 2,3$),  with $A$ and $n$ describing (possibly varying)
material properties,  and $k$ the (possibly large) angular frequency.
The discrete  systems 
are large (because  at least $\mathcal{O}(k^d)$ degrees of freedom  are  needed to resolve the oscillatory  solution), non self-adjoint
(because of the radiation condition present in scattering problems), and indefinite. They
are therefore notoriously difficult to solve and  many ``standard'' preconditioning techniques that are motivated by positive-definite problems are unusable in practice.

While there are a large  number of groups actively working on preconditioning Helmholtz problems (see the discussion in \S \ref{subsec:brief}),
enjoying  substantial algorithmic success backed up by physically-based insight,  there remains
limited rigorous analysis on this topic.

The main theoretical aim of our work is to analyse a preconditioner (based on a Helmholtz-related modification of a classical method for elliptic  problems)  that  is additive (and can be applied in massively parallel computing environments), and for which (under appropriate assumptions) the number of GMRES iterations is provably  independent of both $k$ and  the  degree of the underlying finite element  method; the proof requires that a certain amount of absorption is introduced into \eqref{eq:HetHelm}.

The action of our preconditioner involves parallel solution of subproblems
that, although   substantially smaller than the global problem,  can still be costly to solve by sparse direct methods.
The fast resolution of these subproblems is not a topic of this paper, but
multilevel approaches for these have been   discussed in \cite{GrSpZo:18},
although rigorous analysis of these  remains an open question.

We consider linear systems that arise from discretisation of \eqref{eq:HetHelm} by $H^1$-conforming nodal finite element methods of polynomial order
$p\geq 1$ on shape-regular meshes of diameter $h = h(k)$ chosen to avoid the ``pollution effect'', i.e.,  to  ensure that  the finite element error remains bounded as $k \rightarrow \infty$ (explained later in Remark \ref{rem:pollution}).    
We are  concerned with  one-level additive-Schwarz domain decomposition  preconditioners, where the local problems have impedance boundary conditions on the subdomain boundaries and the action of the preconditioner combines the local problems additively using a partition of unity; see \S\ref{sec:preconditioner}  for a precise definition. 
\igg{The  preconditioner we  analyse is the one-level part of a two level preconditioner (called  ``OBDD-H'')  first proposed (and studied empirically for  moderate $k$) by Kimn and Sarkis in \cite{KiSa:07}.}

The domain decomposition consists of a family of shape-regular overlapping subdomains of {characteristic length scale}
$H$ and overlap determined by parameter $\delta \leq H$. The overlap is assumed to be \emph{finite}, i.e.
any point in the domain $\Omega$ belongs to no more than a fixed number $\Lambda$ of subdomains for all $h,H$. 
 For our main theorems in \S \ref{sec:main_results}, our assumptions are
\beqs
  \text{(a)} \ \  kh \rightarrow 0  \quad \text{and} \quad \text{(b)} \ \  k\delta  \rightarrow \infty, \quad \text{as} \quad k \rightarrow \infty,   
  \eeqs
The condition (a) is naturally satisfied when we choose the mesh fine enough to avoid the   pollution effect, 
whereas condition (b) requires that the overlaps of the subdomains have to contain a (possibly slowly) increasing number of wavelengths as $k$ increases. 

In the following two sections we explain the background and the novel contributions of the current  paper. Since an
up-to-date survey of related work is contained in \cite{GrSpZo:18}, we restrict here to  a  brief literature
survey in \S \ref{subsec:brief}. 

\subsection{One-level additive-Schwarz  preconditioners: overview and previous work}\label{sec:1.2}

In the  paper \cite{GrSpZo:18},  the same preconditioner as considered here was analysed 
for  discretisations of the homogeneous  Helmholtz equation with absorption: 
\beq\label{eq:homo}
\Delta u +(k^ 2 + \ri \eps) u =-f,  
\eeq
where $0\leq \vert \eps \vert \leq k^2 $ is the absorption parameter and
the case  $\eps = 0$ is   referred to as the \emph{propagative} case.
       Alternatively, one can  perturb the wavenumber
       $k \mapsto k+ \ri \rho$, in which case, asymptotically as $k \rightarrow \infty$, $\eps \sim k^2$ corresponds to $\rho \sim k$ and $\eps \sim  k$ corresponds to $\rho \sim 1$.
        Absorptive (or ``lossy'') Helmholtz problems are of physical interest in their own right (see, e.g., the MEDIMAX problem
   \cite{MEDIMAX}, \cite[\S6.1]{BoDoGrSpTo:17c},
 and the references in   \cite{MeSaTo:19}). 
  Here our theory covers the absorptive case but we show experimentally that   our preconditioners can also be effective in the propagative case. 

In \cite{GrSpZo:18}, 
   a rigorous bound on the number of GMRES iterations was given 
under conditions relating the wavenumber $k$, the absorption $\eps$, and the size of the  overlap. 
To motivate the results of the present paper, we briefly summarise here the main theoretical results of \cite{GrSpZo:18} for  the PDE \eqref{eq:homo}. 

We write the finite element  discretisation of  \eqref{eq:hetero} as $\MA_\eps \bu_\eps =\bff$ and the one-level additive-Schwarz preconditioner as $\MB^{-1}_\eps$ (see \S\ref{sec:preconditioner} below for precise definitions).
One of the results in 
\cite{GrSpZo:18} showed, under the conditions outlined  above, and provided the subdomains are star-shaped with respect to a ball,  that 
there exist constants $C_1, \, C_2$ (which may depend on the polynomial degree $p$, but not    on other parameters) such that, 
for $k$ sufficently large,
\beq
\quad \Vert \MB^{-1}_\eps \MA_\eps\Vert_{\MD_k} \ \leq \ \ \Lambda \left(C_1 + C_2\min\left\{\frac{H}{\delta},  \frac{k}{\vert \eps\vert \delta}\right\}\right),  \label{eq:norm}
\eeq
where $\|\cdot\|_{\MD_k}$  denotes the Euclidean norm weighted with $\MD_k$, the stiffness matrix induced on the finite element space by the Helmholtz energy inner product
\begin{align} \label{eq:Helmen}
  (v, w)_{1,k} = (\nabla v, \nabla w) + k^2 (v,w),\end{align}
where $(\cdot, \cdot)$ is  the usual $L^2(\Omega)$ inner product

Since the right-hand side of  \eqref{eq:norm} is  bounded above independently of $k$,  one can  
obtain a bound on the number of iterations of GMRES 
applied to the preconditioned system $\MB^{-1}_{\eps}\MA_{\eps}$ by proving, in addition,
a lower bound for the distance of the field of values of $\MB^{-1}_{\eps}\MA_{\eps}$ from the origin and then using the so-called ``Elman estimate" for GMRES \cite{El:82, EiElSc:83}, \cite{BeGoTy:06}.
In \cite{GrSpZo:18} it was shown  that there is another constant $C_3>0 $ (which again may depend on $p$ but not on any other parameter) such that
\beq\label{eq:fov}
 \min_{\bV \not = \boldsymbol{0}} \frac{\vert \langle \bV, \MB_\eps^{-1}\MA_\eps  \bV\rangle_{\MD_k}\vert  }{\Vert \bV \Vert_{\MD_k} ^2 } \ \geq \  \left(\frac{1}{2 \Lambda}  -  C_3 \, \Lambda  \min\left\{\frac{H}{\delta},  \frac{k}{\vert \eps\vert \delta}\right\} 
 \right).  
\eeq
Therefore, by choosing  parameters so that $k/(\vert \eps\vert  \delta)$ is sufficiently small, 
 one can prove that the field of values of $ \MB_\eps^{-1}\MA_\eps$ (in the $\MD_k$ inner product) is bounded away from the origin. 
The Elman estimate then implies that GMRES applied to $\MB_\eps^{-1}\MA_\eps$ converges with the number of iterations independent of $k$, $\eps$, $h$, and $H$.

Although the argument in \cite{GrSpZo:18} is not rigorous for the pure Helmholtz case $\eps = 0$, the method  still empirically
provides a very effective preconditioner for the $\eps = 0$ case, provided  $H$ and $\delta$ do not decay too quickly as  $k$ increases -- see both the numerical experiments in \cite[\S4]{GrSpZo:18} and \S\ref{sec:numerical}. A heuristic explaining this is given in \cite[Appendix 1]{GrSpZo:18}, using the fact that  the problem with absorption is a good preconditioner for the pure Helmholtz problem (i.e.~$\MA_\eps^{-1}\approx \MA^{-1}$) when $\eps$ is not too big \cite{GaGrSp:15}.

Important features of the results in \cite{GrSpZo:18}   are that (a) they hold for
Lipschitz polyhedral domains and  cover sound-soft scattering problems, truncated 
using  first order absorbing boundary conditions (see Definition \ref{def:TEDP} below);
(b) the theory covers  finite element methods of any fixed order on shape regular meshes,
and  general  shape-regular subdomains (but the dependence of the estimates on the order is not explicit); (c) via a duality argument,  the theory covers both left- and right-preconditioning simultaneously; and 
(d) the  proof constitutes
a substantial extension of classical Schwarz theory to the non-self-adjoint case.

\subsection{The main results of this paper}
\label{subsec:novel}

\paragraph{Theory.}

On the theoretical side, the main achievements of this paper are 
\ben
\item extending the analysis in \cite{GrSpZo:18} to keep track of the dependence on the  polynomial degree $p$, and 
\item extending the analysis in \cite{GrSpZo:18} to the variable-coefficient Helmholtz problem
\beq\label{eq:hetero}
\nabla\cdot(A\gu) + (k^2+ \ri \eps) n u =-f
\eeq
(modelling wave propagation in heterogeneous media).
\een

Regarding 1:
the significance of analysing the $p$-dependence is that high-order methods suffer less from the pollution effect than low-order methods (see Remark \ref{rem:pollution} below), and are therefore preferable for high-frequency Helmholtz problems.
We show that for sufficiently large $k$,  
the constants $C_1, C_2, C_3$  
in \eqref{eq:norm} and \eqref{eq:fov} are
independent of the degree $p$. 
The Elman estimate then implies that, for sufficiently large $k$, GMRES applied to $\MB_\eps^{-1}\MA_\eps$ converges with the number of iterations independent of $p$.
The threshold  on  $k$ for achieving these results is, in principle, $p$-dependent, but in all our numerical experiments the iteration counts are independent of $p$ for all $k$.

  Regarding 2:  the significance of the analysis allowing variable coefficients is that a main motivation for the development of preconditioners for FEM discretisations of \eqref{eq:HetHelm} comes from practical applications involving waves travelling in  heterogeneous media (e.g.~in seismic imaging).
We show that the constants $C_2, C_3$  
in \eqref{eq:norm} and \eqref{eq:fov} have to be replaced by coefficient-dependent  quantities, but these only
depend on  the ``local''  variation of $A, n$
on each of the  subdomains $\Omega_\ell$; consequentially we obtain conditions for $k$-independent GMRES iterations, with these conditions depending on $A$ and $n$ only through this ``local" variation.

More precisely, our main theoretical results (stated for left-preconditioning, but holding also for right-preconditioning -- see Remark \ref{rem:rightpc} below) prove that there are absolute constants $c_2, c_3$ \emph{independent of all parameters} such that 
\beqs
\quad \Vert \MB^{-1}_\eps \MA_\eps\Vert_{\MD_k} \ \leq \ \ \Lambda \left(8 + c_2\cClocal(A,n) \,
\min\left\{\frac{H}{\delta},  \frac{k}{\vert \eps\vert \delta}\right\} 
\right), 
\eeqs
where $\MD_k$ is the same stiffness matrix introduced above 
and,
\beq
\min_{\bV \not = \boldsymbol{0}} \frac{\vert \langle \bV, \MB_\eps^{-1}\MA_\eps  \bV\rangle_{\MD_k}\vert  }{\Vert \bV \Vert_{\MD_k} ^2 } \ \geq \ \left(\frac{1}{2 \Lambda}  -  c_3 \, \Lambda \cClocal(A,n)
\min\left\{\frac{H}{\delta},  \frac{k}{\vert \eps\vert \delta}\right\}\right),  
\label{eq:fov1}
\eeq
where the ``local contrast'' $\cClocal(A,n)$ depends only on the variation of $A,n$ in each subdomain
  and not globally (see \eqref{eq:contrast}). 
Hence we can still guarantee that GMRES will converge in a $k$-independent number of iterations 
provided we now also make $k/(\vert \eps \vert \delta) $ sufficiently small relative to  the local contrast. This  is true  if, for example,  the overlapping subdomains are fixed and $\eps$ is a sufficiently large constant multiple of
$k$ (equivalently the wave number is $k + \ri \rho$ with $\rho$  a sufficiently large constant).    


Another important feature of our  theory  is that we prove convergence results for  implementations  of GMRES \emph{both} in the weighted inner product  $\langle \cdot, \cdot \rangle _{\MD_k}$ \emph{and} in the standard Euclidean inner product;
this is in contrast to \cite{GrSpVa:17}, \cite{BoDoGrSpTo:17c}, and \cite{GrSpZo:18}, which only prove convergence results
in the weighted inner product.
Corollary \ref{cor:GMRES} below shows that  standard GMRES requires at most 
$(\log k)/(pC)$ more iterations than weighted  GMRES, where $C$ is a constant independent of all parameters. 

Our theory relies on error estimates for the finite-element method that are explicit in both $k$ and the  variable coefficients. Obtaining such error estimates is 
the subject of current research, see \cite{Ch:16},\cite{BaChGo:17},\cite{SaTo:18},\cite{ChNi:19},\cite{GaMo:19},\cite{GrSa:20},\cite{GaSpWu:18},\cite{LaSpWu:19}, and we make contact with these results in Remarks \ref{rem:pollution}, \ref{rem:meshthresh1}, and \ref{rem:meshthresh2} below.

\paragraph{Computation.}
In our numerical experiments, we investigate the dependence of the preconditioner on both the polynomial degree $p$ and the coefficients $A$ and $n$ for a variety of 2D variable coefficient  Helmholtz problems (both with and without absorption). In agreement with the theoretical results,  we find that the number of GMRES iterations for the preconditioned system  is independent of $p$ for $p=1,2,3,4$, and depends only on the local variation of $A$ and $n$.

We also investigate the limits of the theory by computing the field of values of the preconditioned system $\MB^{-1}_\eps\MA_\eps$ for various choices of absorption $\eps$. We give results showing that  the estimates \eqref{eq:fov}, \eqref{eq:fov1} appear to be sharp; i.e.~we give an example where $k/\vert \eps\vert  \delta$ does not approach $0$ and for which  the origin lies inside the field of values of $\MB^{-1}_\eps\MA_\eps$.
This example indicates that the dependence of the bound \eqref{eq:fov1} on $k/\vert \eps\vert  \delta$ cannot be removed. However for this example, GMRES often performs  well, once again highlighting the important point that
having both the norm bounded above and the distance of the field of values bounded away from the origin are sufficient but certainly not necessary for good convergence of GMRES.

\bre[How trapping affects our results]\label{rem:trapping}
The behaviour of the solution of  \eqref{eq:hetero} when $\eps=0$ depends crucially on whether or not the coefficients $A$ and $n$ and any impenetrable obstacles in the domain are such that the problem is \emph{trapping}. A  precise definition of the converse of trapping (known as \emph{nontrapping}) is given in Definition \ref{def:nontrapping} below, but roughly speaking a problem is trapping if there exist arbitrarily-long rays in the domain (corresponding to high-frequency waves being ``trapped").

Under the strongest form of trapping, the Helmholtz solution operator grows exponentially through a discrete sequence of frequencies $0<k_1<k_2<\ldots<k_j \tendi$; see \cite[\S2.5]{BeChGrLaLi:11} for this behaviour caused by an obstacle, \cite{Ra:71}, \cite[Theorem 7.7]{GrPeSp:19} for this behaviour caused by a continuous coefficient, and \cite{PoVo:99}, \cite[Section 6]{MoSp:17} for this behaviour caused by a discontinuous coefficient.

The vast majority of numerical analysis of the Helmholtz equation when $\eps=0$ therefore takes place under  a nontrapping assumption (see the references in \cite[\S1.3.1]{LaSpWu:19}) and we make a similar assumption in this paper for our theory when $\eps=0$; see Assumption \ref{ass:nontrapping} below.

However, it is known empirically that this exponential growth in the solution operator is ``rare", in the sense that the solution operator under trapping is well behaved for ``most" frequencies, and a first rigorous result establishing this was recently obtained in \cite[Theorem 1.1]{LaSpWu:19}.
We see the rareness of ``bad behaviour" under trapping in the numerical experiments in \S\ref{sec:numerical}; indeed, several of the problems with variable $A$ and $n$ considered in this section are provably trapping, and yet we never see the extreme ill-conditioning associated with trapping for any of the frequencies at which we compute. 
\ere

\subsection{Brief survey of other work on Helmholtz preconditioners}
\label{subsec:brief}

We mention here some recent papers on iterative solution of the Helmholtz equation.
  A complete and up-to-date survey with more detail is available in \cite{GrSpZo:18}. 

  Two important classes of algorithms for high frequency Helmholtz problems are those based on multigrid with ``shifted Laplace'' preconditioning and those based on inexact factorizations via ``sweeping-type methods''.
  Foundational papers for these methods
  are   \cite{ErVuOo:04,ErOoVu:06} and \cite{EnYi:11c,EnYi:11b} respectively.
  Recent work on shifted Laplace has focused on improving robustness using deflation e.g.
  \cite{ShLaRaNaVu:16,ErRaNa:17,LaTaVu:17}.
  Following earlier work on parallel sweeping methods  (e.g.~\cite{PoEnLiYi:13}), 
efficient  parallel implementations of sweeping-type methods 
via off/online strategies are a focus of the polarized-trace algorithms  \cite{ZeDe:16,ZeDe:17,TaZeHeDe:19},

   While much of this work is empirical, often based on considerable physical insight, there has been a
   parallel development of theoretical underpinning of these algorithms. A theoretical basis for sweeping algorithms on the continuous level is given in \cite{ChXi:13}. A modern survey linking sweeping-type algorithms with optimized Schwarz methods is given in  \cite{GaZh:19}. The current authors have been involved in the development of a theory for additive Schwarz
   preconditioners for  Helmholtz problems (e.g., \cite{GrSpVa:17, GrSpZo:18}). So far this theory  has  not examined the dependence of the theory on heterogeneity or finite-element  degree. These topics are the focus of the present paper.     
   
 In the final stages of writing this paper, we learned of the paper \cite{BoClNaTo:20}.
     Both \cite{BoClNaTo:20} and the present paper start from the analysis of the homogeneous Helmholtz equation
     in \cite{GrSpZo:18}  and extend it in various ways.
     Whilst the present paper focuses on extensions to  the heterogeneous Helmholtz equation, 
     \cite{BoClNaTo:20} instead obtains general conditions under which results analogous to those in \cite{GrSpZo:18}
     hold, and then applies these to obtain results about preconditioners for  a heterogeneous
     reaction-convection-diffusion equation. Both papers share the goal of determining the
     constants in the field-of-value estimates in terms of more fundamental constants. 
     In our case these are $k,h,H,\delta, \eps, p $ and $\cClocal(A,n)$ -- see \cite[Section 4.2]{BoClNaTo:20} for more details of their analysis.

\subsection{Organisation of the paper}

In \S \ref{sec:weak} we define precisely the variable-coefficient Helmholtz boundary-value problem, its finite-element discretisation, and the one-level additive Schwarz DD preconditioner. In \S \ref{sec:solution_operators} we study the solution operator of the global and local problems, at both continuous and discrete levels. In \S \ref{sec:dd_results}, we give some preliminary estimates on the interpolations and the local projections. The main theoretical results are presented in \S \ref{sec:main_results} and the numerical results are presented in \S \ref{sec:numerical}.

  \section{The PDE problem, its discretisation    and the preconditioner}
\label{sec:weak}

\subsection{The PDE problem and its finite-element discretisation}

We consider the PDE problem \eqref{eq:HetHelm}, posed in a certain  domain $\domain$ that is a subset of the domain exterior to a bounded
obstacle $\domain_{-}$.

  \begin{assumption}[Assumptions on $ \domain_-, \domain$ and on the coefficients $A,n$]
  \label{ass:1}

\

(i) 
  $\domain_-$ is assumed to be a   bounded Lipschitz polygon ($d=2$) or polyhedron ($d=3$),
  with connected  open complement $\domain_+:= \Rea^d\setminus \overline{\domain_-}$.
  Also,  $\overline{\domain_-}\subset\subset \domain^\prime$, for some  
  connected Lipschitz polygon/polyhedron  $\domain^\prime$.   Then  $\domain:=\domain^\prime\setminus\overline{\domain_-}$ with boundary  $\partial \domain= \Gamma_D \cup \Gamma_I$, where 
 $\Gamma_D:= \partial \domain_-$ (the `scattering' boundary), and $\Gamma_I :=\partial \domain^\prime$ (the `far field' boundary). Note that $\Gamma_D\cap \Gamma_I = \emptyset$.

(ii) $A \in L^\infty(\domain, \SPD)$ (where $\SPD$ denotes the set of $d\times d$ real, symmetric, positive-definite matrices)
and
there exist $0<A_{\min}\leq A_{\max}<\infty$ such that, in the sense of quadratic forms,
\beq\label{eq:Alimits}
 A_{\min}\leq A(\bx)\leq A_{\max}\quad\text{ for almost every }\bx \in \domain.
\eeq

(iii) $n\in L^\infty(\domain,\Rea)$,  and there exist $0<n_{\min}\leq n_{\max}<\infty$ such that
\beq\label{eq:nlimits}
0<n_{\min} \leq n(\bx)\leq n_{\max}<\infty\,\, \text{ for almost every } \bx \in \domain.
\eeq
\end{assumption}
The PDE is posed  in the space 
\beq\label{eq:spaceEDP}
H_{0,\Gamma_D}^1(\Omega):= \big\{ v\in H^1(\Omega) :  v=0 \ton \Gamma_D\big\},
\eeq
where we do not use any specific notation for the trace operator on $\Gamma_D$ (or on any other boundary).
The complex $L^2$ inner products on the domain $\Omega$ and the boundary $\Gamma_I$ are denoted by $(u, v):=\int_{\Omega} u\bar{v}~\rd x$ and $\langle u, v \rangle: = \int_{\Gamma_I} u\bar{v}~\rd s$, respectively. The $L^2$ norm on $\Omega$ is denoted by $\Vert \cdot \Vert$. Subscripts will be used to denote norms and inner products over other domains or surfaces, e.g., $(u, v)_{\tilde{\Omega}}$ and $\langle u, v \rangle_{\tilde{\Gamma}}$.

\begin{definition}[Truncated Exterior Dirichlet Problem (TEDP)]\label{def:TEDP}
Given $\Omega$, $A$, and $n$ satisfying Assumption \ref{ass:1}, $f\in L^2(\domain)$, $g_I\in L^2(\Gamma_I)$, $k>0$, and $\eta>0$,  
we say $u\in H_{0,\Gamma_D}^1(\Omega)$ satisfies the truncated exterior Dirichlet problem if 
\begin{equation}\label{eq:contvp}
a_{\eps}(u,v) = F(v), \quad \tfa v\in H^1_{0,\Gamma_D}(\Omega),
\end{equation}
where 
\begin{equation}\label{eq:a_eps}
a_{\eps}(u,v): = (A \nabla u, \nabla {v}) - (k^2+\ri \eps) (n u,{v}) - \ri \eta \langle \sqrt{n} u,{v}\rangle
 \quad {\text{and}} \quad F(v): = (f,{v}) + \langle g_I , v\rangle.
\end{equation}
\end{definition}

That is, $u$ satisfies the PDE \eqref{eq:hetero}, a zero Dirichlet boundary condition on $\Gamma_D$, and the impedance boundary condition 
\beq\label{eq:impedance}
\partial_{\bn, A} u - \ri \eta  \sqrt{n} u = g_I \ton \Gamma_I,
\eeq
where  $\partial_{\bn, A}u$ is the \emph{conormal derivative} of $u$, which is such that, if $A$ is Lipschitz and $u\in H^2$, then $\partial_{\bn, A}u= (A \nabla u)\cdot \bn$,  with   $\bn$ denoting  the
  unit normal on $\Gamma_I$, pointing outward from $\Omega$
 (see, e.g., \cite[Lemma 4.3]{Mc:00}).

\begin{assumption}[The choice of $\eta$]\label{ass:eta}
In the following, we set \emph{either} $\eta =\mathrm{sign}(\eps) k$ \emph{or} $\eta =\sqrt{k^2+\ri \eps}$ (with both choices understood as $\eta=k$ when $\eps=0$), where the square root is defined with the branch cut on the positive real axis. 
\end{assumption}

We record for later the facts that, under either of the choices of $\eta$ in Assumption \ref{ass:eta},
\begin{equation}\label{eq:eta}
\Im(\eta) \ge 0, \quad \mathrm{sign}(\eps)\Re(\eta)>0.
\end{equation}

\ble[Well-posedness of the TEDP]\label{lem:wp}

With $\eta$ chosen as in Assumption \ref{ass:eta},
the solution of the TEDP of Definition \ref{def:TEDP} exists, is unique, and depends continuously on the data in the following situations.

(i) $|\eps|\neq 0$, $d=2,3$,

(ii) $\eps=0$, $d=2$, 

(iv) $\eps=0$, $d=3$, 
and $A$ is piecewise Lipschitz.
\ele

\bpf[References for proof]
When $|\eps|\neq 0$, the result follows from the Lax--Milgram theorem, since $a_\eps$ is continuous and coercive on $H^1_{0,\Gamma_D}(\Omega)$ by Lemmas \ref{lem:cont} and \ref{lem:coer} below.

For the case $\eps=0$, first observe that the first inequality in \eqref{eq:eta} implies that $a_0$ satisfies the G\aa rding inequality
\beqs
\Re a_0(v,v) \geq A_{\min} \N{\nabla v}^2 - k^2 n_{\max} \N{v}^2,
\eeqs
and therefore well-posedness follows from uniqueness by Fredholm theory (see, e.g., \cite[Theorem 2.33]{Mc:00}). Uniqueness follows from the unique continuation principle (UCP), with the condition that $A$ is piecewise Lipschitz when $d=3$ ensuring that this principle holds. In the case when $A$ is scalar-valued (and, when $d=3$, Lipschitz), these UCP results are recapped and then applied to the TEDP for the Helmholtz equation in  \cite[\S2]{GrSa:20}. When $A$ is matrix-valued, the relevant UCP results are summarised in \cite[\S1]{GrPeSp:19}. The result that, if a UCP holds for Lipschitz $A$, then a UCP holds for piecewise Lipschitz $A$ is proved in \cite{BaCaTs:12}; see the discussion in \cite[\S2.4]{GrPeSp:19}.
\epf

\bre[The behaviour of $A$ and $n$ in a neighbourhood of $\Gamma_I$.]\label{rem:Sommerfeld}
It would be natural to add the extra conditions to Assumption \ref{ass:1} that $\supp(I-A)$ is compact in $\domain^\prime$ and $\supp(1-n)$ is compact in $\domain^\prime$,  implying that $A=I$ and $n=1$ in a neighbourhood of $\Gamma_I$. Indeed, 
the rationale for imposing the impedance boundary condition \eqref{eq:impedance} is that (with $A=I$, $n=1$, $\eta=k$ and $g_I=0$) it is the simplest-possible approximation of the Sommerfeld radiation condition 
\beqs
\frac{\partial u}{\partial r}(\bx) - \ri k u(\bx) = o \left( \frac{1}{r^{(d-1)/2}}\right)
\tas r:= |\bx|\rightarrow \infty, \text{ uniformly in } \widehat{\bx}:= \bx/r,
\eeqs
which appears in exterior Helmholtz problems. When solving a scattering problem by truncating the domain, one would naturally place the truncation boundary $\Gamma_I$ so that it encloses the entire scatterer $\supp(I-A)\cup \supp (1-n)\cup \overline{\Omega_-}$, ensuring  that $A=I$ and $n=1$ in a neighbourhood of $\Gamma_I$. 
All our numerical experiments in \S\ref{sec:numerical} are for this  situation,
but we allow more general $A$ and $n$ in our analysis. 
\es{This is because the local sesquilinear form $a_{\eps,\ell}$, defined by \eqref{eq:local-vp} below, is very similar to $a_\eps$, except that the region of integration is the $\ell$th subdomain of the domain decomposition (denoted $\Omega_\ell$) instead of $\Omega$.
In \S\ref{sec:solution_operators}, we prove results for both $a_\eps$ and $a_{\eps,\ell}$ simultaneously, and in doing so  we do not make any assumptions about $A$ and $n$ on the boundary of the domain in Assumption \ref{ass:1}
}
\igg{This permits the domain decomposition to  have subdomains that pass through regions with  variable coefficient, thus  maximising  generality (e.g., to allow the application of  automatic mesh partitioners).
}
\ere

\begin{remark} [Interior impedance problem] \label{rem:IIP}
  Note that the case  $\domain_- = \emptyset$ (i.e., there is no impentrable scatterer) is included in this set-up, and in this case the space \eqref{eq:spaceEDP} is simply $H^1(\Omega)$. The  resulting PDE problem is often called the
  Interior Impedance Problem.
\end{remark}

\paragraph{Finite-element discretisation.} We approximate \eqref{eq:contvp}
in a  conforming nodal finite element space $\mathcal{V}^h \subset H^1_{0,\Gamma_D}$ that consists of continuous piecewise  polynomials of total degree no more than $p$ on a shape-regular mesh $\mathcal{T}^h$ with mesh size $h$. This
yields the linear system 
\begin{equation}\label{eq:discrete}
\MA_\eps \bu = \bff, \quad \text{where} \quad \MA_\eps = \MS - (k^2+\ri \eps) \MM - \ri \eta \MN,
\end{equation}
and 
\beq\label{eq:defmatrices} 
  (\MS)_{i,j} = (A \nabla \phi_j,\nabla \phi_i), \quad (\MM)_{i,j} = ( n   \phi_j,  \phi_i),
  \quad (\MN)_{i,j} =\langle  \sqrt{n}\phi_j , \phi_i  \rangle, \quad   \text{and} \quad (\bff)_i = F(\phi_i).
\eeq
Here $\{\phi_j:j\in \mathcal{I}^h\}$ is the nodal basis of the finite element space $\mathcal{V}^h$ and $\mathcal{I}^h$ is a suitable index set for  the nodes $\{\bx_j:j\in \mathcal{I}^h\}$. 
We  use  the $h$-version of the finite-element method where, in the context of solving the high-frequency Helmholtz equation, $p$ is fixed and then $h$ is chosen as a function of $k$ and $p$ to maintain accuracy as $k$ increases.

\bre[Under what conditions on $h$, $p$, and $k$ does the finite-element solution exist?]
\label{rem:pollution}
When $|\eps|>0$, $a_\eps$ is continuous and coercive (see Lemmas \ref{lem:cont} and \ref{lem:coer} below) and so, by the Lax--Milgram theorem and C\'ea's lemma, the finite-element solution exists and is unique for any $\cV^h$.
The finite element error is then bounded in terms of the best approximation error, but with a constant that in general depends on $k$ and $\eps$.  

When $\eps=0$, $A=I$, and $n=1$ (i.e., ~for the \emph{constant-coefficient} Helmholtz equation), the question of how $h$ and $p$ must depend on $k$ for the finite-element solution to exist has been studied since the work of Ihlenburg and Babu\v{s}ka in the 90's \cite{IhBa:95, IhBa:97}.
For the $h$-version of the FEM (where accuracy is increased by decreasing $h$ with $p$ fixed), provided that the problem is nontrapping (i.e.~$\Csol$ defined in Definition \ref{def:Csol} below is bounded independently of $k$), then the Galerkin method is quasioptimal when  $k^{p+1}h^p$ is sufficiently small \cite{Me:95, Sa:06, MeSa:10, MeSa:11}.  However  the Galerkin solution exists and is bounded by the data  under the weaker condition that $k^{2p+1}h^{2p}$ is sufficiently small   \cite{Wu:14, ZhWu:13, DuWu:15}.
For the $hp$-version (where accuracy is increased by decreasing $h$ and increasing $p$), the Galerkin method is quasioptimal if $hk/p$ is sufficiently small and $p$ grows logarithmically with $k$ by \cite{MeSa:10, MeSa:11}, with this property holding for scattering problems even under the strongest-possible trapping by \cite[Corollary 1.3]{LaSpWu:19}.

Obtaining the analogues of these results when $A\neq I$ and $n\neq 1$ (i.e.~for \emph{variable-coefficient} Helmholtz problems) is the subject of current research, with the property ``$k^{p+1}h^p$ sufficiently small for quasioptimality" proved for the TEDP in \cite[Proposition 2.5, Theorem 2.15]{ChNi:19} (see also \cite[\S4]{GrSa:20} and \cite[\S6]{GaSpWu:18} for the case $p=1$), 
the property ``$k^{2p+1}h^{2p}$ sufficiently small for error bounded by the data" proved in \cite[Theorem 2.35]{Pe:19}, and the property 
``$k^{3}h^{2}$ sufficiently small for relative error bounded when $p=1$" proved for scattering problems in \cite{LaSpWu:19a} (and  all these  results  assume  that the problem is nontrapping).
We highlight that, when $p>1$, these results require additional smoothness conditions on $\Gamma_D$, $\Gamma_I$, $A$, and $n$
in addition to  Assumption \ref{ass:1}; see \cite[\S2.1]{ChNi:19}, \cite[Assumption 2.31]{Pe:19}.
\ere

\subsection{The  preconditioner}
\label{sec:preconditioner}
To precondition the linear system \eqref{eq:discrete}, we   use a variant of the simple  one-level additive 
Schwarz method,   based on a set of Lipschitz polyhedral    
subdomains   $\{\Omega_\ell\}_{\ell = 1}^N$,  
forming  an overlapping cover of $\Omega$.
We assume that each $\overline{\Omega_\ell}$  
is non-empty and is  a union of  elements of the mesh $\cT^h$. We also  assume that if $\partial \Omega_\ell 
  \cap \Gamma_D  \not = \emptyset$,  then it has positive surface measure. Because $\Omega_\ell$ consists
  of a union of fine-grid elements,  $\partial \Omega_\ell 
  \cap \Gamma_D $ 
  then contains at least one face of a fine-grid element.

Recall that 
a domain is said to have  \emph{characteristic length scale} $L$ if its diameter $\sim L$, its surface area $\sim L^{d-1}$, and its volume $\sim L^d$. We assume that each $\Omega_\ell$ has characteristic length scale $H_\ell$,  and we set $H = \max_\ell H_\ell$.
 In our analysis  we allow $H$ to depend on $k$ including the possibility that $H$ could approach $0$ as $k\rightarrow \infty$.

The key component of the  preconditioner for \eqref{eq:discrete} is the solution of discrete ``local''  
impedance boundary-value problems:
\beqs
\nabla \cdot (A \nabla u)  + (k^2+ \ri \abs)nu = -f  \quad \text{on} \quad \Omega_\ell ,   
\eeqs
with
\beqs
   \partial_{\bnu, A}u - \ri \eta \sqrt{n}  u = 0  \,\,\ton \,\,\Gamma_{I,\ell}:=\partial \Omega_\ell 
\backslash \Gamma_D \quad\text{and} \quad  u = 0  \,\,\ton\,\, \Gamma_{D,\ell}:=\partial \Omega_\ell
\cap \Gamma_D;
\eeqs 
where $\bnu$ is the outward-pointing unit normal vector to $\Omega_\ell$ and $\partial_{\bnu, A}u$ is the \emph{conormal derivative} of $u$.
Observe that the geometric set-up in Assumption \ref{ass:1} implies that $ \partial \Omega_\ell 
\backslash \Gamma_D$ has positive measure, and so each of these local problems is well-posed.
The local impedance sesquilinear form on $\Omega_\ell$ is, for $u,v\in H^1(\Omega_\ell)$,
\begin{equation}\label{eq:local-vp}
a_{\eps,\ell}(u,v) := (A \nabla u, \nabla {v})_{\Omega_\ell} - (k^2+\ri \eps) (n u,{v})_{\Omega_\ell} - \ri \eta \langle \sqrt{n} u,{v}\rangle_{\Gamma_{I,\ell}}.
\end{equation}
We denote by $\MA_{\abs,\ell}$ the matrix obtained by 
approximating \eqref{eq:local-vp} in the local finite element space 
\beq\label{eq:localV}
 \cV_{\ell}^h: =\{v_{h}|_{\Omega_\ell}~:~v_{h}\in \mathcal{V}^h\};
\eeq
this matrix is a local analogue of the matrix  $\MA_\abs$ in \eqref{eq:discrete}. Recalling  that functions in $\cV^h$ vanish on the Dirichlet boundary $\Gamma_D$,    
we observe that functions in $\cV^h_\ell$ also vanish on $\partial \Omega_\ell\cap \Gamma_D$ (which contains at least one face of one fine-grid element if it is non-empty,  but are otherwise
unconstrained).
To connect these local problems, we use a partition of unity $\{\chi_\ell\}_{\ell = 1}^N$ with the properties that, for each $\ell$,
  \beq
  \supp \chi_\ell \subset \overline{\Omega_\ell}, \quad
  0 \leq  \chi_\ell(\bx) \leq 1\,\,  \text{when } \bx \in \overline{\Omega_\ell},\quad\tand\quad
  \sum_\ell  \chi_\ell(\bx) = 1 \,\tfa \bx \in \overline{\Omega}, 
  \label{POUstar}
  \eeq
where we define $\supp \chi_\ell = \{ \bx \in \overline{\Omega} : \chi_\ell(\bx) \not= 0\}$. 
Additional properties of these functions are needed later -- see \eqref{eq:derivpou}, \eqref{eq:Cchidef}.
Note that each  $v_{h} \in \cV^h$  
is uniquely determined by its values $\{V_{q} : = v_{h} (\bx_{q}), \ q \in \cI^h\}$ on all the nodes.  Nodes on the subdomain  $\overline{\Omega_\ell}$ 
are denoted by $ \{ \bx_{q}: q \in \cI^h(\overline{\Omega_\ell})\} $.    Using this notation, we define a restriction matrix  $\MR_\ell$ 
that uses $\chi_\ell$ to map
a nodal vector defined on $\overline{\Omega}$ to a nodal vector on $\overline{\Omega_\ell}$:     
\beq 
\label{eq:explicit}  
(\MR_\ell \bV)_{q} \ = \ \chi_\ell(\bx_{q}) V_{q} , \quad q \in \cI^h(\overline{\Omega_\ell}) . 
\eeq
The preconditioner for $\MA_\abs$ that we analyse in this paper is then simply: 
\begin{equation} \label{eq:ASpc} 
\MB_\abs^{-1}  :=  \sum_{\ell = 1}^N 
\MR_\ell^\top (\MA_{\abs,\ell})^{-1} \MR_\ell  \ ,
\end{equation}
where  $\MR_\ell^\top $ is the   transpose of $\MR_\ell$. 
The action of $\MB_{\eps}^{-1}$ therefore consists of $N$ parallel ``local impedance solves''  added up with the aid of  
appropriate restrictions/prolongations.   

\es{We now describe}
\igg{a related preconditioner,
  \es{for which we give no theory, but which we consider in our numerical experiments (alongside \eqref{eq:ASpc}).}
\es{This second preconditioner} involves the simpler  ``restriction by  chopping''  operator \igg{${\ntMR}_\ell$} that maps a nodal vector on $\overline{\Omega}$ to a nodal vector on $\overline{\Omega}_\ell$,  according to the rule:
\beq 
\igg{({\ntMR}_\ell \bV)_{q} \ = \  V_{q} , \quad q \in \cI^h(\overline{\Omega_\ell}) .} 
\eeq
Then, we  replace  the occurence of $\MR_\ell$ on the  right-hand side of   
\eqref{eq:ASpc} by  ${\ntMR}_\ell$,  to obtain 
\begin{equation} \label{eq:ASpc1} 
{\widehat{\MB}}_\abs^{-1}  :=  \sum_{\ell = 1}^N 
\MR_\ell^\top (\MA_{\abs,\ell})^{-1} {\ntMR}_\ell  \ .
\end{equation}
Preconditioners \eqref{eq:ASpc}  \es{and} \eqref{eq:ASpc1} were  both originally introduced  in \cite{KiSa:07}, where they comprise the one-level
components of the preconditioners \igg{there named, respectively,}  {OBDD-H} and {WRAS-H}.

As the name suggests, \eqref{eq:ASpc1} also bears some resemblance to the Optimized Restricted Additive Schwarz (ORAS) method, discussed, for example in \cite{st2007optimized}, \cite{DoJoNa:15}. \igg{In the simplest purely algebraic case (as described in \cite{st2007optimized}) one  can  start  with a non-overlapping  decomposition of the unknowns into subsets,   which is then  extended to an overlapping one.  Then, ORAS can be written in the form \eqref{eq:ASpc1}, with 
  $\ntMR_\ell$ denoting restriction by chopping onto the $\ell$th overlapping subdomain, and  $\MR_\ell^\top$ denoting  extension by zero with respect to the non-overlapping cover.

  The  term `optimized' usually refers to the case when the parameter(s) in a Robin boundary condition 
  are  chosen to  enhance the convergence rate of an iterative solver.
In the   Helmholtz case here,  the classical impedance condition is used on subdomain boundaries, without optimization.  Nevertheless  the name ORAS is sometimes used for \eqref{eq:ASpc1}, and we use that name here. Also we  call \eqref{eq:ASpc} the SORAS preconditioner for Helmholtz, since the restriction $\MR_\ell$ and  prolongation $\MR_\ell^\top$ are here   applied in a  `symmetric' way. This terminology is used elsewhere in the literature\es{; see,} e.g., \cite{BoDoGrSpTo:17a, BoClNaTo:20}. }

\igg{Finally, we mention that a one level RAS-type  method with Impedance boundary conditions for Helmholtz was discussed in  \cite{GrSpVa:17,GrSpVa:17a}, where it was  called IMPHRAS1.   }
}

 \section{The solution operators for the continuous and discrete global and local problems}
 \label{sec:solution_operators}
 
A key step in the analysis of \eqref{eq:ASpc} is to obtain estimates for the local inverse matrices $\MA_{\eps,\ell}^{-1}$. We obtain these in \S \ref{subsec:disclocal} below. The preceding subsections are devoted the analogous results for the corresponding continuous problems.    
The following assumption is sufficient for many of our technical estimates. The results in \S \ref{sec:main_results} require the addition of stronger assumption (Assumption \ref{ass:3} below).
\begin{assumption}\label{ass:2}
\beq\label{eq:eps}
k \geq 1, \quad |\eps|\leq k^2 \quad\tand\quad hk\leq 1.
\eeq
\end{assumption}

We perform the analysis of this paper in the following weighted $H^1$ norm (also called the ``Helmholtz energy norm") defined by 
 \begin{align} \label{eq:HelmE} 
   \Vert v \Vert_{1,k} := (v,v)_{1,k}^{1/2},
 \end{align}
 where $(v,w)_{1,k}$ is defined in \eqref{eq:Helmen}. 
 When $\widetilde{\Omega}$ is any subdomain of $\Omega$ we write $(\cdot, \cdot)_{1,k,\widetilde{\Omega}}$ and $\Vert \cdot \Vert_{1,k,\widetilde{\Omega}}$ for the corresponding inner product and norm on $\widetilde{\Omega}$.
The inner product $(\cdot,\cdot)_{1,k}$ and norm $\|\cdot\|_{1,k}$ in \eqref{eq:HelmE} induce an inner-product and norm on finite-element functions. 
Suppose  $v_h, w_h \in \cV^h$ are  represented by the vectors $\bV,\bW$ with respect to the nodal basis 
$\{\phi_j:j\in \mathcal{I}^h\}$. Then 
\beqs
(v_h, w_h)_{1,k} = \langle \MD_k\bV, \bW\rangle_2= \bW^*\MD_k \bV
\eeqs
where $\langle\cdot,\cdot\rangle_2$ denotes the standard Euclidean inner product  on $\Com^n$,
and where $\MD_k= \MS + k^2 \MM$, with  $\MS$ and $\MM$ as  defined in \eqref{eq:defmatrices}
with $A=I$ and $n=1$.  
We therefore define on $\Com^n$
\beqs
  \langle \bV ,\bW \rangle_{\MD_k}:=  \bW^*\MD_k \bV, \quad \Vert \bV \Vert_{\MD_k} :=  \langle \bV ,\bV \rangle_{\MD_k}^{1/2}.
\eeqs
With  $A_{\max}, A_{\min}, n_{\max},$ and $n_{\min}$ as  defined in Assumption \ref{ass:1}, we  define $A_{\max,\ell}, A_{\min, \ell}, n_{\max, \ell},$ and $n_{\min,\ell}$ in an analogous way, i.e., for almost every $\bx \in \domain_\ell$, 
\beq\label{eq:Alimitsell}
 A_{\min,\ell}\leq A(\bx)\leq A_{\max,\ell}\quad\text{and} \quad 
0<n_{\min, \ell} \leq n(\bx)\leq n_{\max, \ell}<\infty,
\eeq
with the first inequality holding in the sense of quadratic forms.  

 \subsection{Continuity and coercivity of $a_\eps$ and $a_{\eps,\ell}$}
 
 Since 
   $a_\eps$  and $a_{\eps,\ell}$ given by \eqref{eq:a_eps},\eqref{eq:local-vp} 
   differ only in their
   domains of integration,  
 we state and prove the results for $a_\eps$, with the results for $a_{\eps,\ell}$ following in an analogous way.

Recall that, for any Lipschitz domain $D$ with characteristic length scale $L$, there exists a dimensionless quantity $\Ctr$ such that 
 \beq\label{eq:Ctr}
 \N{v}^2_{\partial D} \leq \Ctr \N{v}_{D} \big( L^{-1} \N{v}_{D} + \N{\nabla v}_{D}\big);
 \eeq
see, e.g., \cite[Theorem 1.5.1.10, last formula on p. 41]{Gr:85}.

 \begin{lemma}[Continuity of $a_\eps$]\label{lem:cont}
Assume that $\Omega$ has characteristic length $L$. Then, for all $u, v\in H^1_{0,\Gamma_D}(\Omega)$,
\beqs
|a_{\eps}(u,v)| \le \Ccont \|u\|_{1,k} \|v\|_{1,k},  
\eeqs
where
\beqs
\Ccont: =  \max\big\{A_{\max}, \sqrt{2}n_{\max}\big\}+ \sqrt{ n_{\max}}\Ctr \frac{|\eta|}{k} 
\left(\frac{1}{2}+ \frac{1}{kL}\right) . 
\eeqs
\end{lemma}

Observe that if $kL\geq 1$ and $|\eta|\leq  Ck$ for some $C>0$, independent of all parameters, 
then $\Ccont$ is independent of $k$, $\eta$, and $L$.

\begin{proof}[Proof of Lemma \ref{lem:cont}]
By the definition of $a_\eps$ \eqref{eq:a_eps}, the Cauchy--Schwarz inequality, and the inequalities \eqref{eq:Alimits}, \eqref{eq:nlimits}, and \eqref{eq:eps},
\beqs
|a_{\eps}(u,v)|\leq A_{\max}\|\nabla u\| \|\nabla v\| + \sqrt{2}k^2n_{\max}\|u\| \|v\| +\sqrt{ n_{\max}} |\eta|\N{u}_{\Gamma_I}\N{v}_{\Gamma_I}.
\eeqs
By the multiplicative trace inequality \eqref{eq:Ctr} (noting that $\Gamma_I \subset \partial \Omega$) and the inequality $a b \leq \frac{1}{2} (a^2 + b^2)$, 
\beqs
\|v\|^2_{\Gamma_I} \leq \frac{\Ctr}{k}k\|v\| \left( \frac{k}{kL} \N{v} +  \|\nabla v\| \right)\leq \frac{\Ctr}{k}
\left(\frac{1}{2}+ \frac{1}{kL}\right)
\N{v}_{1,k}^2;
\eeqs
The result then follows from combining these last two inequalities.
\end{proof}

For later use, we define $\Ccontell$, the continuity constant for $a_{\eps,\ell}$:
\beq\label{eq:Ccontell}
\Ccontell: =   \max\big\{A_{\max, \ell}, \sqrt{2}n_{\max, \ell}\big\}+ \sqrt{ n_{\max,\ell}}\Ctr \frac{|\eta|}{k} 
\left(\frac{1}{2}+ \frac{1}{kH_\ell}\right). 
 \eeq

\begin{lemma}[Coercivity of $a_\eps$]\label{lem:coer}
With $\eta$ chosen as in Assumption \ref{ass:eta},
for all $k>0$ and $v\in H^1(\Omega)$,  
\beq\label{eq:Ccoer}
\vert a_{\eps}(v,v)\vert 
\geq \frac14 \min\big\{ A_{\min}, n_{\min}\big\}\frac{\vert \eps\vert }{k^2}\,  \N{v}^2_{1,k}.
\eeq
\end{lemma}

\begin{proof}
  Within this  proof (only) we use the notation   $\|\mathbf{v}\|_{A}^2 = (A \mathbf{v}, \mathbf{v})$, $\|v\|_n^2 = (nv,v)$, and $\Vert v \Vert_{\partial \Omega, \sqrt{n}}^2 = \langle\sqrt{n} v, v\rangle_{\partial \Omega}$. 
 for any vector-valued 
  function $\mathbf{v}$ and  scalar-valued function $v$   defined on $\Omega$,
  where $A,n$ are the coefficients described in  Assumption \ref{ass:1}.
  Let $z= \sqrt{k^2+ \ri \eps}$ (defined, as in Assumption \ref{ass:eta}, with the branch cut on the real axis).
Writing $z = p+iq$ and using 
the definitions of $a_\eps$ \eqref{eq:a_eps},   we have 
$$ a_{\eps}(v,v) \ = \ \Vert \nabla v \Vert_{A}^2 - (p+\ri q )^2 \Vert v \Vert_
{n}^2  - \ri \eta \Vert v \Vert_{\partial \Omega, \sqrt{n}}^2  \ .
$$
\begin{equation*}
\text{Therefore} \quad \quad\quad\quad 
 \Im \left[ -(p-\ri q)  a_{\eps}(v,v)\right] \  =\  q \Vert \nabla v \Vert_{A}^2 +  
q (p^2+ q^2) \Vert v \Vert_
{n}^2  + \Re\left[  (p - \ri q)\eta \right] \Vert v \Vert_{\partial \Omega, \sqrt{n}}^2  \ . 
\end{equation*}
Hence, dividing through by $\vert z \vert = \sqrt{p^2 + q^2} $, and setting 
$\Theta = - \overline{z} /\vert z \vert$,  we have 
\begin{equation*}
\Im \left[ \Theta   a_{\eps}(v,v)\right] \  =\  \frac{\Im(z)}{\vert z\vert} \left[ \Vert \nabla v \Vert_{A}^2 +  
 \vert z \vert^2  \Vert v \Vert_
{n}^2\right]  + \frac{\Re\left(\overline{z} \eta \right)}{|z|} \Vert v \Vert_{\partial \Omega, \sqrt{n}}^2  \ .  
\end{equation*}
With either of the choices $\eta=\sign(\eps)k$ or $\eta= \sqrt{k^2 + \ri \eps} (=z)$, we have that $\Re(\overline{z} \eta)\geq 0$
(when $\eta=\sign(\eps)k$, this follows from the second inequality in \eqref{eq:eta}).
Furthermore, the definition of $z$ implies that
\beqs
\frac{\Im (z)}{|z|} \geq  \frac{1}{2 \sqrt{1 + \sqrt{2}}}\, \frac{|\eps|}{k^2}\, \geq\, \frac14  \,\frac{|\eps|}{k^2}.
\eeqs
see \cite[Equation 2.10]{GrSpVa:17}.
The result \eqref{eq:Ccoer} follows using these inequalities, along with the inequalities \eqref{eq:Alimits}, \eqref{eq:nlimits}, and $|z|\geq k$.
\end{proof} 

 \subsection{Bounds on the solution operators of the continuous global and local problems}
  
Because the matrices $\MA_{\eps, \ell}^{-1} $ \igg{appearing in  \eqref{eq:ASpc}} correspond to problems with zero
 impedance data, we only need here  to consider the case $g_I=0$ in Definition \ref{def:TEDP} and its local analogue. First we define what we mean by bounds on the global and local solution operators (at the continuous level). 
  
 \begin{definition}[Bound on global solution operator]\label{def:Csol}
 Let Assumption \ref{ass:1} hold and assume further than $A$ is piecewise Lipschitz when $d=3$ and $\eps=0$. 
 Assume that $\Omega$ has characteristic length scale $L$.
 Then, by Lemma \ref{lem:wp}, the TEDP of Definition \ref{def:TEDP} is well-posed, and there exists $\Csol = \Csol (k, \eps, A, n,\Omega_-,\widetilde{\Omega})$ such that the solution $u$ of the TEDP with $g_I=0$ satisfies 
\beq\label{eq:Csol}
\N{u}_{1,k} \leq \Csol L \N{f}
 \quad\tfa k>0.
\eeq
 \end{definition}
 
The factor of $L$ on the right-hand side of \eqref{eq:Csol} is chosen so that $\Csol$ is a dimensionless quantity.
 We define $\Csolell$ in an analogous way as a bound on the solution operator for the local problems  (at the continuous level) on each subdomain $\Omega_\ell$. We first define 
\beqs
H^1_{0,\Gamma_D}(\Omega_\ell):= \{ z \in H^1(\Omega_\ell): \ z = 0 \  \text{on}  \
  \partial \Omega_\ell \cap \Gamma_D \}.
\eeqs

  \begin{definition}[Bound on local solution operator]\label{def:Csolell}
 Let Assumption \ref{ass:1} hold and assume further that $A$ is piecewise Lipschitz when $d=3$ and $\eps=0$.  
For any $\ell$, let $u\in H^1_{0,\Gamma_D}(\Omega_\ell)$ be the solution  of the variational problem
\beq\label{eq:vp_local_cts}
a_{\eps,\ell}(u,v) = (f, v)_{\Omega_\ell}, \quad  \tfa v \in H^1_{0,\Gamma_D}(\Omega_\ell).
\eeq
Then, by Lemma \ref{lem:wp}, $u$ exists, is unique, and there exists $\Csolell = \Csolell (k, \eps, A, n,\Omega_\ell)$ such that 
\beq\label{eq:Csolell}
\N{u}_{1,k,\Omega_\ell} \leq \Csolell H \N{f}_{\Omega_\ell} \quad\tfa k>0.
\eeq
 \end{definition}

%

The rest of this subsection consists of obtaining bounds on $ \Csol$ and $\Csolell$, with the main result contained in Corollary \ref{cor:Csolfinal} below.  First we use 
the coercivity result of Lemma \ref{lem:coer} combined with the Lax--Milgram theorem to obtain the following result.

 \begin{corollary}[Bounds on $\Csol$ and $\Csolell$ when $|\eps|>  0$]\label{cor:Csol}
 Let Assumption \ref{ass:1} hold. If $|\eps|> 0$, then 
 \beq\label{eq:Csolcoer}
 \Csol 
 \leq \frac{4 k}{|\eps|L}\Big(\min\big\{ A_{\min}, n_{\min}\big\}\Big)^{-1}
 \quad\tand\quad
  \Csolell
 \leq \frac{4 k}{|\eps|H_\ell}\Big(\min\big\{ A_{\min, \ell}, n_{\min, \ell}\big\}\Big)^{-1}.
 \eeq
 \end{corollary}
 
\bpf
Since $a_\eps$ is continuous and coercive, the Lax--Milgram theorem implies that the solution of the TEDP with $|\eps|\neq 0$ satisfies
\beqs
\N{u}_{1,k}\leq \frac{4 k^2}{|\eps|}\Big( \min\big\{ A_{\min}, n_{\min}\big\}\Big)^{-1}\sup_{v\neq 0}\frac{\big| (f,v)\big|}{\N{v}_{1,k}}
\leq \frac{4 k}{|\eps|}\Big(\min\big\{ A_{\min}, n_{\min}\big\}\Big)^{-1}\N{f},
\eeqs
and the bound on $\Csol$ in \eqref{eq:Csolcoer} follows on comparison with \eqref{eq:Csol}; the bound on $\Csolell$ is obtained  analogously.
\epf 

\
 
Estimates in the case $\vert \eps\vert =0$ are more delicate and are discussed at the end of this subsection.
First we show how bounds for  $\eps = 0$
imply bounds for  $\eps \not = 0$ (although  these are only  optimal when $\vert \eps \vert$ is small -- see Corollary \ref{cor:Csolfinal} below). 
For this purpose, in the following lemma we write  $\Csol=\Csol(\eps)$ and $\Csolell=\Csolell(\eps)$ to indicate the dependence of these quantities on $\eps$.

\begin{lemma}\label{lem:pert} Let Assumption \ref{ass:1} hold. For all $\vert \eps \vert >0 $,
\beq\label{eq:pert}
\Csol(\eps) \leq \Csol(0)\left( 1 + \frac{n_{\max}}{n_{\min}}\right),
\eeq
with the analogous result also holding for $\Csolell$, $\ell= 1,\ldots, N$.
\end{lemma}

\bpf
If $u$ is the solution of the TEDP with $g_I=0$ and $\eps\neq 0$, then
\beqs
\nabla\cdot(A\gu) + k^2 n u = -f - \ri \eps nu,
\eeqs
and so, by the definition of $\Csol(0)$, 
\beq\label{eq:pert1}
\N{u}_{1,k} \leq \Csol(0) L \Big( \N{f} + |\eps| n_{\max} \N{u}\Big).
\eeq
Putting $g_I=0$ and $v=u$ in the variational problem of the TEDP \eqref{eq:contvp} and then taking the imaginary part, we obtain that
\beqs
\eps (n u,u) + (\Re \eta) \langle \sqrt{n} u, u\rangle_{\Gamma_I} = - \Im (f, u).
\eeqs
The second inequality in \eqref{eq:eta} implies that the two terms on the left-hand side have the same sign, and thus
\beqs
|\eps| n_{\min} \N{u}^2 \leq \N{f} \N{u} \quad \text{ so that } \quad |\eps| n_{\min} \N{u} \leq \N{f}.
\eeqs
Inputting this last inequality into \eqref{eq:pert1}
and recalling the definition of $\Csol(\eps)$, we obtain \eqref{eq:pert}.
The result for $\Csolell(\eps)$ is proved similarly.
\epf

\ 
 
When $\eps=0$, bounds on $\Csol$ and $\Csolell$ depend on whether or not the problem is \emph{nontrapping}.  In the context of scattering by obstacles, the problem is nontrapping when (informally) all the rays starting in a neighbourhood of the scatterer $\supp(I-A)\cup \supp(1-n)\cup \Omega_-$ escape from that neighbourhood in a uniform time. The analogous concept here is that all the rays hit the far-field  boundary $\Gamma_I$
in a uniform time.
 
\begin{definition}[Nontrapping]\label{def:nontrapping}
$A,n,$ and $\Omega_{-}$ are  \emph{nontrapping} if they are  all $C^\infty$ and there exists $T>0$ such that all the Melrose--Sj\"ostrand generalized bicharacteristics (see \cite[Section 24.3]{Ho:85}) starting in $\Omega$ at time $t=0$ hit $\Gamma_I$ at some  time $t\leq T$.
\end{definition} 

\bre[Understanding the definition of nontrapping]
Away from $\Gamma_D$,  the \emph{bicharacteristics} are defined as the solution $(\bx(t),\bxi(t))$ (with $\bx$ understood as position and $\bxi$ understood as momentum) of the Hamiltonian system 
\beqs
\dot{x_i}(t) = \partial_{\xi_i}p\big(\bx(t), \bxi(t) \big), \qquad
\dot{\xi_i}(t)
 = -\partial_{x_i}p\big(\bx(t), \bxi(t) \big),
\eeqs
where the Hamiltonian
is given by the semi-classical principal symbol of the Helmholtz equation \eqref{eq:hetero} with $\eps=0$, namely
\beqs
p(\bx,\bxi):= \sum_{i=1}^d\sum_{j=1}^{d} A_{ij}(\bx)\xi_i \xi_j - n(\bx).
\eeqs
When $(\bx(t),\bxi(t))$ is a bicharacteristic, $\bx(t)$ is called a \emph{ray} (i.e.~the rays are the projections of the bicharacteristics in space).
The notion of \emph{generalized} bicharacteristics (in the sense of Melrose--Sj\"ostrand \cite{MeSj:78, MeSj:82}) is needed to rigorously describe how the bicharacteristics interact with $\Gamma_D$.

The significance of bicharacteristics in the study of the wave/Helmholtz equations stems from the result 
of  \cite[\S VI]{DuHo:72} that (in the absence of boundaries) singularities of pseudodifferential operators (understood in terms of the \emph{wavefront set}) travel along null bicharacteristics (i.e.~those on which $p(\bx,\bxi)=0$); see, e.g., \cite[Chapter 24]{Ho:85}, \cite[\S12.3]{Zw:12}.
\ere

\begin{theorem}[Bound on $\Csol$ and $\Csolell$ when $\eps=0$]\label{thm:Csol}
Suppose that $A,n,$ and $\Omega_-$ are nontrapping in the sense of Definition \ref{def:nontrapping} and assume further that $n=1$.
Then, given $k_0>0$, there exists $\widetilde{C}$, dependent on $A$, $\Omega$, and $k_0$ but independent of $k$, 
such that 
\beqs
\Csol \leq \widetilde{C}\quad \tfa k\geq k_0.
\eeqs
Furthermore, if $\partial \Omega_\ell$ is $C^\infty$, then 
\beq\label{eq:Csolellbound}
\Csolell\leq \widetilde{C}\quad\tfa k \geq k_0 \quad \text{ and for all } \ell=1,\ldots,N.
\eeq
\end{theorem}

\bpf
The bound on $\Csol$ follows from combining \cite[Theorem 1.8]{BaSpWu:16} (which proves the bound when $A=I$ and $n=1$) and \cite[Remark 5.6]{BaSpWu:16} (which describes how the bound also holds when $A\neq I$). Note that \cite{BaSpWu:16} considers the case when $\Omega_-= \emptyset$ (i.e.~there is no Dirichlet obstacle), but the results from \cite{BaLeRa:92} used in the proof of  \cite[Theorem 1.8]{BaSpWu:16} (specifically \cite[Theorems 5.5 and 5.6 and Proposition 5.3]{BaLeRa:92}) also hold when there is a nontrapping Dirichlet obstacle in the domain (see \cite[Equation 5.2]{BaLeRa:92}).

The idea behind the proof of the bound \eqref{eq:Csolellbound} on $\Csolell$  is that (informally) if $A$ and $\Omega_-$ are nontrapping 
then so are $A|_{\Omega_\ell}$ and $\Omega_\ell$, since if the rays starting in $\Omega$ all escape to the impedance boundary $\Gamma_I$, then all the rays in a given subdomain $\Omega_\ell$ must all hit $\partial \Omega_\ell \setminus \Gamma_D$.
More formally, if $A$ and $\Omega_-$ are nontrapping in the sense of Definition \ref{def:nontrapping}, then 
there exists $0<T_\ell\leq T$ such that all the Melrose--Sj\"ostrand generalized bicharacteristics starting in $\Omega_\ell$ at time $t=0$ hit $\partial \Omega_\ell \setminus \Gamma_D$ for time $t\leq T_\ell$. The results of \cite{BaLeRa:92, BaSpWu:16} discussed in the previous paragraph can therefore be applied to the 
problem \eqref{eq:vp_local_cts} on $\Omega_\ell$, resulting in an analogous bound on $\Csolell$.
\epf
 
Based on the recent work \cite{GaSpWu:18} for the exterior Dirichlet problem (as opposed to its truncated variant), one expects the constant $ \widetilde{C}$ in Theorem \ref{thm:Csol} to be related to the length of the longest ray in $\Omega$;  see \cite[Theorems 1 and 2, and Equation 6.32]{GaSpWu:18}.
 
Theorem \ref{thm:Csol} assumes that $\partial\Omega$ and $\partial \Omega_\ell$ are $C^\infty$ (or rather, $C^m$ for some large and unspecified $m$). Because each $\Omega_\ell$ is assumed to be a union of elements of the mesh $\cT^h$,  these smoothness requirements are not realisable in practical implementations of the preconditioner in \S\ref{sec:preconditioner}. However, motivated by Theorem \ref{thm:Csol} we prove results about the performance of the preconditioner under the following ``nontrapping-type" assumption, with Theorem \ref{thm:Csol} giving one scenario when it holds.

\begin{assumption}[``Nontrapping-type" assumption on $\Csolell$]\label{ass:nontrapping}
$A,n,\Omega$, and $\Omega_\ell$, $\ell=1,\ldots,N$ are such that, given any $k_0>0$, there exists
$\widetilde{C}$, dependent on $A, n,\Omega$, $\Omega_\ell$, and $k_0$, but independent of $k$, such that, when $\eps=0$,
\beqs
\Csolell\leq \widetilde{C} \quad\tfa k \geq k_0 \quad\text{ and for all } \ell=1,\ldots,N.
\eeqs
\end{assumption}
  
We now summarise the bounds on $\Csol$, $\Csolell$ coming from Corollary \ref{cor:Csol}, Lemma \ref{lem:pert}, and Assumption \ref{ass:nontrapping}.
For brevity, we state these bounds only for $\Csolell$, but analogous results hold for $\Csol$.

\begin{corollary}[Summary of bounds on $\Csolell$]\label{cor:Csolfinal}
Let Assumption \ref{ass:1} hold. Then 
  \beq\label{eq:Csolfinal2}
\Csolell \leq \frac{4 k}{|\eps|H_\ell}\Big( \min\big\{ A_{\min, \ell}, n_{\min, \ell}\big\}\Big)^{-1} \quad\tfa |\eps|, k, H_\ell >0.
\eeq
Furthermore, if Assumption \ref{ass:nontrapping} holds then, given $k_0>0$, there exists $\widetilde{C}$ (depending on $k_0$ but independent of $k$) such that, for all $\ell=1,\ldots, N$, 
\beq\label{eq:Csolfinal1}
\Csolell \leq  \widetilde{C}\left( 1 + \frac{n_{\max}}{n_{\min}}\right), \quad
\quad\tfa |\eps|\geq 0, \,H_\ell >0,\,\tand k\geq k_0.
\eeq
\end{corollary}

\noi Observe that   \eqref{eq:Csolfinal1} shows that $\Csolell$ is bounded (independently of $k, \eps,$ and $H_\ell$) for $|\eps|H_\ell/k \lesssim1$, whereas \eqref{eq:Csolfinal2} shows that
$\Csolell$ decreases with increasing  $|\eps|H_\ell/k \gg 1$.

\bre
Results essentially equivalent to those summarised in Corollary \ref{cor:Csolfinal} in the case $A=I$ and $n=1$ were obtained in \cite[Theorem 2.7 and 2.9]{GaGrSp:15} and more recently in \cite[Lemma 4.2]{MeSaTo:19}. Recall that, when $A=I$ and $n=1$, the nontrapping bound on the solution of the TEDP for $\eps=0$ is available for smooth $\Omega'$ and $\Omega_-$ by \cite[Theorem 1.8]{BaSpWu:16}  and for Lipschitz star-shaped $\Omega'$ and $\Omega_-$ by \cite[Remark 3.6]{MoSp:14} (following earlier work by \cite{He:07}).
\ere
 
  \subsection{Bounds on the solution operators of the discrete local problems}
  \label{subsec:disclocal}
  
Recall the definition \eqref{eq:localV} of the local spaces $ \cV_{\ell}^h$. 
The discrete local problems are:~given $F$ a continuous linear functional on $ \cV_{\ell}^h$, find $u_{h,\ell} \in  \cV_{\ell}^h$ such that
\begin{equation}\label{eq:localVP}
a_{\eps, \ell}(u_{h,\ell}, v_{h,\ell}) = F(v_{h,\ell})\quad \text{for all}~v_{h,l}\in  \cV_{\ell}^h.
\end{equation}
The following lemma is the discrete analogue of Corollary \ref{cor:Csol}, and is an immediate consequence of the coercivity property proved in Lemma  \ref{lem:coer}.

\begin{lemma}[Bounds on the solutions of the discrete local problems when $|\eps|>0$] \label{lem:discrete1}
Let Assumption \ref{ass:1} hold.
For all $|\abs| > 0$, $h$ and $p$
\eqref{eq:localVP} has a 
unique solution $u_{h,\ell}$ which satisfies
\beq \label{eq:FEest1}
\Vert u_{h,\ell} \Vert_{1,k, \Omega_\ell}  \leq \frac{4 k^2}{|\eps|}\Big( \min\big\{ A_{\min}, n_{\min}\big\}\Big)^{-1}
\max_{v_h \in  \cV_{\ell}^h}  
\left( \frac{\vert F(v_h) \vert}{ \Vert v_{ h} \Vert_{1,k,\Omega_\ell} } \right). 
\eeq
\end{lemma}    

We now prove a bound on the solution of the discrete local problems that is valid when $|\eps|\geq0$, and gives a bound with better $k$-dependence than \eqref{eq:FEest1} when $\vert \eps \vert \ll k$.

We first define the operator $\cS_{\eps,\ell}^*: L^2(\Omega_\ell)\mapsto H^1_{0,\Gamma_{D}}(\Omega_\ell)$ 
as the solution of the variational problem
$$
a_{\eps,\ell}(v, \cS_{\eps,\ell}^*f) = ( v, f)_{\Omega_\ell},   \quad \tfa v \in H^1_{0,\Gamma_{D}}(\Omega_\ell);
$$
i.e.~$\cS_{\eps,\ell}^*$ is the solution operator of the adjoint Helmholtz problem on $\Omega_\ell$ with data in $L^2(\Omega_\ell)$.

\ble[Bound on  $\cS_{\eps,\ell}^*$ in terms of $\Csolell$]\label{lem:discretebound1}
 Let Assumption \ref{ass:1} hold and assume further than $A$ is piecewise Lipschitz when $d=3$ and $\eps=0$.  
Then
\begin{equation}\label{eq:adjoint}
\|\cS_{\eps,\ell}^*f\|_{1,k,\Omega_\ell} \leq \Csolell H \|f\|_{L^2(\Omega_\ell)}.
\end{equation}
\ele

\bpf This follows from definition \eqref{eq:Csolell} and   the fact that, if $u$ is the solution of the variational problem \eqref{eq:vp_local_cts} with data $\overline{f}$, then $a_{\eps,\ell}(v, \overline{u}) = ( v, f)$ for all $v \in H^1_{0,\Gamma_{D}}(\Omega_\ell)$; i.e.~the solution of the adjoint problem with data $f$ is the complex-conjugate of the solution of the standard problem with data $\overline{f}$.
\epf

Following the notation introduced in \cite{Sa:06}, we then define
\beq\label{eq:etaV}
\eta(\cV^h_\ell) := \sup_{f\in L^2(\Omega_\ell)\backslash \{0\}} \min_{v_h\in \cV^h_\ell}\frac{\|\cS_{\eps,\ell}^*f - v_h\|_{1,k,\Omega_\ell}}{\|f\|_{L^2(\Omega_\ell)}}.
\eeq
(Although
this notation clashes slightly with our choice of $\eta$ for the impedance parameter in \eqref{eq:impedance},
we persist here  with
these notations since both are overwhelmingly used in the literature.)

    \begin{lemma}[Bounds on the solutions of the discrete local problems when $|\eps|\geq 0$] \label{lem:discretebound2}
 Let Assumption \ref{ass:1} hold and assume further that $A$ is piecewise Lipschitz when $d=3$ and $\eps=0$.  
If 
\beq\label{eq:meshthresh1}
\frac{|\eps-\sign(\eps)k^2|}{k} \eta(\cV^h_\ell) \leq \frac{ \min\big\{A_{\min, \ell}, n_{\min, \ell}\big\}}{8 \Ccontell \, n_{\max,\ell}}
\eeq
\end{lemma}
Then \eqref{eq:localVP} has a 
unique solution $u_{h,\ell}$ which satisfies
\beq \label{eq:FEest2}
\Vert u_{h,\ell} \Vert_{1,k, \Omega_\ell}  \leq 
\left(
\frac{ 9  + 8 \Csolell H n_{\max,\ell}|\eps-\sign(\eps)k^2| k^{-1}}{  \min\big\{A_{\min, \ell}, n_{\min, \ell}\big\}}
\right)
\max_{v_h \in  \cV_{\ell}^h}  
\left( \frac{\vert F(v_h) \vert}{ \Vert v_{ h} \Vert_{1,k,\Omega_\ell} } \right),
\eeq

\bre[Understanding the condition \eqref{eq:meshthresh1}]\label{rem:meshthresh1}
The summary is that, under suitable smoothness conditions on $A$, $n$, and $\Omega$ and when $\eps\ll k^2$, 
\eqref{eq:meshthresh1} is essentially the condition ``$k^{p+1} h^p H$ sufficiently small", where the constant in ``sufficiently small" in principle depends on $A$ and $n$. (When $p=1$ and $H\sim 1$, this is the familiar ``$k^2h$ sufficiently small" condition for quasioptimality -- recall  Remark \ref{rem:pollution}.)   
In our current setting, \cite[Lemma 2.13]{ChNi:19} shows that  provided (i) $\Csol$ is bounded independently of $k$, and (ii)  $\Gamma_D, \Gamma_I \in C^{\gamma+1,1}$  and $A \in C^{\gamma,1}$ for some integer  
$\gamma \geq  p-1$,
then
\beqs
\eta(\cV_{\ell}^h)\leq C  \big(H(hk)^p + h \big),  
\eeqs
where $C$ depends on $A, n$ and $p$.  \cite[Lemma 2.13]{ChNi:19} also contains analogous  estimates
for polygonal domains and discontinuous coefficients, assuming local mesh refinement to treat singularities
(see \cite[Equations 2.27 and 2.28]{ChNi:19}).  
In the simpler setting of $p=1$,   $A, n \in C^{0,1}$, and $\Omega$ a convex polygon,  \cite[Theorem 4.5]{GrSa:20} showed that
$\eta(\cV_{\ell}^h)\leq C  H\big((hk) + (hk)^2 \big)$. 
In either case, if \eqref{eq:meshthresh1} is to hold for all $\eps \ll k^2$, then the requirement  is that $k(kh)^pH$ should be sufficiently small.
\ere

\bpf[Proof of Lemma \ref{lem:discretebound2}]
By a standard argument (e.g., \cite[Theorem 2.1.44]{SaSc:11}), the bound \eqref{eq:FEest2} is equivalent to proving the ``inf-sup condition'', namely that, given $v_{h,\ell} \in \cV^h_\ell$, there exists $w_{h,\ell}\in \cV^h_\ell$ such that
\beq\label{eq:infsup1}
\frac{|a_{\eps,\ell}(v_{h,\ell}, w_{h,\ell})|}{\N{v_{h,\ell}}_{1,k,\Omega_\ell} \N{w_{h,\ell}}_{1,k,\Omega_\ell}}
\geq 
\frac{  \min\big\{A_{\min, \ell}, n_{\min, \ell}\big\}}{9 + 8 \Csolell H n_{\max,\ell} |\eps-\sign(\eps)k^2| k^{-1}}
\eeq
By the definition \eqref{eq:local-vp} of $a_{\eps,\ell}$ , for any $z\in H^1(\Omega_\ell)$,
\begin{align*}
a_{\eps,\ell}(v_{h,\ell}, v_{h,\ell} + z)  &= a_{\eps,\ell}(v_{h,\ell}, v_{h,\ell}) + a_{\eps,\ell}(v_{h,\ell}, z) \\
&= a_{\sign(\eps)k^2, \ell}(v_{h,\ell}, v_{h,\ell}) - \ri \big( \eps- \sign(\eps)k^2\big) \int_\Omega n |v_{h,\ell}|^2 +a_{\eps,\ell}(v_{h,\ell}, z). 
\end{align*}
We therefore define $z$ as the solution of the variational problem 
\beqs
a_{\eps,\ell}(w, z)= \ri \big( \eps- \sign(\eps)k^2\big) \int_\Omega n w \,\overline{v_{h,\ell}} \quad\tfa w \in H^1_{0,\Gamma_D}(\Omega_\ell),
\eeqs
i.e.~$z = -\ri ( \eps- \sign(\eps)k^2)\cS^*_{\eps,\ell}( n v_{h,\ell})$. With this choice of $z$, 
\beq\label{eq:E1}
\big|a_{\eps,\ell}(v_{h,\ell}, v_{h,\ell} + z)\big| \, = \,\left\vert a_{\sign(\eps)k^2, \ell}(v_{h,\ell}, v_{h,\ell})\right\vert \,   \geq\,  \frac{1}{4} \min\big\{A_{\min, \ell}, n_{\min, \ell}\big\} \N{v_{h,\ell}}^2_{1,k,\Omega_\ell}, 
\eeq
where in the last step we used  the analogue of Lemma \ref{lem:coer} on $\Omega_\ell$. Now let $z_{h,\ell}$ be the best approximation to $z$ in the space $\cV^h_\ell$; by the definition of $\eta(\cV^h_\ell)$ \eqref{eq:etaV} we have
\beq\label{eq:E2}
\N{z-z_{h,\ell}}_{1,k,\Omega_\ell} \leq \eta(\cV^h_\ell)\, \big|  \eps- \sign(\eps)k^2\big| \,n_{\max,\ell}\N{v_{h,\ell}}_{\Omega_\ell}.
\eeq
Then, using Lemma \ref{lem:cont} and the inequalities \eqref{eq:E1} and \eqref{eq:E2}, we have
\begin{align}\nonumber
\big|a_{\eps,\ell}(v_{h,\ell}, v_{h,\ell} + z_{h,\ell})\big|
&\geq \big|a_{\eps,\ell}(v_{h,\ell}, v_{h,\ell} + z)\big| - \big|a_{\eps,\ell}(v_{h,\ell},z-z_{h,\ell})\big|,\\ \nonumber
&\hspace{-2cm} \geq \frac{1}{4}  \min\big\{A_{\min, \ell}, n_{\min, \ell}\big\} \N{v_{h,\ell}}^2_{1,k,\Omega_\ell} - \Ccontell \N{v_{h,\ell}}_{1,k,\Omega_\ell} \N{z-z_{h,\ell}}_{1,k,\Omega_\ell},\\
&\hspace{-2cm} \geq \Big(\frac{1}{4}  \min\big\{A_{\min, \ell}, n_{\min, \ell}\big\}  -  \Ccontell \eta(\cV^h_\ell) \big|  \eps- \sign(\eps)k^2\big| k^{-1} n_{\max,\ell}
\Big) \N{v_{h,\ell}}^2_{1,k,\Omega_\ell}.\label{eq:E3}
\end{align}
Also, by the triangle inequality and the bounds  \eqref{eq:E2} and \eqref{eq:adjoint},
\begin{align}\nonumber
\N{v_{h,\ell} + z_{h,\ell}}_{1,k,\Omega_\ell} 
&\leq \N{v_{h,\ell}}_{1,k,\Omega_\ell}  + \N{z-z_{h,\ell}}_{1,k,\Omega_\ell} + \N{z}_{1,k,\Omega_\ell} ,\\
& \hspace{-2.5cm}\leq\Big ( 1 +   \big|  \eps- \sign(\eps)k^2\big| k^{-1} \eta(\cV^h_\ell)n_{\max,\ell}  +\Csolell H  \big|  \eps- \sign(\eps)k^2\big| k^{-1}n_{\max,\ell} \Big) 
\N{v_{h,\ell}}_{1,k,\Omega_\ell},\label{eq:E4}
\end{align}
and combining \eqref{eq:E3} and \eqref{eq:E4} we obtain with $w_{h,\ell} : = v_{h,\ell} + z_{h,\ell}$, 
\beq
\frac{|a(v_{h,\ell}, w_{h,\ell})|}{\N{v_{h,\ell}}_{1,k,\Omega_\ell} \N{w_{h,\ell}}_{1,k,\Omega_\ell}}
\geq 
\frac{\frac{1}{4}  \min\big\{A_{\min, \ell}, n_{\min, \ell}\big\} -  \big|\eps-\sign(\eps)k^2\big| k^{-1} \eta(\cV^h_\ell)\Ccontell n_{\max,\ell}}{ 
 1 +   \big|  \eps- \sign(\eps)k^2\big| k^{-1} \eta(\cV^h_\ell)n_{\max,\ell}  +\Csolell H  \big|  \eps- \sign(\eps)k^2\big| k^{-1}n_{\max,\ell}
 }.
\eeq
The result \eqref{eq:infsup1} follows under the constraint \eqref{eq:meshthresh1}, noting that, from the definition \eqref{eq:Ccontell} of $\Ccontell$, $\Ccontell \geq \min\{A_{\min,\ell}, n_{\min,\ell}\}$.
\epf

\bre[How to improve the condition on $h$ and $p$ in \eqref{eq:meshthresh1}]\label{rem:meshthresh2}
Remark \ref{rem:meshthresh1} described how, the condition on $h$ and $p$ in \eqref{eq:meshthresh1} (which is a sufficient condition for the bound \eqref{eq:FEest2} to hold) is the requirement that $k^{p+1}h^pH$ is sufficiently small. 
Existence and uniqueness of $u_{h,\ell}$ for $|\eps|\geq 0$, along with a bound with identical $k$-dependence to
\eqref{eq:FEest2}, can be proved 
under the weaker requirement that $k^{2p+1}h^{2p}H^2$ is sufficiently small
using the results of \cite{Pe:19}, under additional smoothness requirements on $A,n$, and $\Omega$  when $p>1$.
\ere



\section{Domain  decomposition, interpolation,  and local projections}\label{sec:dd_results}

\subsection{Domain decomposition and interpolation}
We introduce some technical assumptions concerning  the overlapping subdomains $\Omega_\ell, \ \ell=1,\ldots, N$  introduced in \S \ref{sec:preconditioner}.
For 
each $\ell = 1, \ldots , N$, we 
let $\mathring{\Omega}_\ell$ denote the
part of $\Omega_\ell$ that is not overlapped by any other subdomains. (Note that  $\mathring{\Omega}_\ell = \emptyset$ is possible.) 
For $\mu>0$ let $\Omega_{\ell, \mu}$ denote the set of points in
$\Omega_\ell$ that are a distance no more than $\mu$ from the interior 
boundary $\partial \Omega_\ell\backslash \Gamma $ . Then we assume that there exist constants  $0<\delta_\ell \leq
H_\ell$ 
and $0<b<1$ such that, for each $\ell = 1, \ldots , N$,  
  \beqs
  \Omega_{\ell, b \delta_\ell }\subset \Omega_\ell\backslash
  \mathring{\Omega}_\ell \subset \Omega_{\ell,\delta_\ell };
  \eeqs
The case when $\delta_\ell \geq c H_\ell$  for some constant $c$ independent of  $\ell$ is called \emph{generous} overlap.
We introduce the parameter
\beqs
 \delta : = \min_{\ell = 1, \ldots , N} \delta _\ell.  
 \eeqs
  We  make the {\em finite-overlap assumption}: There exists a finite $\Lambda > 1$ independent of $N$ such that 
\begin{equation} 
\Lambda \ = \ \max \big\{ \#
  \Lambda(\ell): \ell = 1, \ldots , N\big\},    \quad \text{where} \quad   
\Lambda(\ell) = \big\{ \ell' :  \Omega_\ell \cap \Omega_{\ell'} \not =
\emptyset \big\} \ .   \label{eq:finoverlap}
\end{equation}       
It follows immediately from  \eqref{eq:finoverlap}  that, for all $v\in L^2(\Omega)$,
\beq\label{eq:finoverEuan1}
\sum_{\ell=1}^N \N{v}^2_{L^2(\Omega_\ell)} \leq \Lambda \N{v}^2 \ \text{ and } \
\sum_{\ell=1}^N \N{v}^2_{1,k,\Omega_\ell} \leq \Lambda \N{v}^2_{1,k},  \  \text{when} \ \  v\in H^1(\Omega).
\eeq
We recap the following result from \cite[Lemma 3.6]{GrSpZo:18}, \cite[Lemma 4.2]{GrSpVa:17}.
\begin{lemma}\label{lem:norm_of_sum}
For each  $\ell = 1, \ldots, N$,  choose any  function $v_{\ell} \in H^1(\Omega)$,  
with $\supp \, v_\ell \subset \overline{\Omega_\ell}$. Then    
\begin{equation*}
\quad \left\Vert \, \sum_{\ell=1}^N  v_{\ell} \, \right\Vert_{1,k}^2 \ \leq \Lambda 
\sum_{\ell=1}^N \Vert v_{\ell} \Vert_{1,k, \Omega_\ell}^2\ .
\end{equation*}  
\end{lemma}

Concerning the partition of unity introduced in \eqref{POUstar},
we assume  the functions $\chi_\ell$ to be continuous piecewise linear on the mesh $\cT^h$, and to satisfy     
\begin{equation}  
\label{eq:derivpou} \Vert  \nabla \chi_\ell  \Vert_{L^\infty(\tau)} 
\leq \frac{\Cchiell}{\delta_\ell}, \quad \text{for all } \quad \tau \in \cT^h,  
\end{equation}
for some $\Cchiell$ independent of the element  $\tau$. We let 
\beq\label{eq:Cchidef}
\Cchi:= \max_{\ell} \Cchiell\quad\text{ so that }\quad \N{\nabla\chi_\ell}_{L^\infty(\Omega)}\leq 
\frac{\Cchi}{\delta} \quad\tfa \ell=1,\ldots, N.
\eeq
A partition of unity
satisfying this condition  is explicitly constructed in \cite[\S3.2]{ToWi:05}.

Let $\Pi^h: C(\overline{\Omega}) \mapsto \mathcal{V}^h$   denote the nodal interpolation operator, set    $\Pi^h_\ell:=\Pi^h\circ \chi_\ell$  and observe that, if $w_{h,\ell} \in \cV^h_\ell$ with nodal values $\bW$, then 
\beq\label{eq:Pidef}
\Pi^h_\ell w_{h,\ell} := \Pi^h\big(\chi_\ell w_{h,\ell}\big) = \sum_{p \in \cI^h} \big( \MR^\top_\ell \bW\big)_p \phi_p,
\eeq
where $\MR_\ell$ is defined by \eqref{eq:explicit}, and thus $\Pi^h_\ell$ defines a prolongation from  $\cV_\ell^h$ to $\cV^h$. 
$\Pi^h_\ell$ can also be viewed as a restriction operator mapping $C(\overline{\Omega})$ to $\cV_\ell^h$, and is used is this way in Lemma \ref{lem:localBVP}.

The following lemma is proved in \cite[Lemma 3.3]{GrSpZo:18}.  
\begin{lemma}[Error in interpolation of $\chi_\ell w_h$]\label{lem:motivated} 
There exist $\Cintell = \Cintell(p)$, $\ell=1,\ldots,N$ such that
\beq
\|( {\rm I} - \Pi^h) (\chi_l v_h)\|_{1, k, \Om_l}  \leq \Cintell
 \left(1 + kh_\ell \right)\,  \left(\frac{\hmaxell}{\delta_\ell}\right)  \Vert v_h\Vert _{H^1(\Om_l)}   
\quad\tfa v_h \in \cV^h_\ell,
 \label{eq:Cintell}
\eeq
where $\hmaxell \ :=  \ \max_{\tau \subset  \overline{\Omega_\ell}} h_\tau$.
\end{lemma} 

Let 
$ \Cint := \max_{\ell= 1,\ldots,N}\Cintell$. 
Since (by Assumption \ref{ass:2}), $hk\leq 1$, we mostly use \eqref{eq:Cintell} in the form 
\begin{align}
\|( \chi_\ell  - \Pi^h_\ell )  v_h\|_{1, k, \Om_l}  \leq 2\Cint\left(\frac{h}{\delta}\right)  \Vert v_h\Vert _{H^1(\Om_l)}   
\quad\tfa v_h \in \cV^h_\ell \quad\text{ and for all } \ell= 1,\ldots,N , 
\label{eq:Cintell2}
\end{align} 
where we have used the estimate  
  ${h_\ell}/{\delta_\ell}  \ \leq \ {h}/{\delta}$,  
where $h$ is the global maximal mesh diameter and $\delta$ is the global minimum overlap parameter.    \begin{remark}[$h/\delta$ is ``higher order'' than $(k\delta)^{-1}$]\label{rem:hdelta} 
Some of our later results require the additional  assumption that   $kh \rightarrow 0$ and $k \delta \rightarrow \infty $ as $k \rightarrow \infty$ (see Assumption \ref{ass:3} below).   This assumption implies 
$h/\delta = (kh)/(k\delta)$ is ``higher order'', (i.e. approaches zero  more quickly)  than $(k\delta)^{-1}$,  as $k \rightarrow \infty$. 
\end{remark}

 The following bounds are proved using properties of the overlapping domain  decomposition.

\begin{lemma}[Bounds on norms involving $\chi_\ell$]
\begin{equation}\label{eq:pou-bound}
\|\chi_\ell v \|_{1,k,\Omega_\ell} \ \leq \  \sqrt{2} \left(1 + \frac{C_\chi}{k\delta}\right)   \|v\|_{1,k,\Omega_\ell} 
\quad \text{for all }v\in H^1(\Omega_\ell).
\end{equation}
\begin{equation}\label{eq:sum-local-norms}
  \sum_\ell \|\chi_\ell v \|_{1,k,\Omega_\ell}^2 \ \ge \ \frac{1}{\Lambda}\|v\|_{1,k}^2 - \Lambda \frac{C_\chi}{k\delta}
  \, \left(1 + \frac{C_\chi}{k \delta}\right) \, \|v\|_{1,k}^2 \quad \text{for all }v\in H^1(\Omega).
\end{equation}
\end{lemma}
\bpf[References for the proof]
Both \eqref{eq:pou-bound} and \eqref{eq:sum-local-norms} are proved in \cite[Lemma 3.1]{GrSpZo:18}.
The constants on the right-hand sides
are not given explicitly there,  
but  can be obtained by  examining the proof.
\epf

\begin{corollary}[Boundedness of $\Pi_\ell^h$]\label{cor:Piell}
Let Assumption \ref{ass:2} hold.  Then, for all $v_h \in \cV^h$,
\beq\label{eq:CPi}
\|\Pi^h_\ell v_h \|_{1,k,\Omega_\ell} \leq \CPi(p) \N{v_h}_{1,k,\Omega_\ell},
\eeq
where
\beq\label{eq:CPidef}
\CPi(p) := 2 \Cint(p) \left(\frac{h}{\delta}\right)  +\sqrt{ 2}\left(1  + \frac{ \Cchi}{k\delta} \right) = \sqrt{2} + \frac{1}{k \delta} \Big(2 kh \Cint(p)  + \sqrt{2} \Cchi\Big) .
\eeq
\end{corollary}

\bpf
By the triangle inequality $\|\Pi^h_\ell v_h\|_{1,k,\Omega_\ell} \leq \|(I-\Pi^h)\chi_\ell v_h\|_{1,k,\Omega_\ell} + \|\chi_\ell v_h\|_{1,k,\Omega_\ell}$, and the result follows from  \eqref{eq:Cintell2} and \eqref{eq:pou-bound}.
\epf
\subsection{The local projection operators $Q_{ \abs,  \ell}^h$}
To  analyse the preconditioner    \eqref{eq:ASpc},  we  define the projections 
$Q_{ \abs,  \ell}^h  : H^1(\Omega) \rightarrow \tVhl$, by requiring that, given $v\in H^1(\Omega)$,
$Q_{\abs,  \ell}^h  v \in \tVhl$ 
satisfies the equation 
\begin{equation} 
a_{\abs, \ell}(Q_{\abs,\ell}^h v , w_{h,\ell}) \ = \ a_{\abs}(v, \Pi^h(\chi_\ell w_{h,\ell}))
\quad \text{for all} \quad w_{h,\ell}  \in \tVhl.
\label{eq:impproj} 
\end{equation}
For $|\abs|>0$, $Q_{\abs,\ell}$ is well-defined by Lemma \ref{lem:discrete1}. For $\abs=0$, $Q_{\abs,\ell}$ is well-defined by Lemma \ref{lem:discretebound2} when $h$, $p$, and $k$ satisfy the condition \eqref{eq:meshthresh1}  (see also Remark \ref{rem:meshthresh1}).

To combine the actions of these  local projections additively, we define the global projection by
\begin{align} 
 Q_{\abs}^h \  :=\  \sum _{\ell=1}^N \Pi^h (\chi_\ell Q_{\abs,\ell}^h)  = \sum _{\ell=1}^N\Pi^h_\ell Q_{\abs,\ell}^h, \label{eq:global} 
\end{align} 
where again, each term in the sum can be interpreted as an element of $H^1(\Omega)$.  
The following result, proved in \cite[Theorem 2.10]{GrSpZo:18}, shows that  the matrix representation of  
$Q^h_{\abs}$ restricted to  $\cV^h$ coincides with the preconditioned matrix  $\MB_\abs^{-1}\MA_\abs$.

\begin{lemma}[From projection operators to matrices] \label{thm:matrixprec}
If $v_h, w_h \in \cV^h$,  with nodal values given in the vectors $\bV, \bW$, then
\beqs
( v_h , Q_{\abs}^h w_h)_{1,k} \ = \  \langle \bV, \MB_\abs^{-1} \MA_{\abs} \bW\rangle_{D_k}.
\eeqs
\end{lemma}

We now combine  the results in this section with  the results in \S\ref{sec:solution_operators} to prove bounds on how close  $Q_{\eps,\ell}^h$ is to  $\Pi_\ell^h$, with the end result being Corollary \ref{cor:local-approx}.  This estimate of $Q_{\eps,\ell}
^h v_h - \Pi_\ell^h v_h$ is crucial in proving our main results in \S \ref{sec:main_results}. 

\ble[Discrete BVP on  $\Omega_\ell$ satisfied by $Q_{\abs,\ell}^h v_h - \Pi^h_\ell v_h$]\label{lem:localBVP}
Given $v_h \in \cV^h$, 
\beq
a_{\abs, \ell}(Q_{\abs,\ell}^h v_h -\Pi^h_\ell v_h, w_{h,\ell})  = F_\ell(w_{h,\ell}) \quad\tfa  w_{h,\ell}  \in \tVhl,
\label{eq:BVPerror}
\eeq
where
\beq\label{eq:Fell}
F_\ell(w_{h,\ell}) :=   a_{\abs,\ell}((\rI - \Pi^h)(\chi_\ell v_h), w_{h,\ell}) - a_{\abs,\ell}(v_h, (\rI - \Pi^h)(\chi_\ell w_{h,\ell})) +  b_\ell(v_h, w_{h,\ell}),
\eeq
where
\beq
b_\ell(v,w) :=\int_{\Omega_\ell} (A\nabla \chi_\ell)\cdot ( \overline{w} \nabla v - v \nabla \overline{w} ).
\label{eq:defb} 
\eeq
\ele

\bpf
When $w_{h,\ell} \in \cV^h_\ell$,  $\Pi^h(\chi_\ell w_{h,\ell})$ is supported on $\Omega_\ell$
and vanishes on $\partial \Omega_\ell\backslash \Gamma_I$.
  Therefore, by \eqref{eq:impproj}, for all  $ w_{h,\ell}  \in \tVhl$ and $v_h \in \cV^h$, 
\begin{align*} 
a_{\abs, \ell}(Q_{\abs,\ell}^h v_h , w_{h,\ell}) \ = \ a_{\abs,\ell}(v_h, \Pi^h(\chi_\ell w_{h,\ell}))
\end{align*}
 and hence 
\begin{equation*}
a_{\abs, \ell}(Q_{\abs,\ell}^h v_h - \Pi^h(\chi_\ell v_h), w_{h,\ell}) =a_{\abs,\ell}(v_h, \Pi^h(\chi_\ell w_{h,\ell})) - a_{\abs,\ell}(\Pi^h(\chi_\ell v_h), w_{h,\ell}). 
\end{equation*}
The result then follows by observing that 
\begin{align*}
 a_{\abs,\ell}(v_h, \Pi^h(\chi_\ell w_{h,\ell}))  - a_{\abs,\ell}(\Pi^h(\chi_\ell v_h), w_{h,\ell}) 
 & =   a_{\abs,\ell}((I - \Pi^h)(\chi_\ell v_h), w_{h,\ell}) - a_{\abs,\ell}(v_h, (I - \Pi^h)(\chi_\ell w_{h,\ell})) \nonumber \\
 &\hspace{4cm}+ a_{\abs,\ell}(v_h, \chi_\ell w_{h,\ell}) - a_{\abs,\ell}(\chi_\ell v_h, w_{h,\ell}) 
\end{align*}
and, using the symmetry of the matrix $A$,  
\begin{align*}
a_{\abs,\ell}(v_h, \chi_\ell w_{h,\ell}) - a_{\abs,\ell}(\chi_\ell v_h, w_{h,\ell}) &= \big(A\nabla v_h, \nabla(\chi_\ell
w_{h,\ell})\big)_{\Omega_\ell}  -
\big(A \nabla(\chi_\ell v_h),\nabla  w_{h,\ell}\big)_{\Omega_\ell} \\
&=
\int_{\Omega_\ell} (A\nabla \chi_\ell) \cdot ( \overline{w_{h,\ell}} \nabla v_h - v_h \nabla \overline{w_{h,\ell}} ).
\end{align*}
\epf  

\begin{lemma}[Bound on the right-hand side of   \eqref{eq:BVPerror} ] \label{lem:Fbound}

\

\noi (i) For all $v,w \in H^1(\Omega_\ell)$, 
\beqs 
{\vert  b_\ell (v,w) \vert }    \leq  A_{\max,\ell} \, C_{\chi,\ell} \, (k \delta_\ell )^{-1}  \, {\Vert v \Vert_{1, k, \Omega_\ell}  
\Vert w \Vert_{1, k, \Omega_\ell} }.
\eeqs
(ii) 
For all  $v_h \in \cV^h, w_{h,\ell} \in \tVhl$, 
\beqs
\max\Big\{{\vert a_{\abs,\ell} (v_h, (\rI - \Pi^h) (\chi_\ell w_{h,\ell}))\vert ,  \vert a_{\abs,\ell} ((\rI - \Pi^h) (\chi_\ell v_h), w_{h,\ell})\vert}\Big\} \leq 2\Ccontell\Cintell
\frac {h_\ell}{ \delta_\ell}   \Vert v_h \Vert_{1,k, \Omega_\ell} \Vert w_{h,\ell} \Vert_{1,k, \Omega_\ell}.
\eeqs
\end{lemma} 
(iii) As a corollary of (i) and (ii), with $F_\ell$ defined by \eqref{eq:Fell},
\beq\label{eq:boundonF}
\max_{w_{h,\ell} \in  \cV^h_{\ell}}  
\left( \frac{\vert F_\ell(w_{h,\ell}) \vert}{ \Vert w_{ h,\ell} \Vert_{1,k,\Omega_\ell} } \right)\leq
 \left(\frac{A_{\max,\ell}\, C_{\chi,\ell}}{k\delta_\ell} + 4\Ccontell\Cintell\frac {h_\ell}{ \delta_\ell}\right)
\N{ v_{ h}}_{1,k,\Omega_\ell}.
\eeq
\begin{proof}
The result (i) follows from using the definition of $A_{\max,\ell}$ in \eqref{eq:Alimitsell}, then applying  the Cauchy-Schwarz inequality to \eqref{eq:defb}, using the bound \eqref{eq:derivpou}, 
and then applying the Cauchy-Schwarz inequality with respect to the Euclidean inner product 
in $\mathbb{R}^2$.
The result (ii) follows from the continuity of $a_{\eps,\ell}$, the definition of $\Ccontell$ \eqref{eq:Ccontell}, and the bound
\eqref{eq:Cintell}. 
\end{proof}

 \
  
Combining Lemmas \ref{lem:discrete1} and \ref{lem:discretebound2} with the bound \eqref{eq:boundonF}, we obtain the following two bounds on   
$\|Q_{\eps,\ell}^h v_h -\Pi_\ell^h v_h\|_{1,k,\Omega_\ell}$ in terms of $\Csolell$ (the bound on the local continuous solution operator). The  first estimate requires  $|\eps|>0$, but the second holds for all  $|\eps|\geq 0$.
  
\begin{corollary}[Approximation of the local problems in terms of $\Csolell$]\label{cor:local-approx}
  (i) If $|\eps|>0$, then for all $h$ and $p$ and $v_h\in \cV^h$,
\begin{align}
\frac{\|Q_{\eps,\ell}^h v_h -\Pi_\ell^h v_h\|_{1,k,\Omega_\ell} 
}{
   \|v_h\|_{1,k,\Omega_\ell}
   }
 &\leq
4 \frac{k^2}{|\eps|}\Big( \min\big\{ A_{\min, \ell}, n_{\min, \ell}\big\}\Big)^{-1}
\left(\frac{A_{\max,\ell}\, C_{\chi,\ell}}{k\delta_\ell} + 4\Ccontell\Cintell\frac {h_\ell}{ \delta_\ell}\right)
.\label{eq:local-approx1}
\end{align}
(ii) If $|\eps|\geq 0$ and $h$ and $p$ are such that $\cV^h_\ell$ satisfies \eqref{eq:meshthresh1} (see Remark \ref{rem:meshthresh1} for sufficient conditions for this), then, for all $v_h \in \cV^h$,
\begin{align}\nonumber
\frac{
\|Q_{\eps,\ell}^h v_h -\Pi_\ell^h v_h\|_{1,k,\Omega_\ell}
}{
   \|v_h\|_{1,k,\Omega_\ell}
   }
  &\leq
\left(
\frac{4\big( 9+8  \Csolell H n_{\max,\ell}|\eps-\sign(\eps)k^2| k^{-1}\big{)}}{  \min\big\{A_{\min, \ell}, n_{\min, \ell}\big\}}
\right)\\
&\hspace{2cm}
\cdot \left(\frac{A_{\max,\ell} \, C_{\chi,\ell}}{k\delta_\ell} + 4\Ccontell\Cintell\frac {h_\ell}{ \delta_\ell}\right).
\label{eq:local-approx2}
\end{align}
\end{corollary}
\section{The main theoretical results on the convergence of GMRES}
  
\label{sec:main_results}

\subsection{Bounds on the norm and field of values of $\MB_\eps^{-1}\MA_\eps$}  
  
The main purpose of this section is to obtain both (i) an upper bound on the norm of the preconditioned matrix $\MB_\eps^{-1}\MA_\eps$ and (ii) a lower bound on the distance of its field of values from the origin.
 By  Lemma    \ref{thm:matrixprec},  this is equivalent to proving analogous properties
 of the projection operator $Q_\eps^h$.  Our first result,   Theorem \ref{thm:main1},  sets out a criterion on the local projection operators $Q_{\eps,\ell}^h$ that ensures good bounds on the norm and field of values of $Q_\eps^h$. 
We then investigate conditions under which this  criterion  is satisfied;   
these conditions are explicit in the polynomial degree $p$ of the finite elements and in the coefficients $A$ and $n$.

\begin{theorem}\label{thm:main1}
Assume $kh\leq 1$. Suppose
\beq\label{eq:sigma}
\N{Q_{\abs,\ell}^h v_h - \Pi^h_\ell  v_h}_{1,k,\Omega_\ell} \leq \sigma \N{v_h}_{1,k,\Omega_\ell}, \quad
\text{for all }   v_h\in \cV^h \text{ and for all }\ell = 1, \ldots, N. 
\eeq
Then,
\beq\label{eq:upperbound}
\max_{v_h \in \cV^h} \frac
 {\N{Q_{\abs}^h v_h}_{1,k}}{{\N{v_h}_{1,k}}} \ 
\leq \
\Lambda \CPi(p)( \CPi(p) + \sigma),
\eeq
and
\beq\label{eq:lowerbound}
\min_{v_h \in \cV^h}\, \frac{
\big|(v_h,Q_\abs^h v_h)_{1,k}\big|
}{
\N{v_h}^2_{1,k}
}
\ \geq\ 
\left(
\frac{1}{\Lambda} -  \sqrt{2} \sigma \Lambda\right)  \ - \  R  \,   
\eeq
where 
  \begin{align}
R \ =\ 
\frac{\Lambda \Cchi}{k \delta} \left(1 + (1+\sqrt{2}) \sigma + \CPi(p) +  \frac{\Cchi}{k \delta}\right)
\, + \,
2 \Lambda \frac{ kh  \Cint(p)}{k\delta} \left(\sqrt{2} + \sigma  + \CPi(p) +  
    \frac{\sqrt{2} \Cchi}{k \delta} \right),
    \label{eq:R} \end{align}
and $\CPi(p)$ is given in \eqref{eq:CPidef}. 
\end{theorem}
\bpf
Throughout the proof, we use the notation
\beqs
z_l \, :=\,  Q_{\abs,\ell}^h v_h - \Pi^h_\ell v_h, \quad \text{so that, by \eqref{eq:sigma},} \quad
\Vert z_\ell \Vert_{1,k,\Omega_\ell} \leq \sigma \Vert v_h \Vert_{1,k, \Omega_\ell} .
\eeqs 
To obtain the upper bound \eqref{eq:upperbound}, we  use  the triangle inequality, then
\eqref{eq:CPi} and \eqref{eq:sigma},   to obtain 
\beq\label{eq:QE1}
\Vert Q_{\eps,\ell}^h v_h \Vert_{1,k,\Omega_\ell} \ \leq \ \Vert \Pi^h_\ell v_h \Vert_{1,k,\Omega_\ell} \ + \ \Vert z_l  \Vert_{1,k,\Omega_\ell}  \leq  \left(\CPi + \sigma   \right) \Vert  v_h \Vert_{1,k,\Omega_\ell} . 
\eeq
Then, using \eqref{eq:global}, Lemma \ref{lem:norm_of_sum},   \eqref{eq:CPi} and \eqref{eq:QE1}, 
\begin{align*}
 \Vert Q_{\eps}^h v_h \Vert_{1,k}^2  &  \ = \
\left\Vert \sum_\ell \Pi^h\left(\chi_\ell Q_{\eps, \ell}^h v_h\right) \right \Vert_{1,k}^2  \ \leq \  \Lambda \sum_\ell \left\Vert  \Pi^h\left(\chi_\ell Q_{\eps, \ell}^h v_h\right) \right \Vert_{1,k, \Omega_\ell}^2 \\
                                     & \leq \Lambda (\CPi)^2  \sum_\ell \left\Vert   Q_{\eps, \ell}^h v_h  \right \Vert_{1,k, \Omega_\ell}^2  
                                        \leq  \Lambda (\CPi)^2 (\CPi + \sigma)^2  \sum_\ell \left\Vert   v_h  \right \Vert_{1,k, \Omega_\ell}^2 ,
\end{align*} 
and \eqref{eq:upperbound}  then follows on using \eqref{eq:finoverEuan1}. 
 
To obtain the lower bound \eqref{eq:lowerbound}, we first split the left-hand side into several terms and then estimate it term by term:
\begin{align}\nonumber
(v_h, Q_{\eps}^hv_h)_{1,k}& = \sum_{\ell} \left(v_h, \Pi^h_\ell Q_{\eps,\ell}^h v_h \right)_{1,k,\Omega_\ell}\nonumber \\ 
&= \sum_{\ell}\Big[ \left(v_h, (\Pi^h_\ell -\chi_\ell)  Q_{\eps,\ell}^h v_h) \right)_{1,k,\Omega_\ell} 
+  \left(v_h,  \chi_\ell Q_{\eps,\ell}^h v_h \right)_{1,k,\Omega_\ell}- \left(\chi_\ell  v_h,  Q_{\eps,\ell}^h v_h \right)_{1,k,\Omega_\ell}   \nonumber \\
&~~~~~+ \left(\chi_\ell  v_h,  (Q_{\eps,\ell}^h   - \Pi^h_\ell) v_h\right)_{1,k,\Omega_\ell} 
                                                                                                                                                   + \left(\chi_\ell  v_h,  (  \Pi^h_\ell -\chi_\ell) v_h\right)_{1,k,\Omega_\ell} + \|\chi_\ell v_h\|_{1,k,\Omega_\ell}^2
                                                                              \Big]\label{eq:expansion} \\
                          &= \sum_{\ell} \left[\|\chi_\ell v_h\|_{1,k,\Omega_\ell}^2 +
                            \left(\chi_\ell  v_h,  (Q_{\eps,\ell}^h   - \Pi^h_\ell) v_h\right)_{1,k,\Omega_\ell}\right]
  \nonumber \\    & + \sum_\ell \left[\left(v_h,  (  \Pi^h_\ell -\chi_\ell) Q_{\eps,\ell}^h v_h\right)_{1,k,\Omega_\ell} + \left(  \chi_\ell v_h,  (\Pi_\ell^h - \chi_\ell ) v_h\right)_{1,k,\Omega_\ell} \right] \nonumber \\
                          & + \sum_\ell \left[ (v_h, \chi_\ell Q_{\eps,\ell}^h v_h)_{1,k,\Omega_\ell} -
                            (\chi_\ell v_h, Q_{\eps,\ell}^h v_h)_{1,k,\Omega_\ell} \right] \ = : \ T1 + T2 + T3 . \label{eq:rearranged}
\end{align}
(Note that \eqref{eq:rearranged} is just a simple rearrangement of \eqref{eq:expansion}.)

Consider first $T1$. For the second term in its summand, we have, using \eqref{eq:pou-bound} and \eqref{eq:sigma}, 
\begin{align*}
  \left| \left(\chi_\ell  v_h,  (Q_{\eps,\ell}^h   - \Pi^h_\ell) v_h\right)_{1,k,\Omega_\ell}  \right| &
       \leq  \sqrt{2}  \left(1 + \frac{\Cchi}{k\delta}\right)  \sigma \|v_h\|_{1,k,\Omega_\ell}^2.
\end{align*}
Combining this with \eqref{eq:sum-local-norms} and \eqref{eq:finoverEuan1}, we obtain
\begin{align} \label{eq:T1} T1 \ \geq \ \left(\frac{1}{\Lambda} - \sqrt{2} \sigma \Lambda \right) \Vert v_h \Vert_{1,k}^2
- \Lambda \frac{\Cchi}{k \delta} \left( 1 + \sqrt{2} \sigma  + \frac{\Cchi}{k\delta}\right) \Vert v_h \Vert_{1,k}^2 . \end{align} 

Then  for the summand in $T2$,  using  Lemma \ref{lem:motivated} and \eqref{eq:QE1},  we have the estimates
\begin{align}
\left| \left(v_h, (\Pi^h_\ell -\chi_\ell)  Q_{\eps,\ell}^h v_h) \right)_{1,k,\Omega_\ell}\right| \ & \leq  \ 2\Cint \frac{h}{\delta}    (\CPi+ \sigma)  \|v_h\|_{1,k,\Omega_\ell}^2.
\nonumber \\
\left| \left(\chi_\ell v_h, (\Pi^h_\ell -\chi_\ell) v_h) \right)_{1,k,\Omega_\ell}\right| \ & \leq  \ 2\Cint \frac{h}{\delta}  \sqrt{2} \left( 1 + \frac{\Cchi}{k \delta}\right)   \|v_h\|_{1,k,\Omega_\ell}^2.
\nonumber 
\end{align}
Combining these with \eqref{eq:finoverEuan1}, we obtain
\begin{align} \label{eq:T2}
T2 \ \leq \ 2 \Lambda \Cint \frac{h}{\delta} \left( \CPi+ \sigma + \sqrt{2}
  \left(1 + \frac{\Cchi}{k\delta}\right)\right)\Vert v_h \Vert_{1,k}^2 . \end{align}

For the summand in $T3$, we have by \eqref{eq:Cchidef} and \eqref{eq:QE1},
\begin{align}\nonumber
 &  \left| \left(v_h,  \chi_\ell Q_{\eps,\ell}^h v_h \right)_{1,k,\Omega_\ell}- \left(\chi_\ell  v_h,  Q_{\eps,\ell}^h v_h \right)_{1,k,\Omega_\ell}  \right|
   = \left|  \int_{\Omega_\ell}~ \nabla \chi_\ell \cdot\left(\overline{Q_{\eps,\ell}^h v_h} \nabla v_h - v_h
   \nabla(\overline{Q_{\eps,\ell}^h v_h}) \right) \right|\\ \nonumber
 & \mbox{\hspace{1in} } \le \frac{\Cchi}{\delta} \int_{\Omega_\ell}~ \left|\overline{Q_{\eps,\ell}^h v_h} \nabla v_h - v_h\nabla(\overline{Q_{\eps,\ell}^h v_h}) \right|\ \le\  \frac{\Cchi}{k\delta}   \| Q_{\eps,\ell}^h v_h \|_{1,k,\Omega_\ell} \|v_h\|_{1,k,\Omega_\ell}\\ \nonumber
  &\mbox{\hspace{1in} } \le  \frac{\Cchi}{k\delta} (\CPi+\sigma)  \|v_h\|_{1,k,\Omega_\ell}^2.
\end{align}
Hence
\begin{align} \label{eq:T3}
T3 \ \le \  \Lambda \frac{\Cchi}{k\delta}   (\CPi+\sigma)  \|v_h\|_{1,k}^2
\end{align}
The result is obtained by combining the estimates \eqref{eq:T1},  \eqref{eq:T2}, and \eqref{eq:T3} within the triangle
inequality to obtain \eqref{eq:lowerbound},  with $R$ given by \eqref{eq:R}.  
\epf

We now introduce a stronger condition than Assumption \ref{ass:2}.
\begin{assumption}\label{ass:3}  We assume that $h = h(k)$ and $\delta = \delta(k) $ satisfy: 
  \begin{align*}  (i) \ kh \rightarrow 0 \quad \text{and} \quad (ii) \  k \delta \rightarrow \infty . 
  \end{align*}
\end{assumption}
The  assumption \emph{(i)} is a  very natural strengthening of Assumption \ref{ass:2}:
to avoid the pollution effect for the fine-grid problem we require that $k(kh)^{2p}$ is small enough. 
 (see Remarks \ref{rem:meshthresh1} and \ref{rem:meshthresh2}).   
Thus it is natural  to assume that $kh \rightarrow 0$ as $k \rightarrow \infty$.
The assumption     \emph{(ii)} essentially says that the overlap of subdomains should contain a (perhaps slowly) growing
number of wavelengths. 
Whilst minimal overlap has also been
used 
successfully on physically-relevant benchmark problems
(see, e.g., \cite{BoDoGrSpTo:17c}),  
there is numerical evidence that a bigger overlap is
necessary in order to obtain
iteration counts that are independent of $k$. 

\begin{corollary}\label{cor:ver16}
  Let Assumptions \ref{ass:2} and \ref{ass:3} hold, then for
  $k$ sufficiently large,
\begin{align*}
  \max_{v_h \in \cV^h} \frac
 {\N{Q_{\abs}^h v_h}_{1,k}}{{\N{v_h}_{1,k}}} \ 
\leq 
 (8 + 2 \sqrt{2} \sigma)    \Lambda ,
  \quad \text{and} \quad 
\min_{v_h \in \cV^h}\, \frac{
\big|(v_h,Q_\abs^h v_h)_{1,k}\big|
}{
\N{v_h}^2_{1,k}
}
\ \geq\  
 \left(\frac{1}{2\Lambda} - 2 \sigma \Lambda \right).    
\end{align*} 
  
\end{corollary}
\begin{proof}    Using Assumption \ref{ass:3}, 
   together  with \eqref{eq:CPidef} and  \eqref{eq:R},    we have, for $k$ suffciently large, that $\CPi(p) \leq 2 \sqrt{2}$ and
  $\vert R \vert \leq (2 - \sqrt{2}) \Lambda \sigma + 1/(2 \Lambda)$.
 Combining these inequalities with \eqref{eq:upperbound}, \eqref{eq:lowerbound} and \eqref{eq:R}, we obtain the result.  
\end{proof}

Motivated by this corollary, we now obtain an upper 
bound on $\sigma$ and then
investigate under what conditions this can be made  small. This
allows us to 
  obtain 
     a positive lower bound  on the distance of the the field of values from the origin.
We focus on the case $\vert\eps\vert>0$, with a discussion on the case $\vert \eps \vert \geq  0$ given in the  remark following the next lemma. 


\begin{lemma}[Bound on $\sigma$ for $|\eps|>0$]\label{lem:sigmabound}  If 
  Assumptions \ref{ass:2} and \ref{ass:3} hold, 
  $\vert \eps \vert > 0$, 
  and $k$ is sufficiently large, then 
  \eqref{eq:sigma} holds
  with $\sigma$ satisfying the bound    
\begin{align} \label{eq:contrast} \sigma
\leq 
5 \Cchi \frac{k}{|\eps|\delta}\,  \cClocal(A,n) , \quad \text{where} \quad \cClocal(A,n) : = \ \max_\ell \frac{A_{\max,\ell}} { \min\big\{ A_{\min,\ell}, n_{\min,\ell}\big\}}. 
\end{align} 
\end{lemma}
\begin{proof}
  From \eqref{eq:local-approx1}, we have 
\begin{align*} \sigma \ & \leq\ 
4 \frac{k^2}{|\eps|}\Big( \min\big\{ A_{\min, \ell}, n_{\min, \ell}\big\}\Big)^{-1}
\left(\frac{A_{\max,\ell}\, C_{\chi,\ell}}{k\delta_\ell} + 4\Ccontell\Cintell(p)\frac {h_\ell}{ \delta_\ell}\right)\\  
& \leq 
4 \frac{k}{|\eps|\delta}\, \max_\ell
\left(\frac{A_{\max,\ell}C_{\chi,\ell} + 4\Ccont\Cint(p) k h }
{\min\big\{ A_{\min,\ell}, n_{\min,\ell}\big\}} \right), 
\end{align*}
and the result follows on using Assumption \ref{ass:3} (i).
\end{proof}

\begin{remark}\label{rem:contrast}
  \
\noi (i) One can see from the final step of the proof of Lemma \ref{lem:sigmabound}  that the size of  $k$ needed to obtain the estimate \eqref{eq:contrast} depends on both  $p$ and the coefficients $A,n$. \\
    \noi(ii) To get an estimate for $\sigma$ that  is valid for all $\vert \eps \vert \geq 0$, one can repeat the argument in  Lemma \ref{lem:sigmabound},     but using   \eqref{eq:local-approx2} instead of \eqref{eq:local-approx1} and needing also Assumption \ref{ass:nontrapping}.
    The result (for $k$ sufficiently large) is  an  estimate of the form
    $\sigma \leq \cClocal^*(A,n) {H}/{\delta}, $ which is independent of $k$, with $\cClocal^*(A,n)$ a slightly different expression to that in \eqref{eq:contrast},  but still depending only on local variation. 
\ere


          

Now,  combining Corollary \ref{cor:ver16} with Lemma \ref{lem:sigmabound}, we obtain
a condition on $\eps$ that guarantees a positive  lower bound for the field of values: 
\begin{corollary}  \label{cor:proj}
  Under Assumptions  \ref{ass:2} and \ref{ass:3},  
  suppose  $\eps = \eps(k) $ is chosen so that 
  \begin{align} \label{eq:condn1}
    \frac{k}{\vert \eps \vert \delta} \  \leq\  \frac{1}{40\Cchi \cClocal(A,n) \Lambda^2}  \quad \text{for} \ k \ \text{sufficiently large}.
                                  \end{align}
 Then for $k $ sufficiently large, 
\begin{align*}
  \max_{v_h \in \cV^h} \frac
 {\N{Q_{\abs}^h v_h}_{1,k}}{{\N{v_h}_{1,k}}} \ 
\leq 
 9  \Lambda ,
\quad \text{and} \quad 
\min_{v_h \in \cV^h}\, \frac{
\big|(v_h,Q_\abs^h v_h)_{1,k}\big|
}{
\N{v_h}^2_{1,k}
}
\ \geq\  
\frac{1}{4 \Lambda}.    
\end{align*} 
\end{corollary}

\bigskip


  Using Lemma \ref{thm:matrixprec}  we can turn Corollary \ref{cor:proj} into a statement about preconditioned  matrices.

\begin{corollary}  \label{cor:matrices}
Let Assumptions  \ref{ass:2} and \ref{ass:3} and condition  
\eqref{eq:condn1} hold.   
Then,  for $k $ sufficiently large,
\begin{align}\label{eq:finalmatrix}
 \N{\MB_{\eps}^{-1} \MA_{\eps}}_{\MD_k} \ 
\leq  9  \Lambda 
\quad \text{and} \quad 
\min_{\bV \in \mathbb{C}^n}\, \frac{
\big| \left\langle \bV, \MB_\eps^{-1}\MA_\eps\bV \right\rangle_{\MD_k} \big|
}{
\N{\bV}^2_{\MD_k}
}
\ \geq\ 
  \frac{1}{4 \Lambda}. 
\end{align} 
\end{corollary}

In the next corollary we deduce 
bounds on the number of GMRES iterations. 
In the previous papers \cite{GrSpZo:18} and \cite{GrSpVa:17} we  proved analogous results about GMRES in the $D_k$ inner product applied to the
Helmholtz problem with constant coefficients. 
Part (i) of the following result generalises this to 
the variable-coeffiicent case. Part (ii) 
uses a novel argument to prove a corresponding result about GMRES in the Euclidean inner product (i.e. ``standard GMRES").

\begin{corollary}[Bounds on the number of GMRES iterations]
\label{cor:GMRES} 
Under the assumptions of Corollary \ref{cor:matrices}:

\noi (i) 
If GMRES is applied to the linear system \eqref{eq:discrete}, with $\MB_\eps^{-1}$ as a left preconditioner in the inner product induced by $\MD_k$, then, for $k$ sufficiently large, the number of iterations needed to achieve a prescribed relative residual
is independent of $k, \eps, H, \delta, p, A$, and $n$.

\noi (ii) 
If the fine mesh sequence $\cT_h$ is quasiuniform and 
$k(kh)^{2p}\leq 1$  (see Remark \ref{rem:meshthresh2}), then 
GMRES applied in the Euclidean inner product with the same initial residual as in Part (i) 
takes at most extra $(\log_2(k))/(pC)$ iterations to ensure the same relative 
residual as if GMRES were applied in the $\MD_k$ weighted inner product, where $C$ is a constant independent of all parameters.
\end{corollary}
\begin{proof}
(i) Let $\bR_{\MD_k}^n$  denote the $n$-th residual of the $\MB_{\eps}^{-1}$-preconditioned GMRES applied to $\MA_\eps$ in the $\MD_k$ weighted inner product. Using the Elman estimate \cite{El:82, EiElSc:83} and Corollary \ref{cor:matrices}, we have
\beq\label{eq:123}
\frac{ \N{\bR_{\MD_k}^n}_{\MD_k}}{ \N{\bR_{\MD_k}^0}_{\MD_k} } \le (1-c)^n,
\eeq
where $c<1$ depends only on the constants on the right-hand side of the bounds in \eqref{eq:finalmatrix};
the result then follows from the fact that these constants are independent of all parameters except for $\Lambda$. 

(ii) 
Let $\bR^n$ denote the $n$-th residual of the $\MB_{\eps}^{-1}$-preconditioned GMRES applied to $\MA_\eps$ in the Euclidean inner product.
Because of the optimisation property of GMRES residuals \cite{guttel2014some}, we have
\begin{equation}\label{eq:weightedGmres}
\| \bR^n \| \le \| \bR_{\MD_k}^n \|\quad \tfa n.
\end{equation}
By the inverse estimate for quasiuniform meshes (see, e.g., \cite[ Theorem 4.5.11 and Remark 4.5.20]{BrSc:08}), and the fact that $h^{-1}\gg k$, we have that, for any $\bV \in \mathbb{C}^n$, with $v_h\in \cV^h $ denoting the corresponding finite element function,
\begin{equation}\label{eq:weighted-l2-norms}
  kh^{d/2} \|\bV\| \lesssim k \Vert v_h \Vert_\Omega \lesssim \|\bV\|_{\MD_k} \lesssim  \sqrt{h^{-2} + k^2} \Vert v_h \Vert_\Omega \lesssim 
  h^{-1}\|v_h\|_\Omega  \lesssim h^{d/2-1}\|\bV\|, 
\end{equation}
where $\lesssim$ denotes $\leq$ with  hidden constant independent of all parameters of interest.

Suppose GMRES in the  $\MD_k$ weighted inner product 
satisfies the bound \eqref{eq:123} for all $n\in \mathbb{N}$. Suppose
GMRES in the Euclidean inner product performs $n+m$ iterations with the same initial residual  $ \bR^0=\bR_{\MD_k}^0 $. Then, 
by \eqref{eq:weightedGmres}, \eqref{eq:weighted-l2-norms}, and the fact that $k(kh)^{2p}\leq 1$, we have,  
\beqs
\frac{ \N{\bR^{n+m}}}{ \N{\bR^0} } \leq  \frac{ \N{\bR_{\MD_k}^{n+m}}}{ \N{\bR_{\MD_k}^0} } \lesssim   (hk)^{-1}\frac{ \N{\bR_{\MD_k}^{n+m}}_{\MD_k}}{ \N{\bR_{\MD_k}^0}_{\MD_k} }\lesssim k^{\frac{1}{2p}} (1-c)^m (1-c)^n.
\eeqs
Therefore, to ensure 
that the relative residual is bounded by $(1-c)^n$, we require that $ k^{\frac{1}{2p}} (1-c)^m\le 1$. This is ensured by  choosing $m$ to be the smallest integer larger than $\log_{2}(k) (2p \log_{2}(1/(1-c)))^{-1}$, and the result follows with $C:=
2 \log_{2}(1/(1-c))$.
\end{proof}

\begin{remark}[Right preconditioning]\label{rem:rightpc}
Based on the following equality (see \cite{GrSpVa:17, GrSpVa:17a, GrSpZo:18} for details), the results about right preconditioning (working in the $\MD_k^{-1}$ inner product) can be obtained from analogous results about the left preconditioning of the adjoint problem (working in the $\MD_k$ inner product) 
$$
\frac{ \big| \langle \bV_1, \MA_{\eps} \MB_{\eps}^{-1} \bV_2  \rangle_{\MD_k^{-1}} \big|}{\| \bV_1\|_{\MD^{-1}}\| \bV_2\|_{\MD^{-1}} } =
\frac{ \big| \langle \bW_1,  (\MB_{\eps}^*)^{-1} \MA_{\eps}^*\bW_2  \rangle_{\MD_k} \big|}{\| \bW_1\|_{\MD_k}\| \bW_2\|_{\MD_k} } ,\quad \tfa {\bf 0}\neq \bV_i \in \mathbb{C}^n, \bW_i = \MD_k^{-1}\bV_i, i=1,2.
$$ 
The results in \S \ref{sec:solution_operators}-\ref{sec:dd_results} all hold when the problem in Definition \ref{def:TEDP}  is replaced by its adjoint; therefore the results in this section about left preconditioning
(in the $\MD_k$ inner product) also hold for right preconditioning (in the $\MD_k^{-1}$
 inner product).
\end{remark}

\section{Numerical Experiments}
\label{sec:numerical}

In this section we give numerical experiments to validate our theoretical results and to investigate practical one-level preconditioners for heterogeneous Helmholtz problems.  All computations were done within the Freefem++ software system
\cite{hecht2019freefem++} and were
  performed on a single core (with 8GB memory) on the University of Bath's  Balena HPC system. When using finite elements of degree $p$ we compute stiffness and mass matrices using quadrature rules that are exact for polynomials of degree $2p-2$.

We first state the overall  setting for the experiments, all of which consider the plane wave scattering problem. More precisely, given an incident field $u^i$, we seek the   solution 
$
u = u^i+u^s
$
to the TEDP in Definition \ref{def:TEDP} such that the scattered field $u^s$ satisfies the approximate Sommerfeld radiation condition $\partial_{\bn} u^s - \ri k u^s =0$ on the truncated boundary $\Gamma^I$.  The  truncated domain is  the unit square in 2-d,  $\Omega=(0,1)^2$ and the impenetrable obstacle domain $\Omega_{-} =\emptyset$ (i.e., the problem we consider is the Interior Impedance Problem).
Setting $u^i(\bx) =e^{\ri k \hat{\bd}\cdot \bx}$,  a plane wave with $\hat{\bd}=(1/\sqrt{2},1/\sqrt{2})^T$,  the TEDP satisfied by $u$ has the data: $f=0$ and $g =\ri k(\hat{\bd}\cdot\hat{\bn} -1)e^{\ri k \hat{\bd}\cdot \bx},$ where $\hat{\bn}$ is the outward normal.  Other parameters in the TEDP, such as $A, n, \eps$ and $k$, are stated in each experiment. Note that when $A = I$ and $n = 1$ (the homogeneous interior impedance problem), the exact solution of the TEDP is
  $ u(\bx) = u^i(\bx) = e^{\ri k \hat{\bd}\cdot \bx}$. 

  To discretise the TEDP, we first choose a uniform coarse mesh $\cT^H$ of equal square elements of side length $H = 1/M$ on $\Omega$ (with $M$ being a positive integer), and then uniformly refine the coarse mesh to obtain a fine mesh $\cT^h$. According to Remarks \ref{rem:pollution} and \ref{rem:meshthresh2}, we choose the fine mesh size $h\sim k^{-1-\frac{1}{2p}}$.  When absorption is present,   we always choose $0 \leq \eps \leq k^2$. Assumption \ref{ass:2} is therefore always satisfied (in fact $hk \rightarrow 0$ as $k \rightarrow \infty$).

The definition of  a specific preconditioner requires  an overlapping domain decomposition $\{\Omega_\ell\}_{\ell = 1}^N$ and a corresponding partition of unity $\{\chi_\ell\}_{\ell = 1}^N$.  Based on the coarse mesh introduced above, we consider two options for the partition:
\begin{itemize}
\item \emph{Partition strategy 1:} Let $\{\phi^H_\ell \}_{\ell = 1}^N$ be the basis functions of the bilinear finite element space on $\cT^H$. Choose the overlapping subdomains as $\Omega_\ell = {\rm supp}(\phi^H_\ell)$. Then
  $\{\chi_\ell := \phi^H_\ell,  \ \ell = 1, \ldots, N$\}  forms a natural  partition of unity for $\{\Omega_\ell \}_{\ell = 1}^N$ with generous overlap, and was used  in \cite{GrSpZo:18}. 

\item \emph{Partition strategy 2:} Take the coarse elements $\{\tilde{\Omega}_\ell\}_{\ell = 1}^N$ of $\cT^H$ as a non-overlapping domain decomposition. Extend each domain $\tilde{\Omega}_\ell$ to get $\Omega_\ell$ by adding one or more layers of adjacent  fine grid
  elements, so that the boundary of each extended domain has  distance no more than $\delta$ from  $\partial  \tilde{\Omega}_\ell$. Let $\pi_\ell$ be a non-negative function in $\Omega_\ell$ with $\pi_\ell(\bx) =1$ for $\bx\in \tilde{\Omega}_\ell $ and decreasing linearly to $0$ on the boundary of $\tilde{\Omega}_\ell$. Then set $\chi_\ell(\bx) =\pi_\ell(\bx)/(\sum_i \pi_i(\bx))$  (see for example \cite[Lemma 3.4]{ToWi:05}).
\end{itemize}
In all experiments below we solve the preconditioned system  using  standard GMRES (in the  Euclidean inner product) with relative residual stopping tolerance  $10^{-6}$ and 
  with  initial guess chosen to be a random complex vector with   entries  uniformly distributed on the complex  unit circle. 

In Experiment \ref{exp:p}  we show that, in a wide range of cases, the performance of the preconditioner appears to be independent of the polynomial degree of the elements. In some cases this is predicted by the theory of \S \ref{sec:main_results}, but the property persists even outside the reach of the theory. We also illustrate the benefits of higher order elements in terms of efficiency.  In Experiment \ref{exp:subdomains} we compare the performance of Partitioning Strategies 1 and 2
(including different overlap choices in Strategy 2).
In the rest of the experiments we use
finite element order $p=3$, and  Partition Strategy 2 with overlap $\delta = H/4$.
  In Experiment \ref{exp:fov} we compute and plot the boundary of the field of values of the preconditioned operator for different choices of  absorption $\eps$.
  These show that the analysis leading to Corollaries \ref{cor:matrices} and \ref{cor:GMRES}
  is sharp in a way
  made precise below.  In Experiments  \ref{exp:hete}-\ref{exp:localhete}, we investigate how the heterogeneity affects  the performance of the preconditioner; these experiments support the theory in terms of its dependence on $\eps$ and on the local variation of coefficients, but also illustrate interesting and complex behaviour outside the range of the theory. 

 In all tables of iteration numbers below we give iteration counts  for the SORAS preconditioner \eqref{eq:ASpc}, which is the  preconditioner analysed above. In  brackets in some of the tables we also give iteration numbers  for the ORAS preconditioner \eqref{eq:ASpc1}, since this is popular in practice.  
   However there is no theory  for ORAS applied to Helmholtz problems.
 
\begin{experiment}[Effect of polynomial degree]\label{exp:p}
\ \\
  \noindent
  Linear system setting: $$A=I, n=1,\quad \text{with} \quad  p\in \{1,2,3,4\} $$
Preconditioner setting: $$\text{Partition strategy 1 with  } H=k^{-0.3}$$
\end{experiment}

Tables \ref{tb:polyEps}-\ref{tb:poly} give the numbers of GMRES iterations for solving the homogeneous Helmholtz problem with absorption $\eps =k^{1.5}$ and without absorption $\eps = 0$, respectively.  The preconditioner is the SORAS method \eqref{eq:ASpc} with  Partition Strategy 1, $\delta = H$ and so Assumptions \ref{ass:2}, \ref{ass:3}  are  satisfied. Also in the case when $\eps = k^{1.5}$ we have $k/\eps \delta  = k^{-0.1} \rightarrow 0$. Thus by Corollary \ref{cor:GMRES}, we expect the performance of the preconditioner to be independent of $k$ and $p$ 
as $k \rightarrow \infty$. This is what we observe in Table   \ref{tb:polyEps}. We see the same $p-$independence in Table \ref{tb:poly}, with near $k-$independence as well.   
The last number in Column $2$ of  Table \ref{tb:polyEps}-\ref{tb:poly} is missing, since the system with $p=1$ was too large
for  the single core used for the experiments.  
Higher $p$ yields smaller systems since the restriction on the mesh diameter  $h \sim k^{-1-1/2p}$
becomes less stringent.

Another advantage of using higher order methods is the improved  accuracy of the numerical solution. Table \ref{tb:errors} gives  the relative errors of the FEM solution in $L^2$ and $H^1$.  In each  column the error remains close to constant
but the error  is reduced by at least one order of magnitude each time the   degree is increased by $1$
(recall that the exact solution $u$ is known analytically in this case so errors can be computed).  For the rest of our tests we fix $p=3$.

\begin{table}[H]
\setlength\extrarowheight{2pt} 
\centering
\begin{tabularx}{0.8\textwidth}{C|CCCC}
\hline
\hline
$k\backslash p$	&$1$	&$2$	&$3$	&$4$	\\
\hline
40	&12 	&12 	&12 	&12 	\\
80	&12 	&12 	&12 	&12 	\\
120	&12 	&12 	&12 	&12 		\\
160	&--~	&12 	&12 	&12 	\\
\hline
	\hline
\end{tabularx}
\caption{Experiment \ref{exp:p}: \#GMRES iterations versus  degree $p$, $\eps =k^{1.5}$, SORAS preconditioner, Partition Strategy 1, $H = k^{-0.3}$. }\label{tb:polyEps}
\end{table}

\begin{table}[H]
\setlength\extrarowheight{2pt} 
\centering
\begin{tabularx}{0.8\textwidth}{C|CCCC}
\hline
\hline
$k\backslash p$	&$1$	&$2$	&$3$	&$4$	\\
\hline
40	&13 	&14 	&13 	&13 	\\
80	&12 	&13 	&12 	&12 	\\
120	&13 	&14 	&14 	&13 		\\
160	&--~	&17 	&16 	&15 	\\
\hline
	\hline
\end{tabularx}
\caption{Experiment \ref{exp:p}: \#GMRES iterations versus degree $p$, $\eps =0$, SORAS, Partition Strategy 1, $H = k^{-0.3}$. }\label{tb:poly}
\end{table}

\begin{table}[H]
\setlength\extrarowheight{2pt} 
\centering
\scalebox{0.9}{
\begin{tabularx}{\textwidth}{C|CC|CC|CC|CC}
\hline
\hline
\multirow{2}{*}{$k \backslash p$}& \multicolumn{2}{c}{$1$}& \multicolumn{2}{c}{$2$ }& \multicolumn{2}{c}{$3$}& \multicolumn{2}{c}{$4$} \\ \cline{2-9}
& $e_0$  &  $e_1$& $e_0$  &  $e_1$& $e_0$  &  $e_1$& $e_0$  &  $e_1$\\ 
\hline 
$40$	&$5.73$e$-2$	&$9.09$e$-2$	&$1.33$e$-3$	&$1.04$e$-2$	&$6.76$e$-5$	&$6.49$e$-4$	&$2.79$e$-6$	&$6.51$e$-5$\\
$60$	&$5.73$e$-2$	&$8.11$e$-2$	&$1.30$e$-3$	&$8.58$e$-3$	&$5.04$e$-5$	&$5.19$e$-4$	&$2.31$e$-6$	&$5.20$e$-5$\\
$80$	&$5.72$e$-2$	&$7.59$e$-2$	&$1.25$e$-3$	&$7.39$e$-3$	&$4.10$e$-5$	&$4.42$e$-4$	&$2.33$e$-6$	&$4.52$e$-5$\\
$100$	&$5.72$e$-2$	&$7.25$e$-2$	&$1.25$e$-3$	&$6.63$e$-3$	&$3.60$e$-5$	&$3.97$e$-4$	&$2.52$e$-6$	&$4.10$e$-5$\\
\hline
	\hline
\end{tabularx}}
\caption{Experiment \ref{exp:p}: Relative errors $e_0 := \frac{\|u-u_h\|_\Omega}{\|u\|_\Omega}$, $e_1 := \frac{\|\nabla(u-u_h)\|_\Omega}{\|\nabla u\|_\Omega}$
}\label{tb:errors}
\end{table}

\igg{\es{We see from Tables \ref{tb:polyEps} and \ref{tb:poly}}
 that the GMRES iteration counts remain 
 \es{fairly constant}
 for all $k$ tested.
\es{In contrast,} the theoretical results 
require that $k$ should be sufficiently large before the  bounds on the norm and the field of values of the preconditioned matrix can be  guaranteed
\es{(see, e.g., Corollary  \ref{cor:proj})}. In fact under the assumptions of Experiment \ref{exp:p}, we have $C_\chi = \sqrt{2}$,  and $\Lambda = 4$,  so  with $\delta = k^{-0.3}$ and $\eps = k^{1.5}$,   condition
  \eqref{eq:condn1} reads $ k^{-0.2} \leq (640 \sqrt{2})^{-1}, $ requiring a  very large $k$ to be satisfied. However these are sufficient and not necessary conditions for good GMRES convergence, and we see this 
  \es{clearly} in \es{the results of} Experiment \ref{exp:p}. To explore this issue  further, in the next experiment we compute numerical bounds on the norm and field of values \igg{of the preconditioned matrix} \es{(}which were theoretically estimated in Corollary \ref{cor:proj}\es{)} and we see that these \es{bounds} vary  very little
  \es{over the range of $k$ considered}. However we know no way of proving such results for pre-asymptotic $k$. }

\igg{\begin{experiment}[Numerical bounds for the field of values \igg{and norm} ]\label{exp:fov-bounds}
\ \\
 \noindent
  Linear system setting: $$A=I, n=1,\quad \text{with} \quad  p\in \{1,2,3,4\} $$
Preconditioner setting: $$\text{Partition strategy 1 with  } H=k^{-0.3},k^{-0.4}, k^{-0.5}$$
\end{experiment}}
\igg{We begin by recalling that any complex matrix $\MX$ (with the same dimension as the finite element space $\cV^h$) can be written as the sum of its Hermitian and skew-Hermitian parts:  $\MX = \MX_R + \ri \MX_C$ with  $\MX_R$ and $\MX_C$  both Hermitian. Thus, for any vector $\bV$ of nodal values, 
  $$ \langle \bV, \MX\bV\rangle  = \langle \bV,  \MX_R \bV\rangle  + \ri \langle
  \bV,  \MX_C \bV\rangle , $$
  so that, provided $\MX_R$ is positive definite, then the distance of the field of values of $\MX$ from the origin can be estimated below
  by the minimum eigenvalue of $\MX_R$.
We use this \es{fact} to obtain a lower bound on the \igg{distance of the  field of values of the preconditioned matrix}   $\MB_\eps^{-1}\MA_\eps$ \igg{from the origin}, by noting that, for all $\bV \not =\mathbf{0}$,  
$$\begin{aligned}
  \frac{
 \left\langle    \bV, \MB_\eps^{-1}\MA_\eps\bV \right\rangle_{\MD_k} }
 {
\N{\bV}^2_{\MD_k}
} &= \
\frac{
  \left\langle    \bW, \MD_k^{1/2}\MB_\eps^{-1}\MA_\eps \MD_k^{-1/2}\bW \right\rangle}
 {
\N{\bW}^2   
},  \quad \text{where} \quad \bW = \MD_k^{1/2}\bV,
\end{aligned}
$$
and \es{then} estimating the modulus of the right-hand side from below by the minimum eigenvalue of the Hermitian part of  $\MD_k^{1/2}\MB_\eps^{-1}\MA_\eps \MD_k^{-1/2}$ (when this is positive).

Analogously, the quantity $\N{\MB_{\eps}^{-1} \MA_{\eps}}_{\MD_k}$ is computed as the largest eigenvalue of
  $$ \MD_k^{-1/2} ((\MB_\eps^{-1})\MA_\eps)^* \MD_k   (\MB_\eps^{-1}\MA_\eps)
  \MD_k^{-1/2}.$$
We computed  these eigenvalues using  the package SLEPc within the finite element package FreeFEM++.

Tables \ref{tb:fov-b1}-\ref{tb:fov-b3} list the computed bounds for
different subdomain sizes $H=k^{-0.3}, k^{-0.4},$ and $H=k^{-0.5}$.  In each table,  the first number in \es{the} brackets is a lower bound of the distance of the field of values of the preconditoned matrix from  the origin, and the second number in \es{the} brackets is the $\MD_k$ norm of the preconditioned matrix.

We observe that the field of values is bounded  away from the origin for all choices of $H$ and for all values of $k$ tested,
even though these are outside the theoretical range of
Corollary \ref{cor:proj}.
Nevertheless we see qualitative agreement with Corollary \ref{cor:proj} in the sense that the distance of the field of values from the origin
\es{gets}  smaller as  $\sigma\sim \frac{k}{\epsilon H}$ increases.
\es{In} contrast, the norm of the preconditioned matrix changes very little (and is very close to $1$) as $H$ and $k$ vary. \es{Furthermore}, all the bounds appear to be  independent of the polynomial degree $p\in \{1,2,3,4\}$.}

\begin{table}[H]
  \igg{
    \setlength\extrarowheight{2pt} 
\centering
\begin{tabularx}{\textwidth}{C|CCCC}
\hline
\hline
$k\backslash p$	&$1$	&$2$	&$3$	&$4$	\\
\hline
40	& (0.175, 1.030) 	& (0.176, 1.029) 	& (0.176, 1.028) 	& (0.176, 1.040)	\\
80	& (0.203, 1.020)	& (0.203, 1.020)	& (0.203, 1.020)	& (0.203, 1.021)	\\
120	& 	(0.193, 1.022)	& (0.193, 1.022)	& 	(0.193, 1.021)	&	(0.193, 1.022)	\\
160	&	(--, --) & (0.203, 1.019)		& (0.203, 1.019)		&  (0.203, 1.019)	\\
\hline
	\hline
\end{tabularx}
\caption{\igg{Experiment \ref{exp:fov-bounds}: bounds for the distance of the  field of values from the origin (first  \es{number in brackets}) and the norm (second  \es{number}) of the preconditioned matrix, $\eps =k^{1.5}$, SORAS preconditioner, Partition Strategy 1, $\delta \sim H = k^{-0.3}$.}
}\label{tb:fov-b1}
}
\end{table}

\begin{table}[H]
\igg{\setlength\extrarowheight{2pt} 
\centering
\begin{tabularx}{\textwidth}{C|CCCC}
\hline
\hline
$k\backslash p$	&$1$	&$2$	&$3$	&$4$	\\
\hline
40	& (0.147, 1.043) 	& (0.148, 1.041) 	& (0.148, 1.041) 	& (0.148, 1.054)	\\
80	& (0.154, 1.037)	& (0.154, 1.037)	& (0.154, 1.036)	& (0.154, 1.038)	\\
120	& 	(0.155, 1.035)	& (0.155, 1.035)	& 	(0.155, 1.035)	&	(0.155, 1.035)	\\
160	&	(--, --) & (0.150, 1.037)		& (0.150, 1.037)		&  (0.150, 1.037)	\\
\hline
	\hline
\end{tabularx}
\caption{\igg{Experiment \ref{exp:fov-bounds}: bounds for the distance of the  field of values from the origin 
(first  \es{number in brackets}) and the norm (second  \es{number})
 of the preconditioned matrix , $\eps =k^{1.5}$, SORAS preconditioner, Partition Strategy 1, $\delta \sim H = k^{-0.4}$.}
}\label{tb:fov-b2}
}\end{table}

\begin{table}[H]
\igg{\setlength\extrarowheight{2pt} 
\centering
\begin{tabularx}{\textwidth}{C|CCCC}
\hline
\hline
$k\backslash p$	&$1$	&$2$	&$3$	&$4$	\\
\hline
40	& (0.100, 1.071) 	& (0.101, 1.070) 	& (0.101, 1.069) 	& (0.101, 1.082)	\\
80	& (0.104, 1.060)	& (0.104, 1.060)	& (0.104, 1.059)	& (0.104, 1.061)	\\
120	& 	(0.100, 1.046)	& (0.100, 1.046)	& 	(0.100, 1.046)	&	(0.100, 1.048)	\\
160	&	(--, --) & (0.093, 1.064)		& (0.093, 1.064)		&  (0.093, 1.064)	\\
\hline
	\hline
\end{tabularx}
\caption{\igg{Experiment \ref{exp:fov-bounds}: bounds for the distance of the  field of values from the origin 
(first  \es{number in brackets}) and the norm (second  \es{number})
of the preconditioned matrix , $\eps =k^{1.5}$, SORAS preconditioner, Partition Strategy 1, $\delta \sim H = k^{-0.5}$.
    \label{tb:fov-b3}}}
}
\end{table}

\begin{experiment}[Effect of sizes of subdomains and  overlap]
\label{exp:subdomains}

\

\noi Linear system setting: $$A=I, \quad n=1,\quad  p=3$$
Preconditioner setting: $$\text{different partition strategies with different  } H, \delta$$
\end{experiment}
Experiment \ref{exp:p} used   fairly large  subdomains and generous overlap and 
 would represent a relatively  heavy communication load in parallel implementation.
Here we look for more practical alternatives, first comparing Partition Strategies  1 and 2.

In Table \ref{tb:subdomain1} we give results using Partition Strategy 1 with  $H=k^{-\alpha}$, for various $\alpha = 0.4, 0.5, 0.6$.
When  $\eps=k^{1.5}$ the   GMRES convergence is still  independent of $k$,  as $k$ increases,   for all $\alpha$
(and this is guaranteed by Corollaries \ref{cor:matrices} and \ref{cor:GMRES}  when $\alpha = 0.4$. The choice $\alpha = 0.4$ also gives  $k-$independent iterations  when $\eps = 0$, 
but the iterations grow significantly for     $\alpha>0.5$.

In Table \ref{tb:subdomain2}, we compare this with   Partition Strategy 2
with  $H=k^{-0.4}$, and  various choices of  overlap $\delta$.
Note that smaller overlap means a reduced  communication load in parallel implementations.
By comparing the columns   in Table \ref{tb:subdomain1} (for $\alpha=0.4$)  with Table \ref{tb:subdomain2} (for
$\delta=H/4$), 
we see  that reducing the overlap from $H$ to $H/4$ does not degrade the GMRES convergence.
However smaller overlap choices do not give good preconditioners for the case $\eps = 0$.
In Tables  \ref{tb:subdomain1},  \ref{tb:subdomain2} we have also included  results  for the ORAS preconditioner \eqref{eq:ASpc1}. Although this is often better than the SORAS preconditioner, we have no theory for it, and we see that in the next experiment that a theory based on simply estimating the field of values is bound to fail.

\begin{table}[H]
\setlength\extrarowheight{2pt} 
\centering
\begin{tabularx}{0.85\textwidth}{C|CCC|CCC}
\hline
\hline
& \multicolumn{3}{c}{$\eps =k^{1.5}$}& \multicolumn{3}{c}{$\eps =0$}
 \\ \cline{2-7}
$k$ & $\alpha =0.4$  &  $\alpha =0.5$ &  $\alpha =0.6$& $\alpha =0.4$&  $\alpha =0.5$ &  $\alpha =0.6$\\ 
\hline 
$40$		&12 (6)	&14 (9)	&20 (15)	&16 (10)	&24 (18)	&40 (28)\\
$80$		&12 (7)	&15 (10)	&20 (16)	&20 (14)	&30 (23)	&45 (39)\\
$120$	&12 (6)	&14 (11)	&20 (18)	&22 (16)	&32 (25)	&55 (41)\\
$160$	&12 (6)	&13 (10)	&21 (19)	&21 (15)	&31 (26)	&65 (55)\\
\hline
	\hline
\end{tabularx}
\caption{Experiment \ref{exp:subdomains}:
\#GMRES iterations with the SORAS (ORAS) preconditioner:  subdomain size $H^{-\alpha}$,  Partition strategy 1}\label{tb:subdomain1}
\end{table}

\begin{table}[H]
\setlength\extrarowheight{1pt} 
\centering \scalebox{0.7}{
\begin{tabularx}{1.25\textwidth}{C|CCCC|CCCCC}
\hline
\hline
& \multicolumn{4}{c}{$\eps =k^{1.5}$}& \multicolumn{5}{c}{$\eps =0$}
 \\ \cline{2-10}
$k$ & $\delta =\frac{H}{4}$  &  $\delta =\frac{1}{k}$	&  $\delta =2h$	&  $\delta =h$& $\delta =\frac{H}{4}$&  $\delta =\frac{2}{k}$	&  $\delta =\frac{1}{k}$	&  $\delta =2h$	 &  $\delta =h$\\ 
\hline 
$40$		&13 (7)	&17 (10)	&17 (10)	&23 (12)	&18 (13)	&19 (13)	&29 (17)	&29 (17)	&40 (19)	\\
$80$		&12 (7)	&16 (10)	&19 (11)	&25 (13)	&20 (15)	&22 (17)	&30 (21)	&40 (23)	 &61 (26)	\\
$120$	&12 (7)	&17 (11)	&20 (12)	&26 (14)	&22 (18)	&27 (23)	&38 (27)	&53 (29)	&77 (31)	\\
$160$	&12 (7)	&18 (12)	&21 (13)	&28 (15)	&24 (19)	&35 (29)	&47 (31)	&64 (34)	&97 (35)	\\
\hline
	\hline
\end{tabularx}}
\caption{Experiment \ref{exp:subdomains}: \#GMRES iterations with the SORAS (ORAS) preconditioner: subdomain size  $H=k^{-0.4}$,  overlap $\delta$,  Partition strategy 2}\label{tb:subdomain2}
\end{table}

In the rest of our  tests, we use the preconditioners based on  Partition Strategy 2 with $H=k^{-0.4}$ and $\delta =H/4,$
since this works well in cases both with and without absorption and presents a reasonable compromise between subdomain size and overlap.
In fact in practice, we set $H=1/M \sim k^{-0.4}$ where  $M$ is an integer. For $k=40,80,120,160$, $M$ is equal to $4,5,6,7$ respectively. 
Thus in the following experiments  we are   using  subdomains that shrink in size  as $k$ increases.

\begin{figure}[t]
    \centering
    \begin{subfigure}[t]{0.5\textwidth}
        \centering
        \includegraphics[width=.9\textwidth]{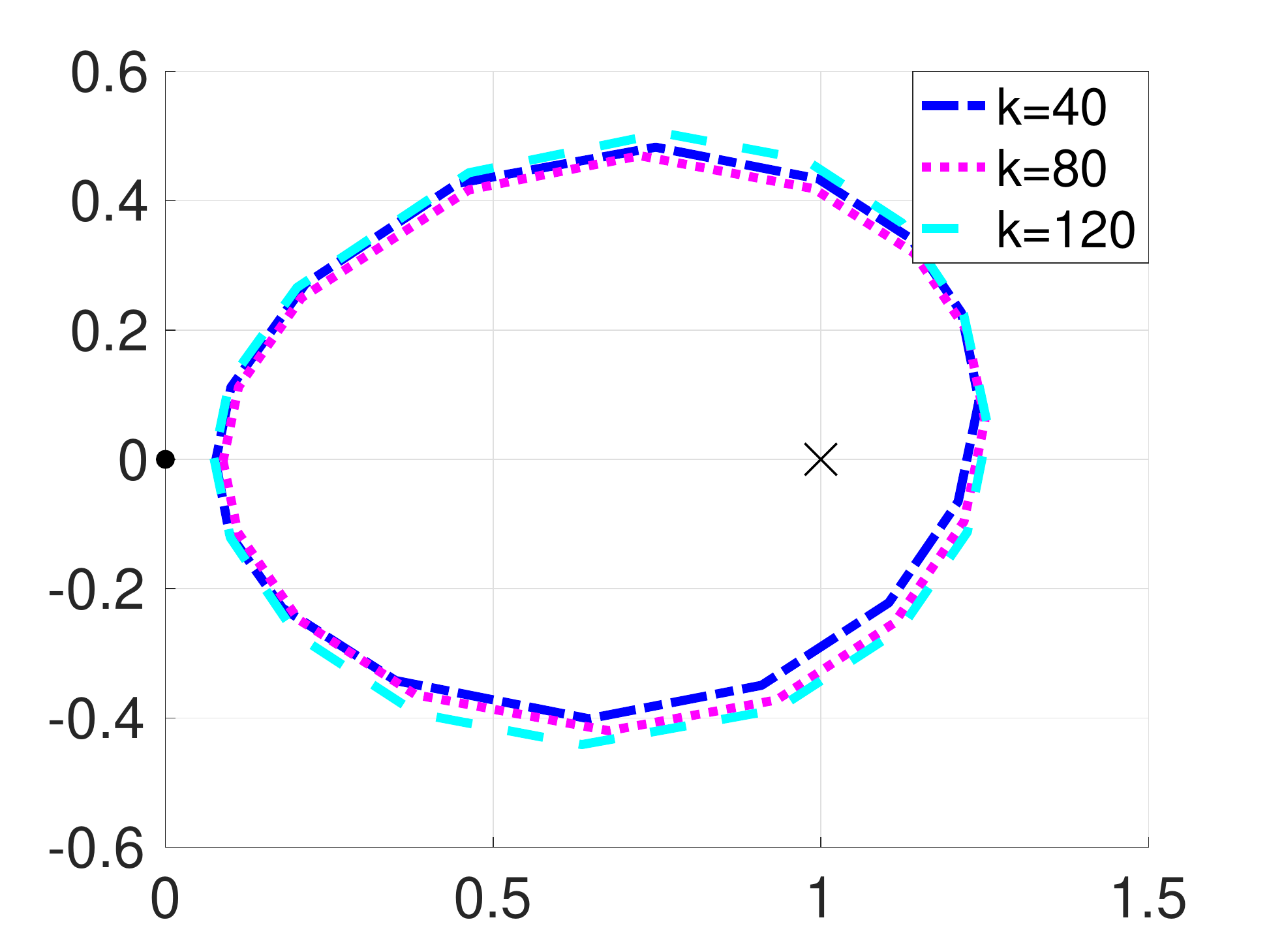}
        \caption{SORAS, $\eps=k^{1.5}$}\label{fig:fov-eps1.5}
         \end{subfigure}%
     \hfill
    \begin{subfigure}[t]{0.5\textwidth}
        \centering
        \includegraphics[width=.9\textwidth]{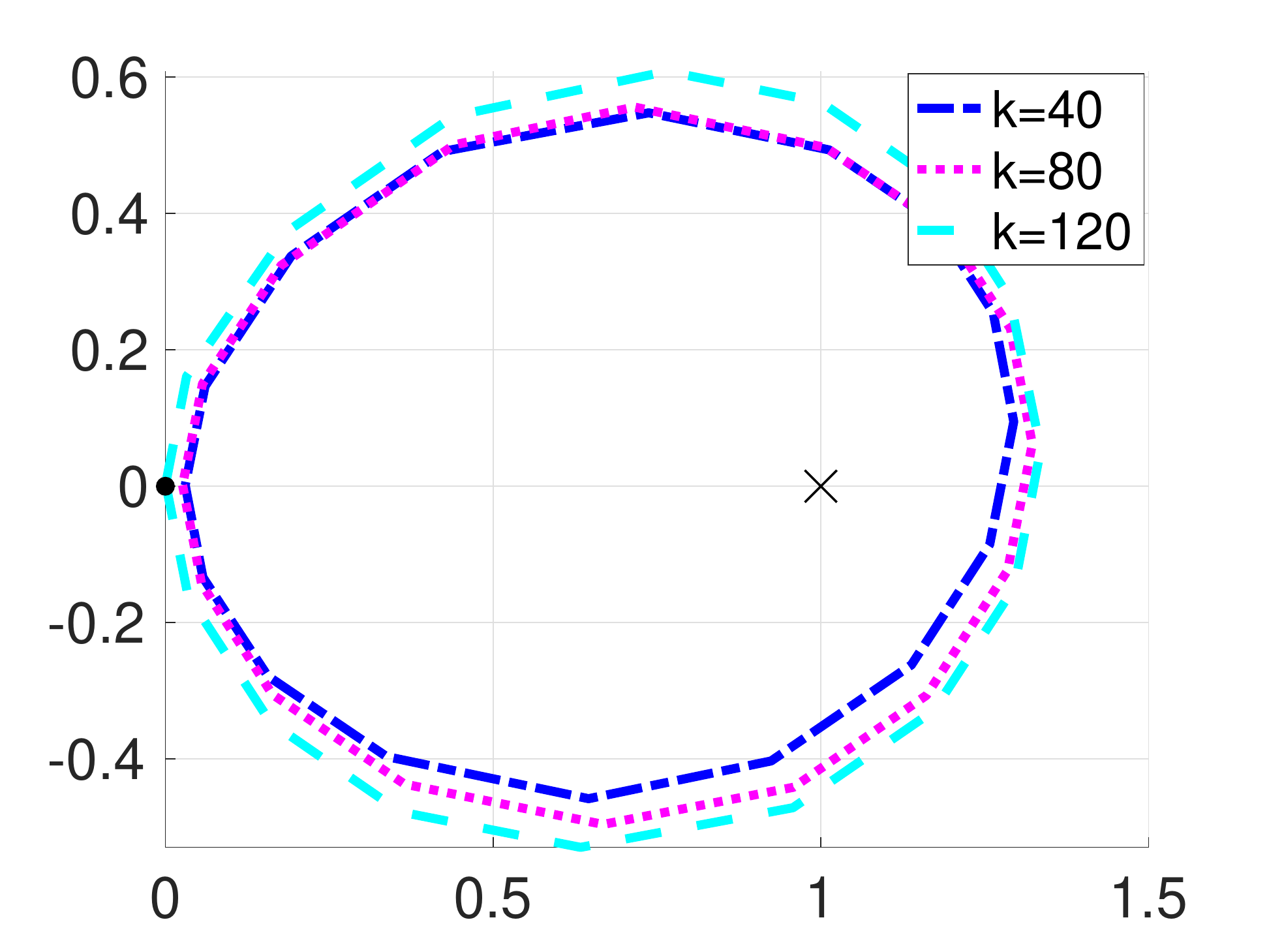}
        \caption{ SORAS, $\eps=k^{1.4}$}\label{fig:fov-eps1.4}
        \end{subfigure}%
        
           \begin{subfigure}[t]{0.5\textwidth}
        \centering
        \includegraphics[width=.9\textwidth]{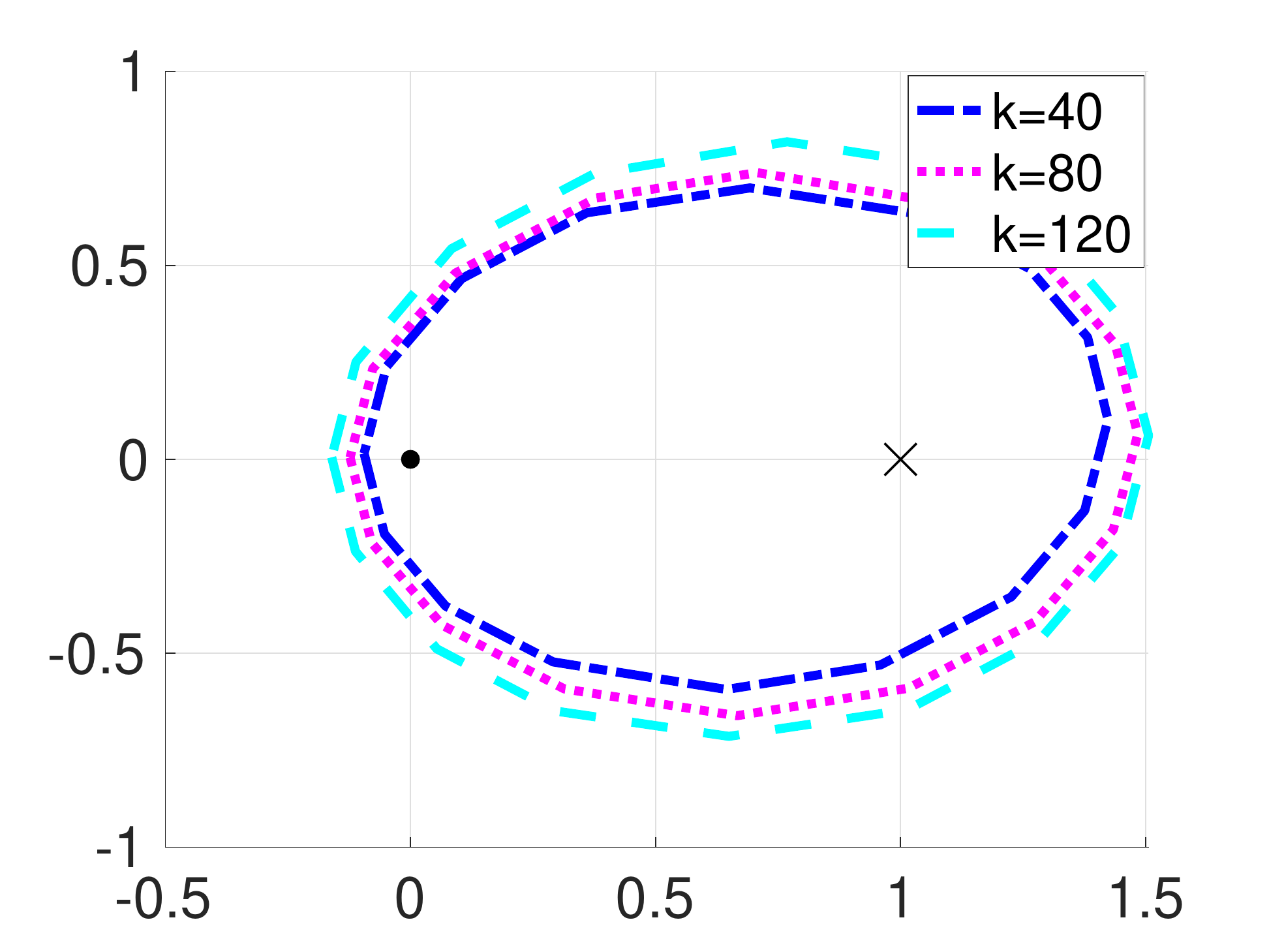}
        \caption{SORAS,  $\eps=k$}\label{fig:fov-eps1}
         \end{subfigure}%
     \hfill
    \begin{subfigure}[t]{0.5\textwidth}
        \centering
        \includegraphics[width=.9\textwidth]{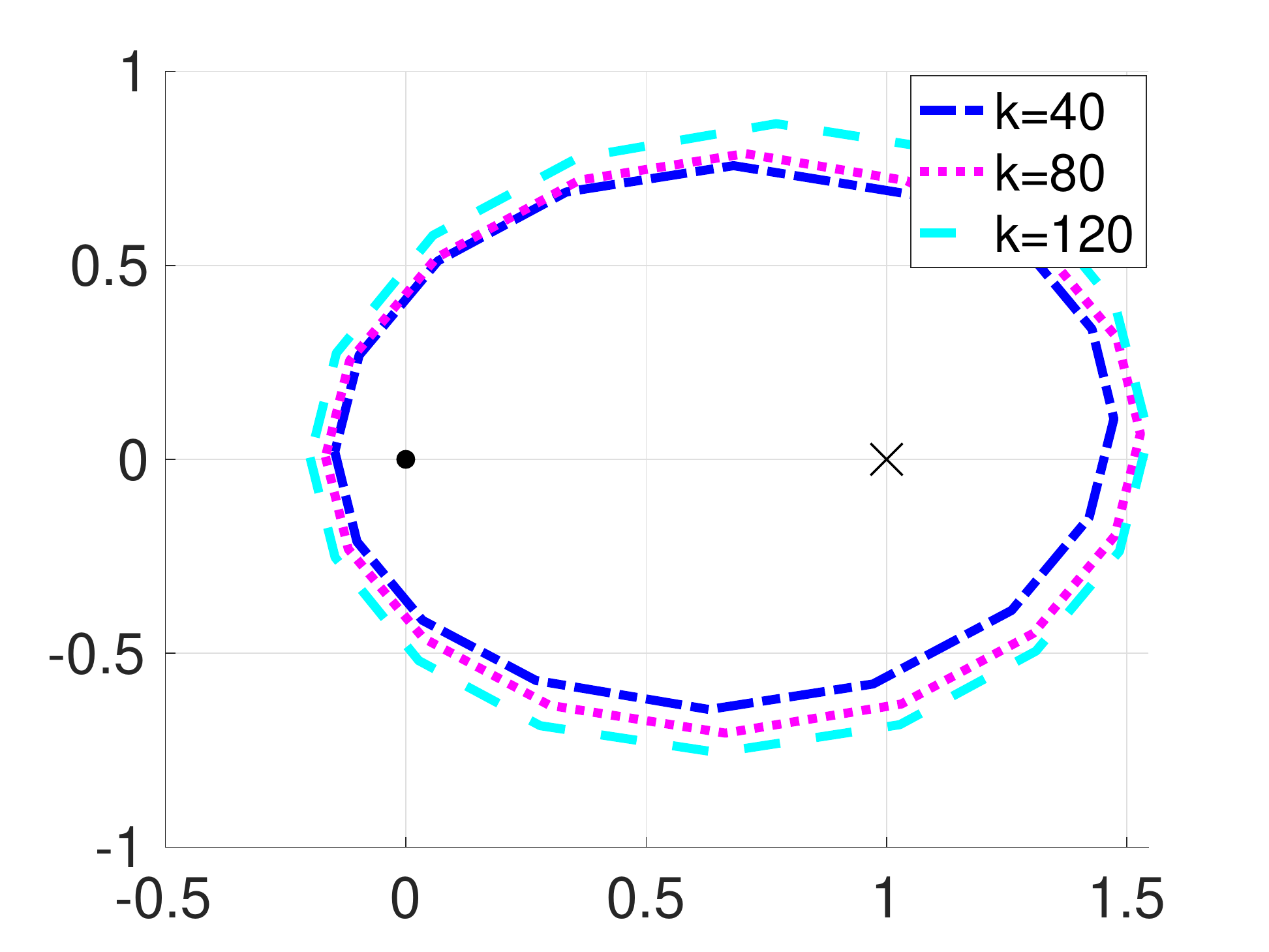}
        \caption{SORAS,  $\eps=0$}\label{fig:fov-eps0}
	\end{subfigure}%
           \caption{Experiment \ref{exp:fov}: field of values  for various  $\eps$ (SORAS). The origin is denoted with a bold dot and the point $1$ with an $\times$. }\label{fig:fov}
\end{figure}

\begin{experiment}[Plots of the field of values]
\label{exp:fov}

\noi Linear system setting: $$A=I,\  n=1, \text{ and }  p=3$$
Preconditioner setting: $$\text{Partition strategy 2 with } H=k^{-0.4}, \delta=\frac{H}{4}$$
\end{experiment}

Corollary \ref{cor:GMRES} used the Elman estimate to
prove $k$-independent bounds on the number of GMRES iterations. This estimate requires
an upper bound on the norm of the preconditioned system and a positive lower bound on the distance of its field
of values 
  from the origin and gave us a rigorous  result about the $k$-independence of GMRES iterations for absorptive problems.
  In this experiment we illustrate two things:
  (i)  that our estimates for the field of values are sharp in terms of their dependence on absorption,  and
  (ii) that (given an upper bound on the norm of the preconditioned operator) the positive  separation of the field of values from the origin is sufficient but appears  far from necessary for good convergence of GMRES.   Point (ii) is reinforced again by later experiments.

Throughout we have used  the algorithm of Cowen and Harel \cite{CoHa:95} to 
 plot the boundary of the field of values (in the inner product of  $\langle\cdot, \cdot\rangle_{\MD_k}$) for any given matrix.  Recall that the field of values is a convex subset of $\bbC$.   In all plots of the field of values, the origin in $\bbC$ is denoted with a bold dot, while the point at $1$ is denoted with a cross.  

Figure \ref{fig:fov} shows the boundaries of the field of values for the SORAS preconditioner in Experiment \ref{exp:fov} 
 with different choices of $\eps$.  Here $\delta = H/4 \sim k^{-0.4}$.
 In Figure \ref{fig:fov-eps1.5} $k/(\vert\eps\vert  \delta) = k^{-0.1} \rightarrow 0$ as $k \rightarrow \infty$
 and the field of values is well away from the origin, thus   confirming  our estimate from Corollary \ref{cor:matrices}.
 Figure \ref{fig:fov-eps1.4} shows the case   $k/(\vert\eps\vert  \delta) = \mathcal{O}(1)$;
here the requirement \eqref{eq:condn1} of Corollary \ref{cor:matrices}   just fails,  
and we see in   Figure \ref{fig:fov-eps1.4}  that the boundary of the  field of values moves towards  the origin as $k$ increases.
 For the
 cases shown in Figures \ref{fig:fov-eps1}-\ref{fig:fov-eps0},  $k/(\vert\eps\vert  \delta)$ blows up as $k$ increases
 and here we see that  the field of values contains the origin. This experiment verifies the sharpness of the field of values  estimates in \S \ref{sec:main_results} in terms of their  dependence on $\eps$. 
 However  the numerical results in Table \ref{tb:subdomain2} 
 show that the preconditioners work well even in some cases where the corresponding field of values of the preconditioned problem contains the origin.  For the field of values given in Figure  \ref{fig:fov-eps0},
 the preconditioner arguably still works well with increasing $k$ (6th column of Table \ref{tb:subdomain2}). Indeed we expect that Columns 3-6 of Table \ref{tb:subdomain2} all correspond to fields of values that contain the origin.    
 Thus  GMRES continues to work well even in some cases where our   sufficient  conditions for $k$-independent iterations
 are violated.

 We further emphasise this point by  plotting  in  Figure \ref{fig:fov-oras} the field of values of the ORAS preconditioned matrices for $\eps = k^{1.5}$ and $\eps = 0$. GMRES iteration numbers for these are given (in brackets)  in Columns 2 and 6 of Table \ref{tb:subdomain2}.
 In Column 2 we see convincingly   $k$-independent convergence of GMRES and  in Column 6 we see  very good convergence; however the field of values contains the origin in all cases.

\begin{figure}[t]
    \centering
           \begin{subfigure}[t]{0.5\textwidth}
        \centering
        \includegraphics[width=.9\textwidth]{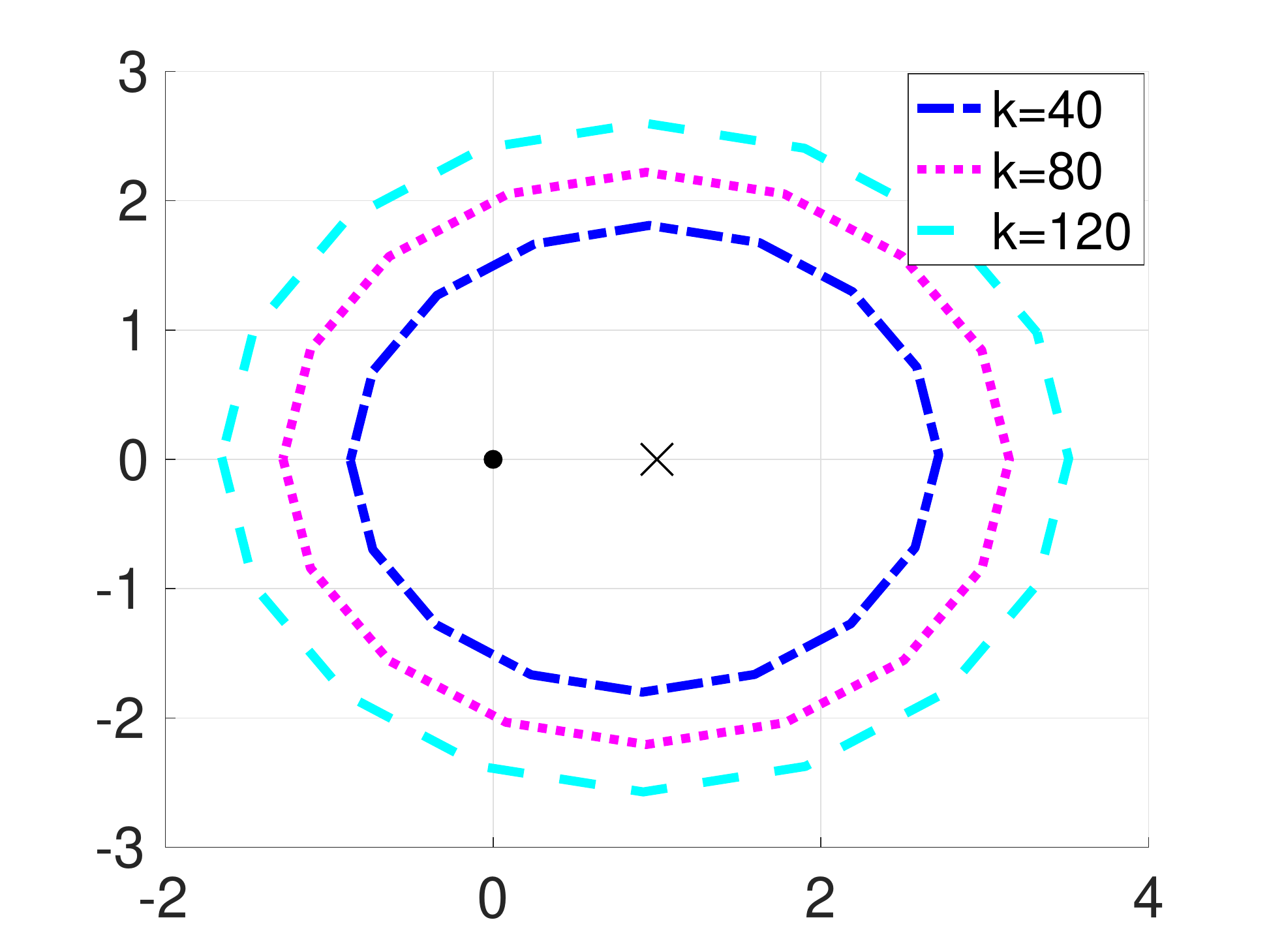}
        \caption{ORAS, $\eps=k^{1.5}$}\label{fig:fov-eps1.5oras}
         \end{subfigure}%
     \hfill
    \begin{subfigure}[t]{0.5\textwidth}
        \centering
        \includegraphics[width=.9\textwidth]{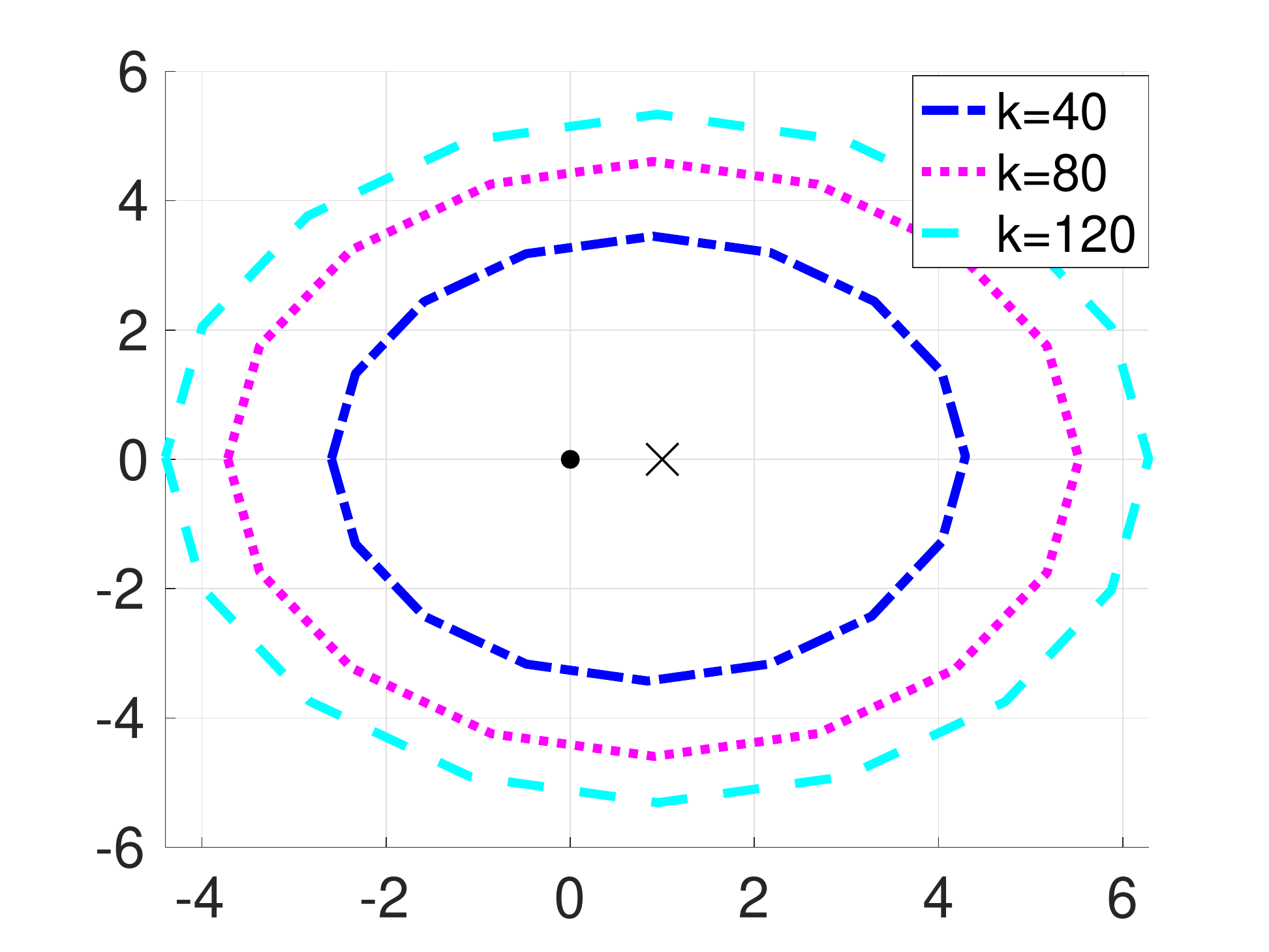}
        \caption{ORAS,  $\eps=0$}\label{fig:fov-eps0oras}
    \end{subfigure}%
   \caption{Experiment \ref{exp:fov}: field of values for various  $\eps$ (ORAS). The origin is denoted with a bold dot and the point $1$ with an $\times$.}\label{fig:fov-oras}
\end{figure}

 \begin{experiment}[Effect of the heterogeneity of the media.]
  \label{exp:hete}

\

\noi Linear system setting: $$\text{variable }A, \,\,n, \text{ and }  p=3$$
Preconditioner setting: $$\text{Partition strategy 2 with } H=k^{-0.4}, \delta=\frac{H}{4}$$
\end{experiment}

In the domain $\Omega = (0,1)^2$ we introduce a  penetrable obstacle that is either 
a disk of radius $1/4$ or a square of side length $1/2$, centred at $(1/2,1/2)$.
In this experiment the penetrable obstacle corresponds to  either  $A$ or $n$ being variable.
Although our theory allows $A$ to be a matrix, in these experiments it is scalar.  
Exterior to the penetrable obstacle the coefficients
are $A = 1$  and $n = 1$. The  profiles studied  are shown in Figure \ref{fig:ws}.
In these,  grey denotes a coefficient value equal to  $1$,  while blue denotes a  value $<1$ and red denotes a value $>1$ (with actual values to be given below).
In addition:     
\begin{enumerate}
\item Figures \ref{fig:sctws1} and \ref{fig:sctws4}: coefficient linearly decreases from the center to the boundary of the obstacle;
\item Figures \ref{fig:sctws2} and \ref{fig:sctws5}: coefficient linearly increases from the center to the boundary of the obstacle;
\item Figures \ref{fig:sctws3} and \ref{fig:sctws6}: coefficient oscillates with the maximum and minimum values inside obstacle.
\end{enumerate}
For the oscillating profiles in Figures \ref{fig:sctws3} and \ref{fig:sctws6}, there are 7 layers of uniform thickness as we proceed outward from the center to the boundary of the obstacle with the maximum attained in the red layer and the minimum  in the blue layer. Thus by  stating
the maximum and minimum values of the coefficient,  all profiles are uniquely defined.

\begin{figure}[t]
    \centering
    \begin{subfigure}[t]{0.3\textwidth}
        \centering
        \includegraphics[width=\textwidth]{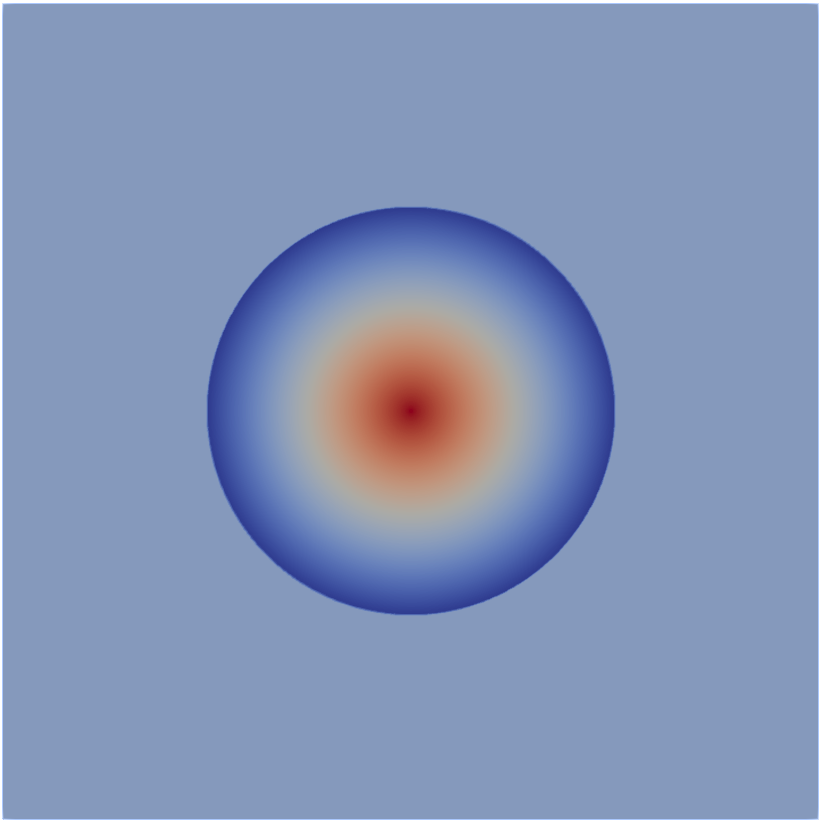}
        \caption{\tiny maxima at center}\label{fig:sctws1}
    \end{subfigure}%
    \hfill
    \begin{subfigure}[t]{0.3\textwidth}
        \centering
        \includegraphics[width=\textwidth]{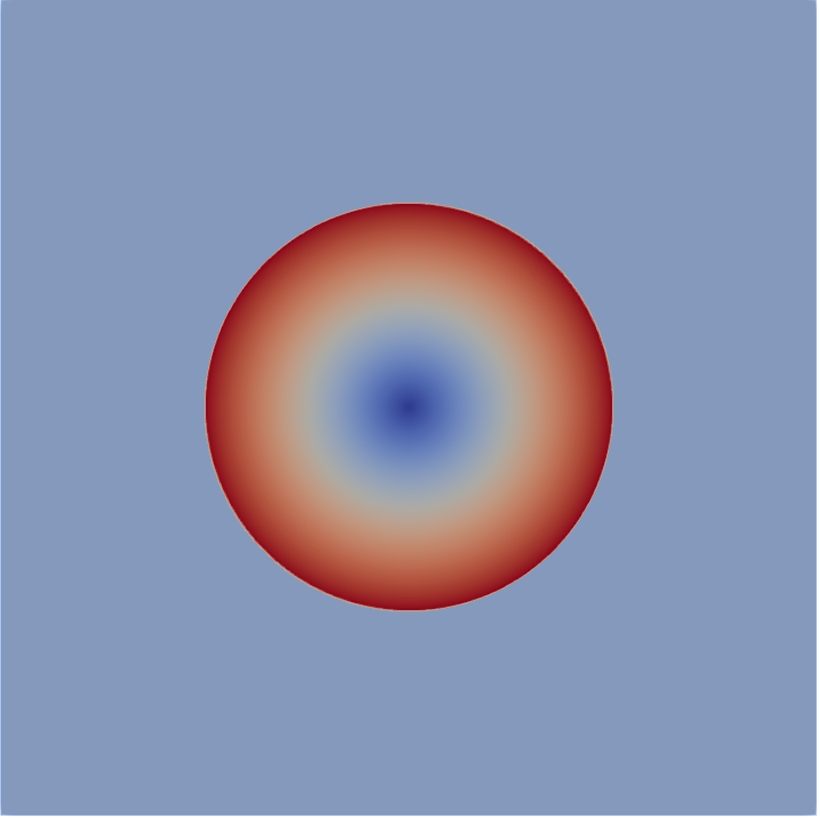}
        \caption{\tiny minima at center}\label{fig:sctws2}
    \end{subfigure}
        \hfill
    \begin{subfigure}[t]{0.3\textwidth}
        \centering
        \includegraphics[width=\textwidth]{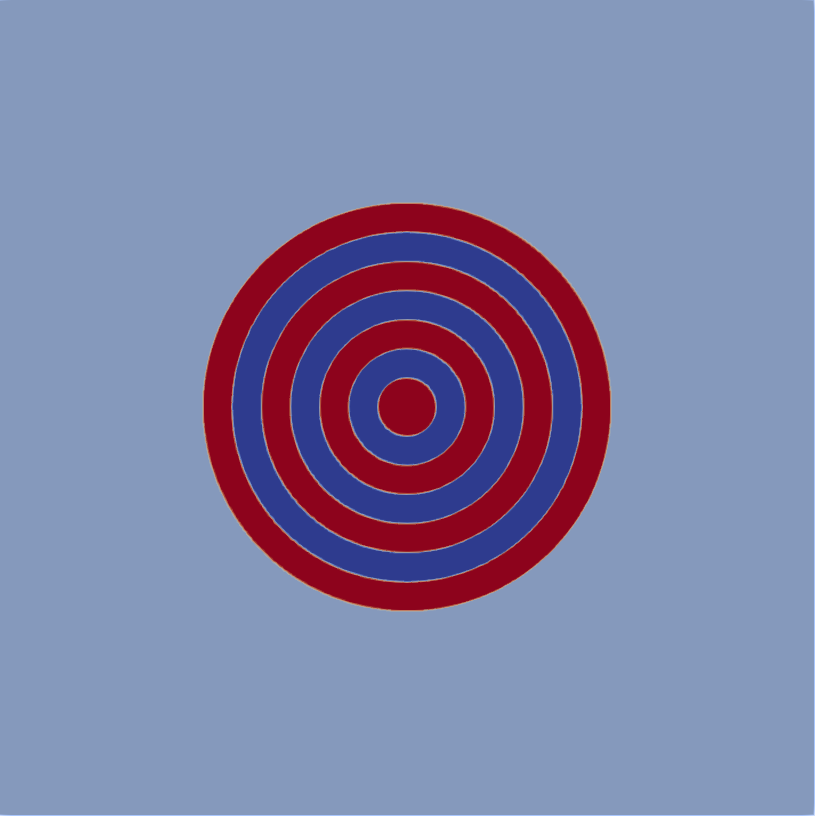}
        \caption{ oscillates}\label{fig:sctws3}
         \end{subfigure}%
     
            \begin{subfigure}[t]{0.3\textwidth}
        \centering
        \includegraphics[width=\textwidth]{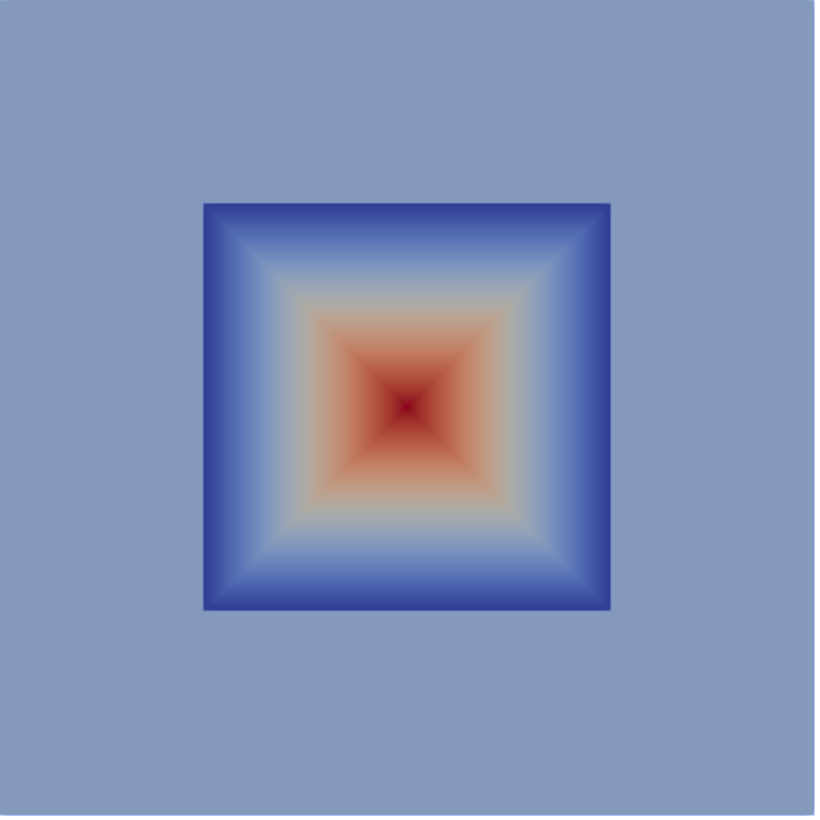}
        \caption{\tiny maxima at center}\label{fig:sctws4}
    \end{subfigure}%
    \hfill
    \begin{subfigure}[t]{0.3\textwidth}
        \centering
        \includegraphics[width=\textwidth]{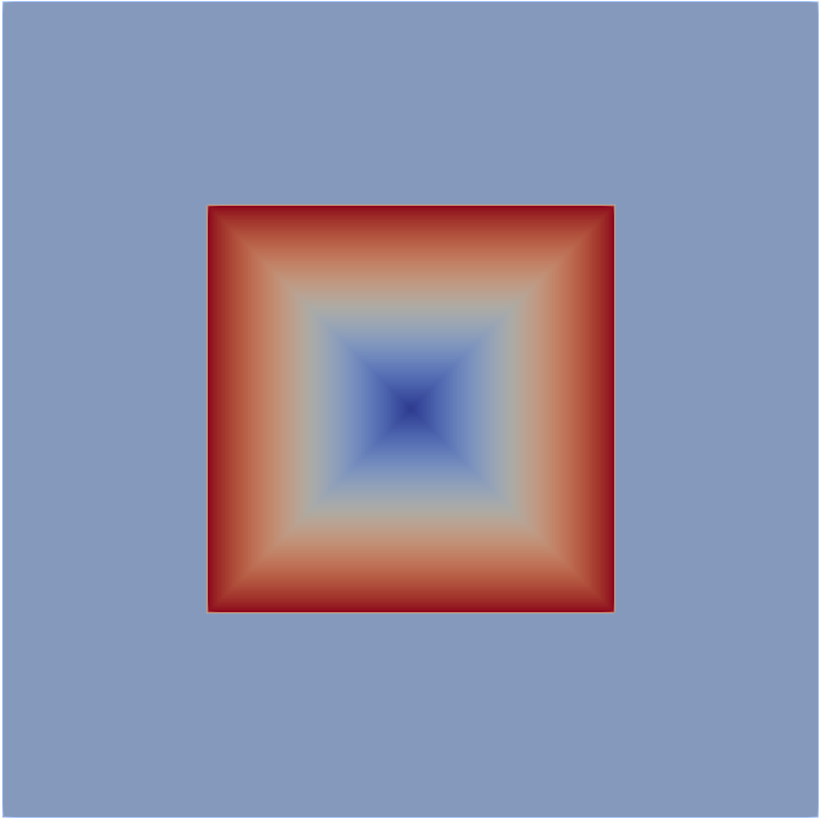}
        \caption{\tiny minima at center}\label{fig:sctws5}
    \end{subfigure}
        \hfill
    \begin{subfigure}[t]{0.3\textwidth}
        \centering
        \includegraphics[width=\textwidth]{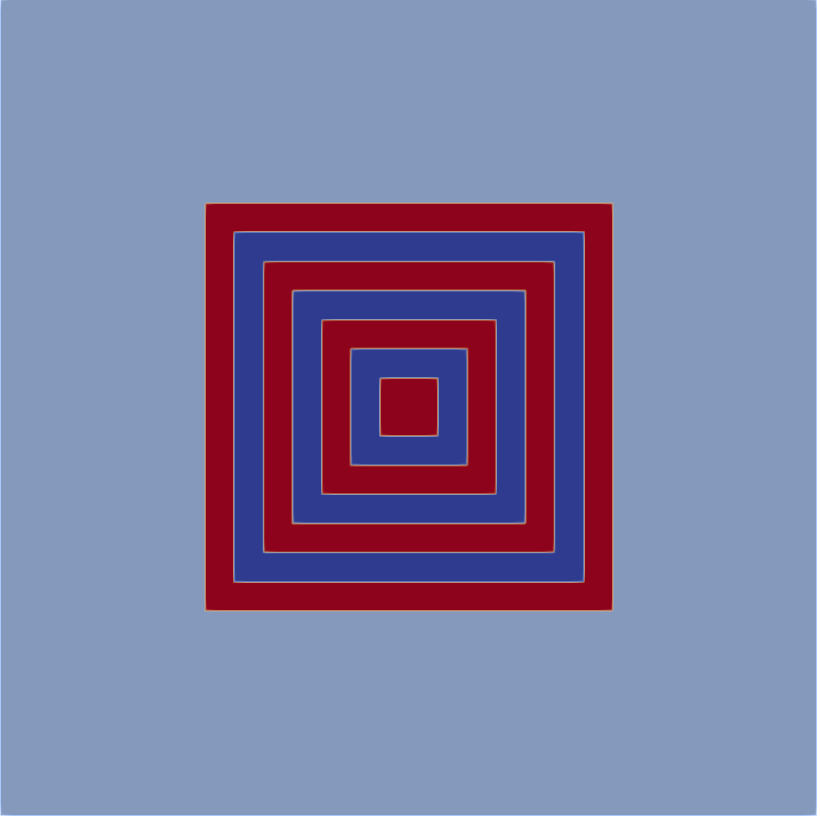}
        \caption{ oscillates}\label{fig:sctws6}
    \end{subfigure}
   \vspace{-.2cm}
    \caption{Different profiles for variable coefficients}\label{fig:ws}
\end{figure}

\noi
\begin{remarknonumber}[Are the coefficients $A$ and $n$ arising from the profiles in Figure \ref{fig:ws} trapping or nontrapping?]
We now describe to what extent it is known whether $A$ and $n$ arising from the profiles in Figure \ref{fig:ws} are trapping or nontrapping. We highlight, however, (recalling the discussion and references in Remark \ref{rem:trapping}) that even if $A$ or $n$ are trapping, we only expect to see the ``bad behaviour" at certain frequencies, and indeed we do not see the extreme ill-conditioning associated with trapping for any of the frequencies used  in the examples below.
The summary is that 
\bit
\item[(i)] Profile (c) is provably trapping for both $A$ and $n$. 
\item[(ii)] Profile (b) is provably trapping for $n$, and Profile (a) is provably trapping for $A$.
\item[(iii)] We expect Profile (a) to be trapping for $n$ and Profile (b) to be trapping for $A$ (but neither is rigorously proved).
\item[(iv)] We expect Profiles (d), (e), and (f) to be nontrapping for $n$ and $A$.
\eit

For understanding these  points, recall that red indicates a coefficient value $>1$, blue indicates a coefficient value $<1$, and the coefficient away from the obstacle equals $1$.

Regarding (i): when the coefficients jump on a smooth convex interface, the problem is trapping if $n$ jumps \emph{down} (moving outwards radially from the centre) or $A$ jumps \emph{up}; see \cite{PoVo:99}, \cite[Section 6]{MoSp:17}. 

Regarding (ii):   When $n$ is a radial function that jumps \emph{down} (moving outwards from the centre) on a circular interface from a linear function to a constant, the problem is trapping by \cite[\S6.2]{Mo:19}, \cite{BaDaMo:19}, similarly when $A$ is a radial function that jumps \emph{up}. 

Regarding (iii): when $n$ is a continuous radial function that \emph{decreases} linearly and $n_{\max}$ is sufficiently large, then $n$ is trapping by \cite{Ra:71}, \cite[Theorem 7.7]{GrPeSp:19}. We therefore expect the linear decrease in Profile (a) to mean that this profile is trapping for $n$. We cannot immediately conclude this from \cite{Ra:71}, \cite[Theorem 7.7]{GrPeSp:19}, since Profile (a) is discontinuous; however, because of the localisation of trapped waves we expect the subsequent nontrapping jump of $n$ not to affect the trapping caused by the linear decrease).
 Similarly, when $A$ is a continuous radial function that \emph{increases} linearly and $A_{\max}$ is sufficiently large we expect Profile (b) to be trapping for $A$.
 
Regarding (iv): it is not yet rigorously known whether Profiles (d), (e), and (f) are trapping or nontrapping for $A$ or $n$. Indeed, all the examples of trapping described above rely on a trapped wave \emph{either} circling an interface that is either radial, or smooth and convex \cite{PoVo:99}, \cite[\S6.2]{Mo:19}, \cite{BaDaMo:19}, \emph{or} supported by a change in a radial coefficient \cite{Ra:71}, \cite[Theorem 7.7]{GrPeSp:19}. Certainly for Profile (f) we expect $A$ and $n$ to be nontrapping, since any wave moving parallel to one of the sides of the square interfaces loses energy when it hits the next side perpendicularly; thus long-lived waves moving parallel to square interfaces cannot exist.
\end{remarknonumber}

In Tables \ref{tb:heteABSvaryn} and  \ref{tb:heteABSvaryA} we give  the performance of GMRES  for the six profiles above in the case   $\eps =k^{1.5}$. In Table \ref{tb:heteABSvaryn}, we fix $A=1$, $n_{\min} = 0.02$ and   $n_{\max} =50$,  while  in Table \ref{tb:heteABSvaryA}, we fix $n=1$,   $A_{\min} = 0.02$ and $A_{\max} =50$. 
We see that while the iteration counts  are affected by variation in  $A$, there seems no effect from the variation in $n$. In fact the performance in Table \ref{tb:heteABSvaryn} is similar to the homogeneous case $A = n = 1$
(Table \ref{tb:polyEps}).  
  The worst cases in Table \ref{tb:heteABSvaryA} are for oscillating $A$ (columns Fig \ref{fig:sctws3} and Fig \ref{fig:sctws6}),
  while $A$ increasing outwards is worse than $A$ decreasing outwards. Neverthess it does appear that iteration numbers in Table \ref{tb:heteABSvaryA} are not increasing significantly with $k$.

  This is a case where the field of values plots do give some indication of  convergence rates.    In Figure   \ref{fig:fov-heteABS}
  we plot the field of values for the two cases of Tables \ref{tb:heteABSvaryn} and  \ref{tb:heteABSvaryA} with $k = 120$ ($n$ varying on the left and $A$ varying on the right). In this case $k/\eps \delta \sim k^{-0.1} \rightarrow 0$, so  Corollary \ref{cor:matrices}  ensures that, for $k$ large enough (relative to $p$ and $\cClocal(A,n)$),  the field of values  does not include the origin.
When  $n$ is  varying and not $A$, $\cClocal(A,n)$ 
  just depends on $n_{\mathrm{min}}$ whereas when  $A$ is varying and not $n$ it depends on both $A_{\mathrm{min}}$  and $A_{\mathrm{max}}$. 
  However we also  tested the case of $n_{\max}=10^4$, $n_{\min}=10^{-4}$ and $A=1$, in which case the iteration counts are 
  similar to those  in Table \ref{tb:heteABSvaryn}.
  
\begin{table}[H]
\setlength\extrarowheight{2pt} 
\centering
\begin{tabularx}{.9\textwidth}{C|CCCCCC}
\hline
\hline
$k\backslash n$	&Fig \ref{fig:sctws1}	&Fig \ref{fig:sctws2}	&Fig \ref{fig:sctws3}	&Fig \ref{fig:sctws4}	&Fig \ref{fig:sctws5}	&Fig \ref{fig:sctws6}\\
\hline
40	&13 	&13 	&13 	&13 	&13 	&13 \\
80	&12 	&12 	&12 	&12 	&12 	&12 	\\
120	&12 	&12 	&12 	&12 	&12 	&12 	\\
160	&12 	&12 	&12 	&12 	&12 	&12 	\\
\hline
	\hline
\end{tabularx}
\caption{Experiment \ref{exp:hete}: SORAS, $\eps=k^{1.5}, A = 1, n_{\min}=0.02, n_{\max} = 50$}\label{tb:heteABSvaryn}
\end{table}

\begin{table}[H]
\setlength\extrarowheight{2pt} 
\centering
\begin{tabularx}{.9\textwidth}{C|CCCCCC}
\hline
\hline
$k\backslash A$	&Fig \ref{fig:sctws1}	&Fig \ref{fig:sctws2}	&Fig \ref{fig:sctws3}	&Fig \ref{fig:sctws4}	&Fig \ref{fig:sctws5}	&Fig \ref{fig:sctws6}\\
\hline
40	&20 	&27 	&47 	&20 	&29 	&41 \\
80	&17 	&30 	&51 	&18 	&27 	&46 	\\
120	&21 	&30 	&54 	&21 	&30 	&50 	\\
160	&17 	&26 	&38 	&19 	&28 	&46 	\\
\hline
	\hline
\end{tabularx}
\caption{Experiment \ref{exp:hete}: SORAS, $\eps=k^{1.5}, n = 1, A_{\min}=0.02, A_{\max} = 50$}\label{tb:heteABSvaryA}
\end{table}

\begin{figure}[t]
    \centering
\begin{subfigure}[t]{0.5\textwidth}
        \centering
        \includegraphics[width=.9\textwidth]{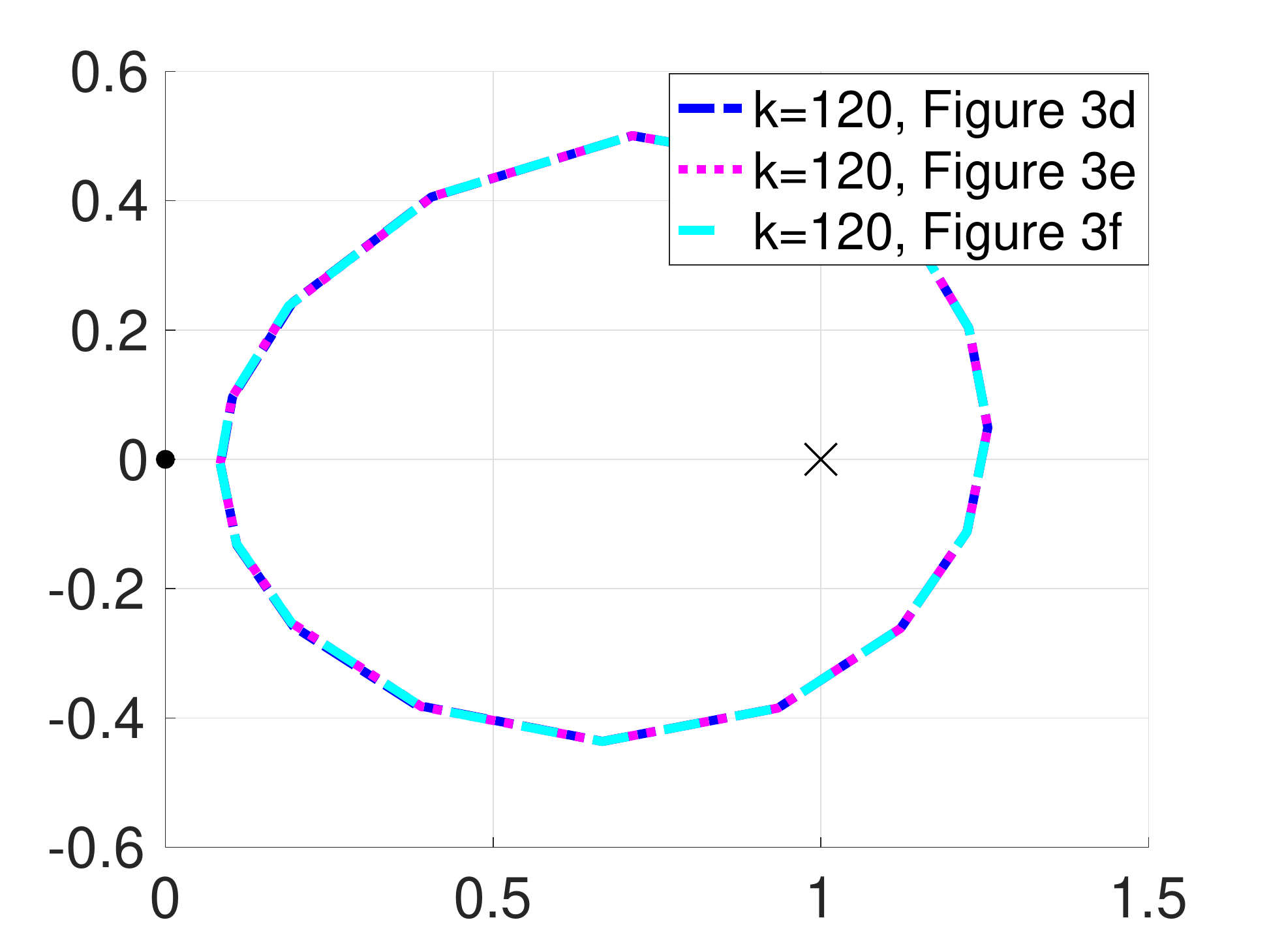}
        \caption{$\eps =k^{1.5}$, $A= 1, n_{\min}=0.02, n_{\max} = 50$}\label{fig:FovHeteEps1_5n50}
    \end{subfigure}%
\hfill     
    \begin{subfigure}[t]{0.5\textwidth}
        \centering
        \includegraphics[width=.9\textwidth]{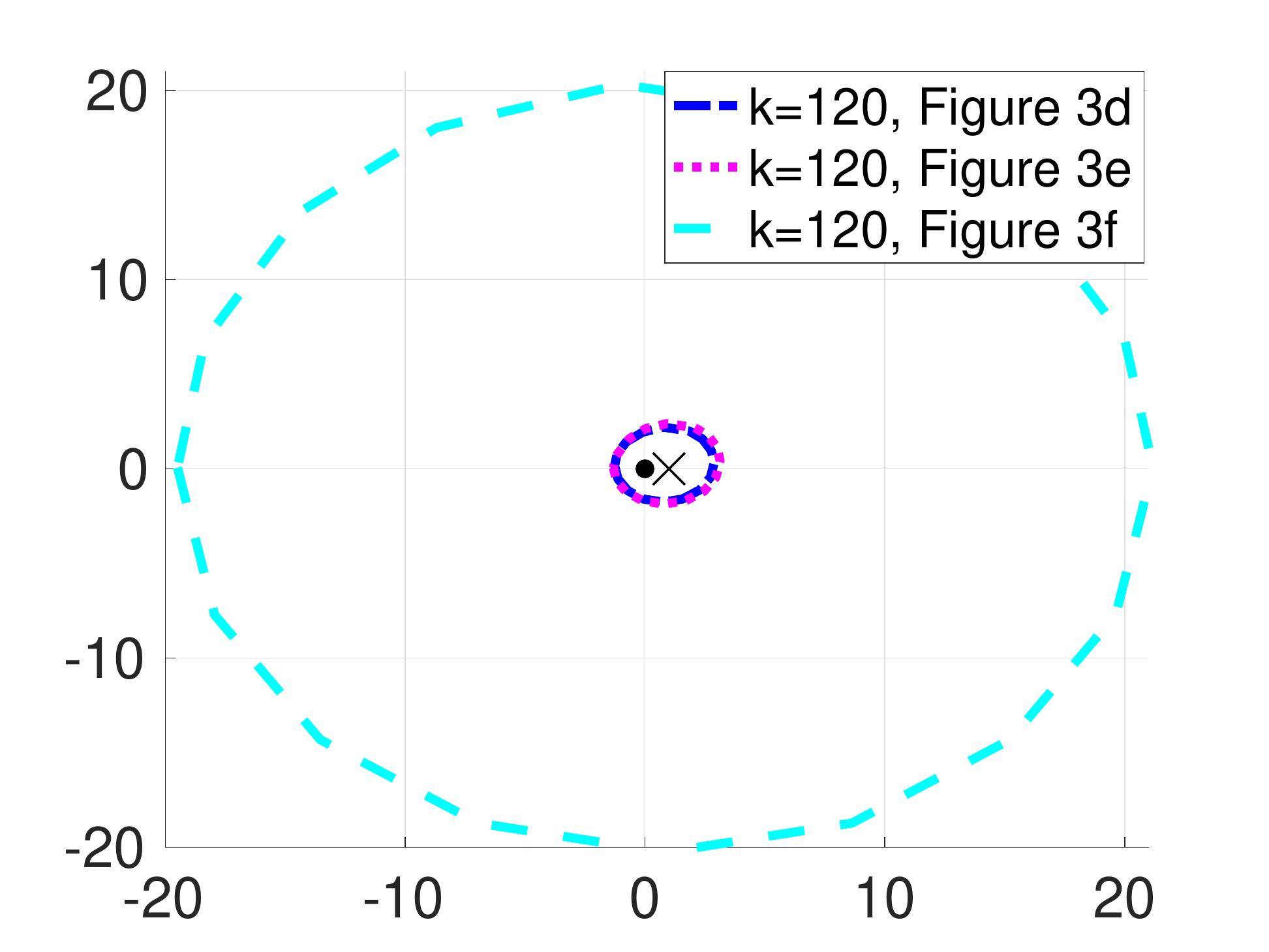}
        \caption{$\eps =k^{1.5}$, $n = 1, A_{\min}=0.02, A_{\max} = 50$}\label{fig:FovHeteEps1_5A50}
         \end{subfigure}%
    \caption{Experiment \ref{exp:hete}: field of values  of SORAS preconditioned matrix,   heterogeneous case with square obstacle, $\eps =k^{1.5}$}\label{fig:fov-heteABS}
\end{figure}

\noi Without absorption, the performance of GMRES becomes sensitive to variation in  both  $A$ and $n$. For a harder problem, in Tables \ref{tb:heteVaryn2}-\ref{tb:heteVaryA2}, we set the quantity $\max(A_{\max}, n_{\max})/(\min(A_{\min}, n_{\min}))$ to  $4$ and let either $A$ or $n$ vary,
 keeping the other fixed. The performance is worst in the cases of oscillating coefficients, while in the other cases the performance as $k$ increases is still very reasonable.
 
In Tables \ref{tb:heteVaryn4}-\ref{tb:heteVaryA4}, we increase  the quantity $\max(A_{\max}, n_{\max})/(\min(A_{\min}, n_{\min}))$ to $16$. The way the variation of $A$ affects the performance of the preconditioners does not change.
But there is a big increase in GMRES iterations when the range of  $n$ gets bigger.
This is reflected in  our estimates for the pure Helmholtz problem $\eps = 0$, for which the approximation of local problems (given by the estimate \eqref{eq:local-approx2}) deteriorates as $n_{\max}$ increases.
We also plot  the field of values for the cases in Tables \ref{tb:heteVaryn4}-\ref{tb:heteVaryA4}.
The case of oscillating $A$ produces a  much larger fields of values, while varying  $n$ does not change the boundary of the field of values much.

\begin{table}[H]
\setlength\extrarowheight{2pt} 
\centering
\begin{tabularx}{.9\textwidth}{C|CCCCCC}
\hline
\hline
$k\backslash n$	&Fig \ref{fig:sctws1}	&Fig \ref{fig:sctws2}	&Fig \ref{fig:sctws3}	&Fig \ref{fig:sctws4}	&Fig \ref{fig:sctws5}	&Fig \ref{fig:sctws6}\\
\hline
40	&18 	&21 	&24 	&18 	&19 	&28 \\
80	&22 	&26 	&39 	&20 	&21 	&30 	\\
120	&27 	&34 	&50 	&24 	&24 	&26 \\
160	&29 	&37 	&64 	&25 	&25 	&38 \\
\hline
	\hline
\end{tabularx}
\caption{Experiment \ref{exp:hete}:  SORAS $\eps=0, A = 1, n_{\min}=0.5, n_{\max} = 2.0$}\label{tb:heteVaryn2}
\end{table}

\begin{table}[H]
\setlength\extrarowheight{2pt} 
\centering
\begin{tabularx}{.9\textwidth}{C|CCCCCC}
\hline
\hline
$k\backslash A$	&Fig \ref{fig:sctws1}	&Fig \ref{fig:sctws2}	&Fig \ref{fig:sctws3}	&Fig \ref{fig:sctws4}	&Fig \ref{fig:sctws5}	&Fig \ref{fig:sctws6}\\
\hline
40	&18 	&18 	&20 	&18 	&18 	&21 \\
80	&21 	&18 	&38 	&18 	&18 	&28 	\\
120	&31 	&21 	&35 	&21 	&20 	&29 \\
160	&32 	&22 	&47 	&23 	&21 	&33 \\
\hline
	\hline
\end{tabularx}
\caption{ Experiment \ref{exp:hete}: SORAS, $\eps=0, n = 1, A_{\min}=0.5, A_{\max} = 2.0$}\label{tb:heteVaryA2}
\end{table}

\begin{table}[H]
\setlength\extrarowheight{2pt} 
\centering
\begin{tabularx}{.9\textwidth}{C|CCCCCC}
\hline
\hline
$k\backslash n$	&Fig \ref{fig:sctws1}	&Fig \ref{fig:sctws2}	&Fig \ref{fig:sctws3}	&Fig \ref{fig:sctws4}	&Fig \ref{fig:sctws5}	&Fig \ref{fig:sctws6}\\
\hline
40	&22 	&26 	&38 	&21 	&25 	&32\\
80	&35 	&47 	&44 	&31 	&28 	&61\\
120	&54 	&56 	&61 	&41 	&39 	&67 \\
160	&61 	&58 	&55 	&40 	&39 	&59\\
\hline
	\hline
\end{tabularx}
\caption{Experiment \ref{exp:hete}: SORAS  $\eps=0, A = 1, n_{\min}=0.25, n_{\max} = 4.0$}\label{tb:heteVaryn4}
\end{table}

\begin{table}[H]
\setlength\extrarowheight{2pt} 
\centering
\begin{tabularx}{.9\textwidth}{C|CCCCCC}
\hline
\hline
$k\backslash A$	&Fig \ref{fig:sctws1}	&Fig \ref{fig:sctws2}	&Fig \ref{fig:sctws3}	&Fig \ref{fig:sctws4}	&Fig \ref{fig:sctws5}	&Fig \ref{fig:sctws6}\\
\hline
40	&19 	&20 	&23 	&19 	&19 	&27 \\
80	&24 	&21 	&48 	&19 	&21 	&40 \\
120	&29 	&26 	&49 	&25 	&21 	&56 \\
160	&28 	&24 	&57 	&25 	&20 	&57 \\
\hline
	\hline
\end{tabularx}
\caption{ Experiment \ref{exp:hete}: SORAS,  $\eps=0, n = 1, A_{\min}=0.25, A_{\max} = 4.0$}\label{tb:heteVaryA4}
\end{table}

To conclude, the observations in the heterogeneous case are: 
\begin{itemize}
\item In the case of absorption, the performance is mainly affected by the variation of $A$, but only weakly affected by $n$;
\item In the case without absorption,  the performance is sensitive to the variation of both $A$ and $n$.
  Increasing the range of $n$ now seems to increase the iteration count more strongly than increasing the range of $A$.
\end{itemize}
 
\begin{figure}[H]
    \centering
 \begin{subfigure}[t]{0.5\textwidth}
        \centering
        \includegraphics[width=.9\textwidth]{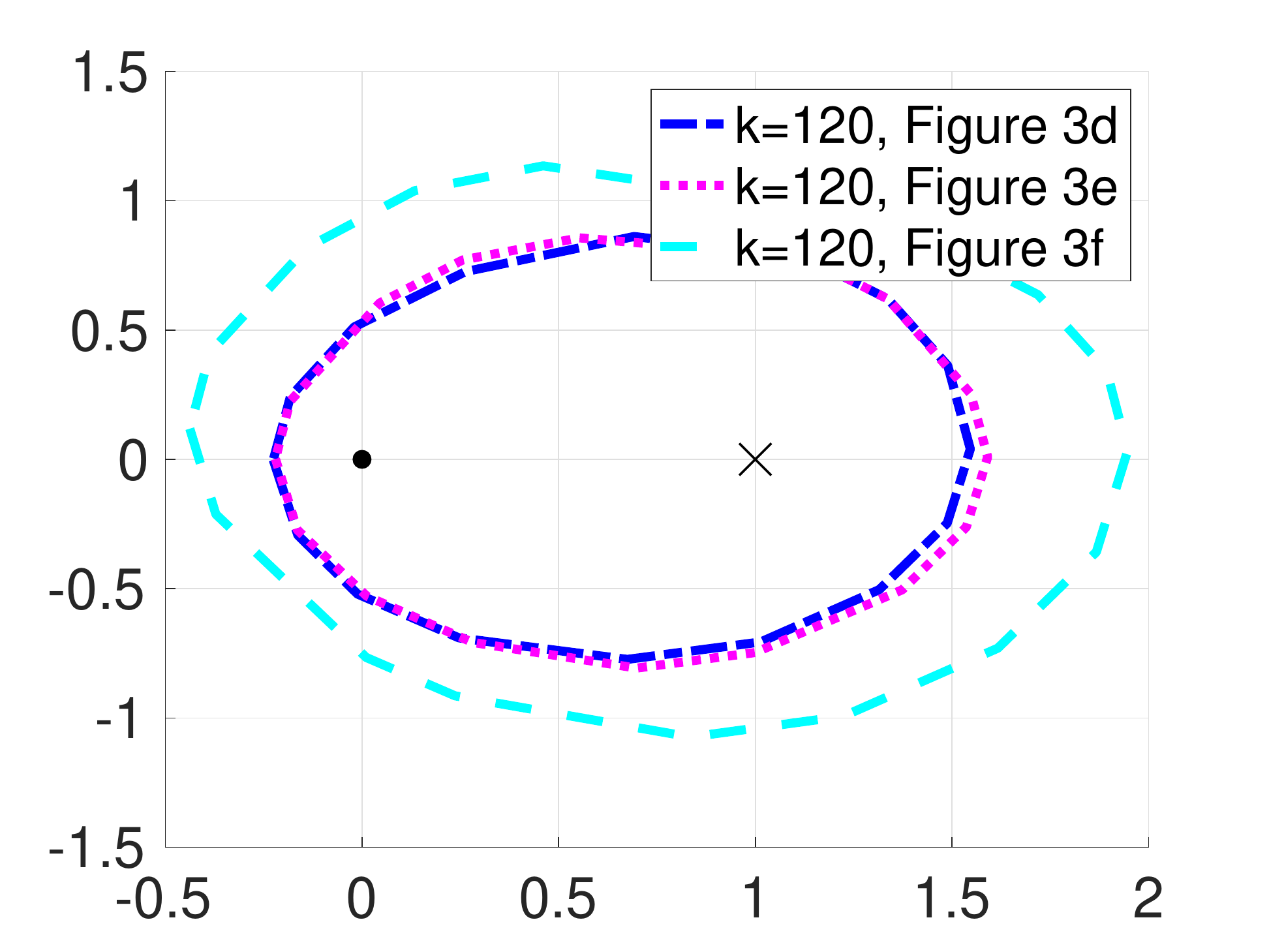}
        \caption{$\eps =0$, $A = 1, n_{\min}=0.25, n_{\max} = 4$}\label{fig:FovHeteEps0n4}
      \end{subfigure}%
      \hfill 
    \begin{subfigure}[t]{0.5\textwidth}
        \centering
        \includegraphics[width=.9\textwidth]{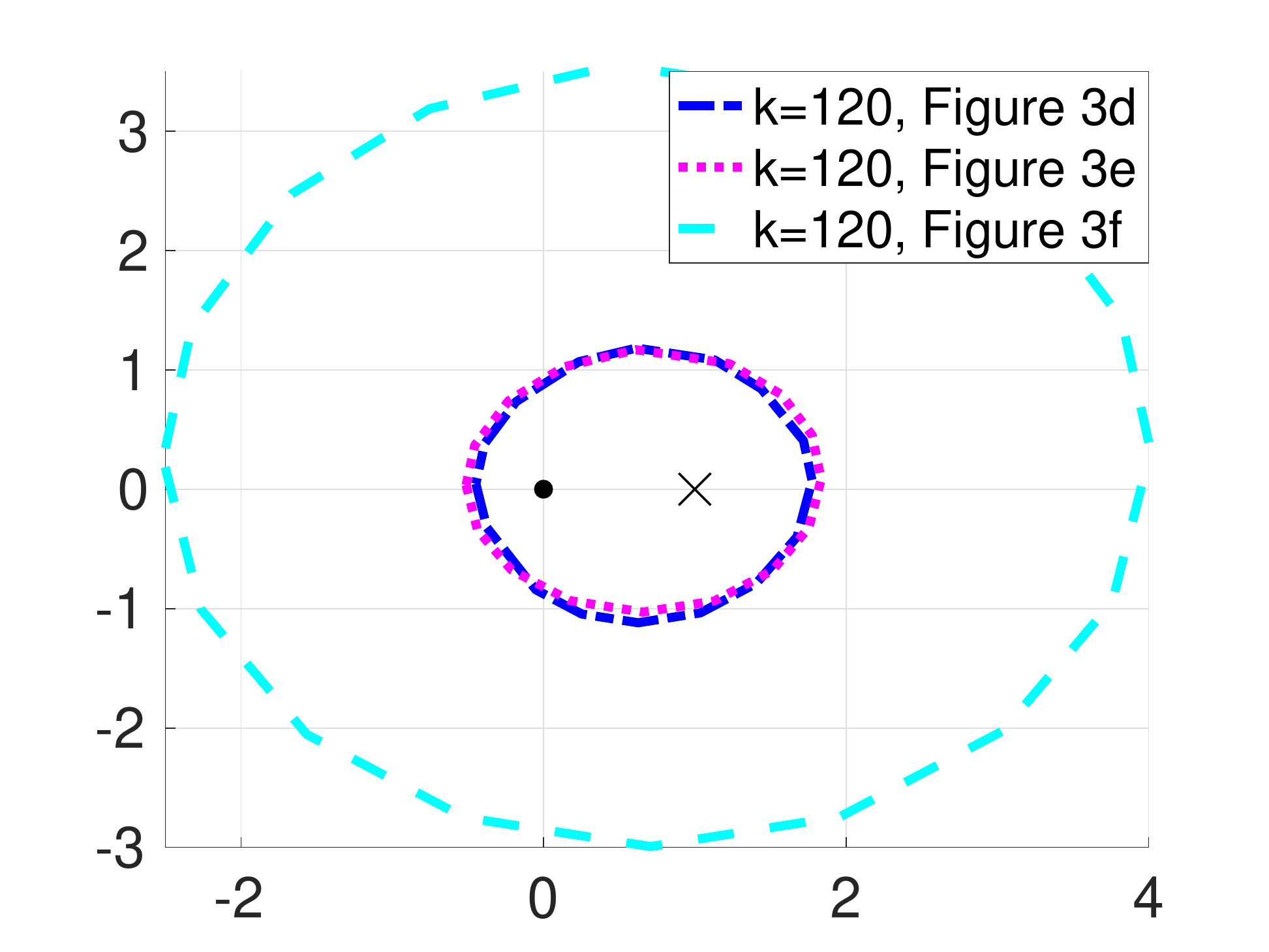}
        \caption{$\eps =0$, $n = 1, A_{\min}=0.25, A_{\max} = 4$}\label{fig:FovHeteEps0A4}
         \end{subfigure}%
   \caption{Experiment \ref{exp:hete}: SORAS,  field of values in the heterogeneous case with square obstacle, $\eps =0$}\label{fig:fov-hete}
\end{figure}

\begin{experiment}[Effect of local variation of the heterogeneity.]
\label{exp:localhete}

\

\noi Linear system setting: $$\text{variable }A, n, \text{ and }  p=3$$
Preconditioner setting: $$\text{Partition strategy 2 with } H=\frac{1}{8}, \delta=\frac{H}{4}$$
\end{experiment}
In Experiment  \ref{exp:hete}, the number of the subdomains  is  $4\times 4$, $5\times 5$, $6\times 6$ and $7\times 7$ for the cases  $k=40,80,120$ and $160$, respectively.  Thus there  always exist subdomains where both global maximum and minimum coefficient values are attained
for any of the   profiles in Figure \ref{fig:ws}. Here we use smaller subdomain sizes  (in fact $8\times 8$)
in order to  illustrate   the  local heterogeneity dependence identified in the theory (Corollary \ref{cor:matrices}). In the case of the  linear variation profiles
Figures \ref{fig:sctws1}, \ref{fig:sctws2}, \ref{fig:sctws4}, \ref{fig:sctws5},
both  the maximum and minimum values of the varying coefficients cannot be attained in a single subdomain.
 For the oscillating profiles \ref{fig:sctws3}, \ref{fig:sctws6}, however, there  always exist  some subdomains  capturing both extreme values. Therefore Profiles  \ref{fig:sctws3}, \ref{fig:sctws6} have worse estimates for the local contrast
$\cClocal(A,n)$ than the others and hence yield worse estimates for the field of values. In Tables \ref{tb:heteVaryA4FixH}-\ref{tb:heteVaryn4FixH} we present the results for these  profiles with $\eps =0$; we see that Profiles \ref{fig:sctws3}, \ref{fig:sctws6} have noticeably worse iteration counts  than the others.

\begin{table}[H]
\setlength\extrarowheight{2pt} 
\centering
\begin{tabularx}{.9\textwidth}{C|CCCCCC}
\hline
\hline
$k\backslash A$	&Fig \ref{fig:sctws1}	&Fig \ref{fig:sctws2}	&Fig \ref{fig:sctws3}	&Fig \ref{fig:sctws4}	&Fig \ref{fig:sctws5}	&Fig \ref{fig:sctws6}\\
\hline
40	&40 (28)	&41 (30)	&52 (33)	&42 (28)	&40 (30)	&64 (43)\\
80	&46 (38)	&39 (29)	&103 (88)	&38 (27)	&36 (28)	&66 	(59)\\
120	&42 (37)	&32 (27)	&59 (61)	&31 (25)	&26 (23)	&68 (73)\\
160	&35 (34)	&31 (25)	&73 (96)	&27 (23)	&24 (27)	&85 (103)\\
\hline
	\hline
\end{tabularx}
\caption{Fixed subdomains, $\eps=0, n = 1, A_{\min}=0.25, A_{\max} = 4.0$}\label{tb:heteVaryA4FixH}
\end{table}

\begin{table}[H]
\setlength\extrarowheight{2pt} 
\centering
\begin{tabularx}{.9\textwidth}{C|CCCCCC}
\hline
\hline
$k\backslash n$	&Fig \ref{fig:sctws1}	&Fig \ref{fig:sctws2}	&Fig \ref{fig:sctws3}	&Fig \ref{fig:sctws4}	&Fig \ref{fig:sctws5}	&Fig \ref{fig:sctws6}\\
\hline
40	&48 (32)	&63 (54)	&73 (72)	&49 (38)	&57 (43)	&75 (56)\\
80	&58 (51)	&75 (70)	&79 (84)	&52 (44)	&50 (40)	&103 	(102)\\
120	&73 (66)	&74 (90)	&90 (132)	&52 (44)	&45 (39)	&102 (124)\\
160	&75 (82)	&69 (87)	&100 (124)	&48 (44)	&41 (36)	&59 (73)\\
\hline
	\hline
\end{tabularx}
\caption{Fixed subdomains, $\eps=0, A = 1, n_{\min}=0.25, n_{\max} = 4.0$}\label{tb:heteVaryn4FixH}
\end{table}

\noi{\bf Acknowledgement:\ } 
We thank Victorita Dolean (University of Strathclyde and Universit\'e C\^{o}te d'Azur) for useful mathematical discussions and also  advice on FreeFEM++.
We gratefully acknowledge support from the UK Engineering and Physical Sciences Research Council  Grants EP/R005591/1 (EAS) and   EP/S003975/1
(SG, IGG, and EAS). 
This research made use of the Balena High Performance Computing (HPC) Service at the University of Bath.
\igg{We thank the anonymous  referees for making several penetrating remarks which have improved the manuscript.} 

\footnotesize{
\bibliographystyle{plain}
\bibliography{biblio_epsilon3}
}

\end{document}